\documentclass[a4paper,twoside]{amsart}
\usepackage{pdfsync}

\usepackage{amsmath,amsthm,amsfonts,latexsym,amscd,amssymb,enumerate,mathrsfs}
\usepackage{graphicx}
\usepackage{comment}
\usepackage{extarrows}
\usepackage[backrefs,?,numeric]{amsrefs}
\usepackage[pdftex,pdftitle={Numeric rho invariants}, backref,pdftex]{hyperref}

\usepackage{tikz}
\usepackage[all]{xy}

\swapnumbers
\theoremstyle{plain}
\newtheorem{theorem}{Theorem}[section]
\newtheorem{lemma}[theorem]{Lemma}

\newtheorem{proposition}[theorem]{Proposition}

\theoremstyle{definition}
\newtheorem{definition}[theorem]{Definition}
\newtheorem{example}[theorem]{Example}

\newtheorem{set-up}[theorem]{Geometric set-up}
\newtheorem{remark}[theorem]{Remark}

\newtheorem{assumption}[theorem]{Assumption}

\DeclareMathOperator{\cyl}{cyl}

\DeclareMathOperator{\End}{End}    

\DeclareMathOperator{\Ind}{Ind}

\newcommand{\forget}[1]{}

\def  \nuint {\raise10pt\hbox{$\nu$}\kern-6pt\int}

\newcommand\Tr{\operatorname{Tr}}

\newcommand\fg{\mathfrak{g}}
\newcommand\fk{\mathfrak{k}}
\newcommand\fp{\mathfrak{p}}

\def \Sp {{\cal S}}

\newcommand\Di{D\kern-6pt/}
\newcommand\cDi{{\mathcal D}\kern-6pt/}
\newcommand\spi{S\kern-6pt/}
\newcommand \cspi{\Sp\kern-6pt/}

\newcommand\CC{\mathbb C}

\def \cal {\mathcal}

\newcommand\RR{\mathbb R}

\newcommand\pa{\partial}
\newcommand\Ker{\operatorname{Ker}}

\usepackage{xcolor}
\definecolor{darkgreen}{cmyk}{1,0,1,.2}
\definecolor{m}{rgb}{1,0.1,1}

{\catcode`@=11\global\let\c@equation=\c@theorem}

\allowdisplaybreaks[2]

\begin{document}
\pagestyle{myheadings}
\markboth{Paolo Piazza, Hessel Posthuma, Yanli Song and Xiang Tang }{Higher orbital integrals, rho numbers and  index theory}

\title{Higher orbital integrals, rho numbers and  index theory}

\author{Paolo Piazza}
\address{Dipartimento di Matematica, Sapienza Universit\`{a} di Roma, I-00185 Roma, Italy}
\email{piazza@mat.uniroma1.it}
\author{Hessel Posthuma}
\address{Korteweg-de Vries Institute for Mathematics, University of Amsterdam, 1098 XG Amsterdam, the Netherlands}
\email{H.B.Posthuma@uva.nl}
\author{Yanli Song}
\address{Department of Mathematics and Statistics, Washington University, St. Louis, MO, 63130, U.S.A.} 
\email{yanlisong@wustl.edu}
\author{Xiang Tang}
\address{Department of Mathematics and Statistics, Washington University, St. Louis, MO, 63130, U.S.A.}
\email{xtang@math.wustl.edu}



\maketitle

\begin{abstract}
Let $G$ be a connected, linear real reductive group.
We give sufficient conditions ensuring the well-definedness of the delocalized eta invariant $\eta_g (D_X)$ associated
to a Dirac operator $D_X$ on a cocompact $G$-proper manifold $X$ and to the orbital integral
$\tau_g$ defined by a semisimple element $g\in G$. Along the way, we give a detailed account of the large time behaviour of the heat kernel and of its short time bahaviour
 near the fixed point set of $g$.
We prove that such a delocalized eta  invariant enters as  the boundary correction term in an  index theorem
computing the pairing between the index class and the 0-degree cyclic cocycle defined by $\tau_g$ on a $G$-proper
manifold with boundary.
More importantly, we also prove a higher version of such a theorem, for the pairing of the index class and  the higher cyclic cocycles defined
by the higher orbital integral $\Phi^P_g$ associated to a cuspidal parabolic subgroup $P<G$ with Langlands decomposition  $P=MAN$ and a semisimple
element $g\in M$.
We employ these results in order to define (higher) rho numbers associated to $G$-invariant positive scalar curvature metrics.
		\end{abstract}

\tableofcontents

\section{Introduction}

This article is a contribution to higher index theory on $G$-proper manifolds, with $G$ a connected, linear real reductive group. 
Before considering the case of $G$-proper manifolds with $G$ a Lie group,
it  is worth spending a few words on the pivotal example of Galois $\Gamma$-coverings. Let $\Gamma$ be 
a finitely generated discrete group and let $\Gamma\to X \to X/\Gamma$
be a Galois covering; let $D$ be a $\Gamma$-equivariant Dirac-type operator on a spin$^c$ manifold $X$, acting on the sections of
a $\Gamma$-equivariant vector bundle $E$. 
We initially assume that 
 $X/\Gamma$ is a smooth compact manifold {\it without} boundary.
 We also assume that $\Gamma$ is Gromov hyperbolic or of polynomial growth with respect to a word metric.
 The seminal work of Connes and Moscovici \cite{ConnesMoscovici} allows to define a pairing between
 $HC^* (\CC\Gamma,\langle e \rangle)$,
 the cyclic cohomology of $\CC\Gamma$
 localized at the unit element, and the K-theory of the Roe $C^*$-algebra $K_* (C^* (X,E)^\Gamma)$. In particular,
 by applying this pairing to the index class associated to $D$, $\Ind (D)\in K_* (C^* (X,E)^\Gamma)$, it is possible to define 
  higher indices associated to $D$, parametrized by the elements in  $HC^* (\CC\Gamma,\langle e \rangle)$, a group which is in fact isomorphic
  to $H^* (\Gamma,\CC)$. 
  The higher index formula of Connes-Moscovici provides a geometric formula for these higher indices, with very interesting geometric
  applications
  to higher signatures and higher $\widehat{A}$-genera.
One can also pair the index class with  $HC^* (\CC\Gamma,\langle x \rangle)$, the cyclic cohomology localized at the
  conjugacy class $\langle x \rangle$, but this pairing turns out to be identically zero. 
    
Despite these interesting geometric applications the higher index invariants we have just introduced will
be inadequate whenever we have a geometric situation in which the index class vanishes.
This is the case, for example, if $D$ is the spin Dirac operator associated to a metric of positive scalar curvature.
Thus, in geometric questions involving, for example, the moduli space of metrics of positive scalar curvature $\mathcal{R}^+ (M)/{\rm Diffeo}(M)$,
one is led to
consider {\it secondary}  invariants. One such invariant is  the delocalized eta invariant of Lott, $\eta_{\langle x \rangle} (D)$,
associated to an  invertible Dirac operator $D$ and,  initially, to a finite conjugacy class $\langle x \rangle$ in $\Gamma$, see \cite{lott-delocalized}.
This invariant was extended to conjugacy classes of polynomial growth in \cite{PiazzaSchick_PJM} and then, much more generally, to
arbitrary conjugacy classes of Gromov hyperbolic groups in the deep work of Puschnigg, see \cite{Puschnigg}.

The geometric interest for such an invariant stems from its connection with the Atiyah-Patodi-Singer index class on a Galois covering
with boundary $\Gamma\to Y \to Y/\Gamma$, once we assume the boundary operator associated to $D$ to be invertible. Thanks to the higher  Atiyah-Patodi-Singer
index formula of Leichtnam and Piazza and Wahl, see  \cite{LeichtnamPiazzaMemoires,LeichtnamPiazzaBSMF,Wahl1}, one can prove that the 
pairing of the Atiyah-Patodi-Singer  index class with the cyclic cocycle  
 $[\tau_{\langle x\rangle}]\in HC^0 (\CC\Gamma,\langle x\rangle)$,
\begin{equation}\label{0-degre}
\tau_{\langle x\rangle} (\sum_\gamma \alpha_\gamma \gamma):= \sum_{\gamma\in \langle x\rangle}  \alpha_\gamma \gamma
\end{equation}
is well defined if $\Gamma$ is Gromov hyperbolic or of polynomial growth and precisely equal to the delocalized eta invariant of Lott; in formulae
\begin{equation}\label{intro:W}
\langle  \Ind  (D), [\tau_{\langle x \rangle}] \rangle=-\frac{1}{2} \eta_{\langle x \rangle} (D_\partial)
\end{equation}
We call the formula appearing in \eqref{intro:W} a 0-degree {\it delocalized APS index formula.}
Recent work of Chen-Wang-Xie-Yu \cite{ChenWangXieYu},
Sheagan \cite{Sheagan} and Piazza-Schick-Zenobi \cite{PSZ}
extend
this pairing to all elements in $ HC^* (\CC\Gamma, \langle x \rangle)$, $x$ any element in a Gromov hyperbolic or polynomial growth
group, obtaining correspondingly {\it higher delocalized Atiyah-Patodi-Singer
index theorems} for Gromov hyperbolic groups. These express the pairing between $\tau\in HC^* (\CC\Gamma, \langle x \rangle)$
and the Atiyah-Patodi-Singer index class $\Ind (D)$ in terms of certain {\it higher rho numbers}, secondary invariants of Dirac operators that are particularly
useful in studying, for example, metrics of positive scalar curvature. See  \cite{PSZ} for these geometric applications. Notice that these higher rho numbers 
appear here for the boundary operator $D_\partial$  on $\partial Y$ but they can be defined for any Galois covering
$\Gamma\to X \to X/\Gamma$, with $X$ without boundary.

\medskip
We now turn to a $G$-proper manifold, initially without boundary, with compact quotient and a $G$-equivariant Dirac operator $D$
acting on the sections of a $G$-equivariant twisted spinor bundle $E$.
Here $G$ is a unimodular Lie group. 
Let $\mathcal{L}_{G}^c (X,E)$ be the algebra of $G$-equivariant smoothing operators of $G$-compact support.
There is a compactly supported  index class $\Ind_c (D)\in K_* (\mathcal{L}_G^c (X,E))$ and 
 a homomorphism $H^*_{{\rm diff}} (G)\to HC^* (\mathcal{L}_G^c (X,E))$, from  the differentiable cohomology of $G$, $H^*_{{\rm diff}} (G)$, to 
 the cyclic cohomology of $\mathcal{L}_G^c (X,E)$, $HC^* (\mathcal{L}_G^c (X,E))$. We think to these cyclic cocycles in
 $HC^* (\mathcal{L}_G^c (X,E))$
 as localized at the 
 identity element of $G$. 
  The higher  index theorem of Pflaum-Posthuma-Tang \cite{PPT} gives
 a formula for the pairing of $\Ind_c (D)$ with the cyclic cocycle $\tau^M_\varphi$ associated to $[\varphi]\in H^*_{{\rm diff}} (G)$. 
  Under additional assumptions on the group $G$, satisfied for example by a {\it connected, linear real reductive group},  this result was improved to a $C^*$-index theorem in \cite{PP1}. In the latter work a key role 
 is played by a suitable
 dense holomorphically closed subalgebra  $\mathcal{L}_{G,s}^\infty (X,E)$ (Definition \ref{defn:laffargue}) of the Roe algebra $C^* (X,E)^G$ and by a smooth index class $\Ind_\infty (D)
 \in \mathcal{L}_{G,c}^\infty (X,E)$.
  Applications were given once again 
 to higher signatures and higher $\widehat{A}$-genera. This index theorem was extended to $G$-proper manifolds
 with boundary in the recent paper \cite{PP2}; this is a higher APS index theorem on $G$-proper manifolds for cyclic cocyles 
{\it  localized at the identity element}.

 \medskip
 Let us now go back to a $G$-proper manifold without boundary $X$. In contrast with the free case, in the proper case we also have {\it delocalized} index theorems
 associated to delocalized cyclic cocycles. Let $g$ be a semisimple element and set $Z:= Z_G (g)$.
 First of all we have the orbital integral
 $\tau_g : C^\infty_c (G) \to \CC$ associated to such a 
 $g$:
 \begin{equation}\label{orbital}
 \tau_g (f):=\int_{G/Z} f(x g x^{-1}) d(xZ)\,.
 \end{equation}
Assume $G$ is a connected, linear real reductive group. Then $\tau_g$  extends to the Lafforgue Schwartz algebra $\mathcal{L}_s(G)$ (Definition \ref{defn:laffargue}), 
a dense holomorphically closed subalgebra of $C^*_r G$, where it defines a $0$-degree cyclic cocycle $[\tau_g]\in HC^0 (\mathcal{L}_s(G))$.
The orbital integral \eqref{orbital}
also  defines 
a 0-degree cyclic cocycle $\tau^X_g$ on $\mathcal{L}^c_G (X,E)$:
\begin{equation}\label{intro:tr-g-closed}
\tau_g^X (T_{\kappa}):= \int_{G/Z}\int_X c(hgh^{-1}x) {\rm tr} (hgh^{-1}\kappa (hg^{-1}h^{-1}x,x))dx \,d(hZ),
\end{equation}
where $\kappa$ denotes the kernel of an element $T_{\kappa}\in \mathcal{L}^c_G (X,E)$ and $c$ is a cut-off function for the action of $G$ on $X$ and ${\rm tr}$ denotes the vector-bundle fiberwise
trace.
 
 The Lafforgue algebra  $\mathcal{L}_s(G)$ defines an algebra of smoothing operators $\mathcal{L}^\infty_{G,s} (X,E)$
 which is a dense and holomorphically closed subalgebra
of the Roe algebra $C^* (X,E)^G$; moreover $\tau_g^X$ extends from $\mathcal{L}^c_G (X,E)$ to $\mathcal{L}^\infty_{G,s} (X,E)$
and defines on this algebra a 0-degree cyclic cocycle  $[\tau_g^X]\in HC^0 (\mathcal{L}^\infty_{G,s} (X,E))$. We prove the following index formula for any semi-simple element $g$,\begin{equation}\label{HC}
\langle \Ind_\infty (D),\tau^X_g \rangle 
=\int_{X^g}c^g {\rm AS}_g (D).
\end{equation}
In  \cite{Hochs-Wang-HC}, Hochs and Wang partially developed the above formula by considering a splitting Dirac operator, c.f. (\ref{dirac-split}). We give a more detailed approach to (\ref{HC}) in this paper.  For more on the subtleties involved in the proof of this index formula we refer the reader to Section \ref{sect:eta-in-general}. 

This formula establishes a 0-degree delocalized index theorem. In Equation (\ref{HC}),  $X^g$ is the fixed point submanifold associated to the $g$-action on $X$; $c^g$ is a compactly supported cutoff function on $X^g$ for the action of $Z$ on $X^g$. Below, we recall the explicit differential form expression for $\operatorname{AS}_g(D)$. Consider the following curvature forms. 
\begin{itemize} 
\item Since  we are assuming that $E$ is a $G$-equivariant twisted spinor bundle on $X$, we can write 
\[
E = \mathcal{E} \otimes W,
\] 
where $\mathcal{E}$ is the spinor bundle associated to the Spin$^c$-structure on $X$ and $W$ is an auxiliary $G$-equivariant vector bundle on $X$.

We define $R^W$ to be the curvature form of the Hermitian connection on $W$;
\item  $R^\mathcal{N}$, the curvature form associated to the Hermitian connection on $\mathcal{N}_{X^g}\otimes \mathbb{C}$, where $\mathcal{N}_{X^g}$ is the normal bundle of the $g$-fixed point submanifold $X^g$ in $X$;
\item  $R^{L}$, the curvature form associated to the Hermitian connection on $L_{\operatorname{det}}|_{X^g}$ ($L_{\operatorname{det}}$ is  the determinant line bundle of the Spin$^c$-structure on $X$ and $L_{\operatorname{det}}|_{X^g}$ is its restriction to $X^g$);
\item $R_{X^g}$, the Riemannian curvature form associated to the Levi-Civita connection on the tangent bundle of  $X^g$.
\end{itemize}
The ${\rm AS}_g(D)$ in Equation (\ref{HC}) is, by definition, the following expression:
\begin{equation}\label{eq:X-geom-form}
{\rm AS}_g(D):=\frac{\widehat{A}\big(\frac{R_{X^g}}{2\pi i}\big) \operatorname{tr}\big(g \exp(\frac{R^W}{2\pi i})\big) \exp(\operatorname{tr}(\frac{R^L}{2\pi i}))}
{\operatorname{det}\big(1-g \exp(-\frac{R^{\mathcal{N}}}{2\pi i})\big)^{\frac{1}{2}}}.
\end{equation}
This will also be denoted by ${\rm AS}_g (X,E)$ or, if there is no confusion on the vector bundle $E$, simply by ${\rm AS}_g (X)$.

  \medskip
  
 One might wonder if there is a {\it higher} delocalized index theorem; the answer is affirmative and it is work of Song-Tang and Hochs-Song-Tang as we shall now  explain. Let $P<G$ a cuspidal parabolic subgroup and $P=MAN$ its Langlands decomposition. Let $g\in M$ be a semisimple element.
 Song and Tang in \cite{st} have defined a higher delocalized cyclic cocycle $[\Phi^P_g]$ on the Lafforgue Schwartz algebra $\mathcal{L}_s(G) $. For  $m=\operatorname{dim}(A)$, $\Phi^P_g$ is an $m$-cyclic cocycle on $\mathcal{L}_s(G)$ generalizing the orbital integral, Equation (\ref{orbital}).  
 When the metric on a $G$ proper cocompact manifold $X$ without boundary is slice compatible (i.e. there is a slice $Z_0$  equipped with a $K$ action such that $X=G\times_K Z_0$ and the metric on $X$ is obtained from a $K$-invariant metric on $Z_0$ and $G$-invariant metric on $G/K$),  the pairing between $\Phi^P_g$ and the $K$-theory index $\operatorname{Ind}_\infty(D)$ is computed by Hochs-Song-Tang in \cite{hst}. More precisely, the index pairing has the following topological formula, 
 \begin{equation}\label{eq:HST-formula}
\begin{split}
 \langle \Ind_\infty (D),\Phi^P_g \rangle =&\int_{(X/AN)^g}c^g_{X/AN}{\rm AS}(X/AN)_g
\end{split}
 \end{equation}
 In the above formula, $X/AN$ is the quotient of $X$ with respect to the $AN< G$ action; the property of the Langlands decomposition implies that the group $M$ acts on the quotient $X/AN$; $(X/AN)^g$ is the fixed point submanifold of the $g$ action on $X/AN$; $c^g_{X/AN}$ is a smooth compactly supported function on $(X/AN)^g$ defined in the same way as $c^g$ in Equation (\ref{HC}) ; the action of the Dirac operator $D$ on the $AN$-invariant sections of $E$ defines a Dirac type operator $D_{X/AN}$ on $X/AN$; the characteristic classes are defined in a similar but slightly different way from those in Equation (\ref{eq:X-geom-form}) for the $M$ action on $X/AN$ (see Equation (\ref{eq:geom-form}) for its explicit expression).

\medskip
\noindent
We can finally state the main goals of this article: 

\medskip
\noindent
{\it (i) state and prove a (higher) delocalized 
Atiyah-Patodi-Singer index theorem on G-proper manifolds with 
boundary;

\smallskip
\noindent
(ii)  define consequently (higher) rho numbers on 
$G$-proper manifolds without boundary and study their 
properties.}

\medskip
\noindent
We state now  in detail our main results; these deal both with cocompact $G$-proper manifolds with boundary
and without boundary.

\medskip
We begin by stating the results on manifolds with boundary; these are  obtained by combining the results of Hochs, Song and Tang,  the use of (an extension of) Melrose's $b$-calculus
to the present context developed in \cite{PP2} and relative K-theory and relative cyclic cohomology  techniques, as  developed in \cite{moriyoshi-piazza,GMPi,PP2}.

First we give an
improvement,
through the above techniques,  of a result already proved in \cite{HWW2}. 
 Let $G$ be a connected, linear real reductive group. Let $g$ be a semisimple element. Let $Y_0$ be a $G$-proper manifold
 with boundary with compact quotient and let $Y$ be the associated manifold with cylindrical ends. We fix
 a $G$-equivariant slice-compatible metric.
   Let $D_0$ be an equivariant Dirac operator
 on $Y_0$, acting on the sections of an equivariant  vector bundle $E_0$.
We denote by $D$ and $E$  the associated objects on $Y$.
 
 \begin{theorem}\label{intro:0-delocalized-aps} (0-degree delocalized APS on G-proper manifolds)\\
 We assume that $D_{\partial Y}$ is $L^2$-invertible. Then:
 \begin{itemize}
 \item[1)] there exists a dense holomorphically closed
 subalgebra $\mathcal{L}^\infty_{G,s} (Y,E)$ of the Roe algebra $C^* (Y_0\subset Y,E)^G$ and a smooth index class 
 $\Ind_\infty (D)\in K_0(\mathcal{L}^\infty_{G,s} (Y, E)) \cong
 K_0(C^*(Y_0\subset Y,E)^G)$;\\
 \item[2)] the delocalized eta invariant 
 $$\eta_g (D_{\partial Y}):=\frac{1}{\sqrt{\pi}} \int_0^\infty \tau^{\partial Y}_g (D_{\partial Y} \exp (-tD^2_{\partial Y}) \frac{dt}{\sqrt{t}}$$
 exists and for the pairing of the index class $\Ind_\infty(D)$ with the $0$-cocycle  $\tau^Y_g\in HC^0 (\mathcal{L}^\infty_{G,s} (Y,E))$
 defined by the orbital integral $\tau_g$  
 the following delocalized 0-degree index formula holds:
 \begin{equation}
\label{main-0-degree}
 \langle \tau^Y_g,\Ind_\infty (D) \rangle= \int_{(Y_0)^g} c^g {\rm AS}_g (D_0) - \frac{1}{2} \eta_g (D_{\partial Y})\,,
 \end{equation}
where the integrand $c^g {\rm AS}_g(D_0)$ in the above integral is defined as that in Equation (\ref{HC}).  
\end{itemize}
 \end{theorem}
 
 \noindent
This theorem is an improvement with respect to \cite{HWW2} because we do {\em not} assume that $G/Z_G (g)$ is compact. Moreover, the two proofs are different.

\smallskip
\noindent
Consider now $Y/AN$, an $M$-proper manifold,  which has a slice decomposition  given by  $M\times_{K\cap M} Z=:Y_M$.
The following theorem is one of the main results of this paper:

\begin{theorem}\label{intro:main-theorem-higher}
Suppose that the metric is slice compatible (Definition \ref{def:slice-metric}).
Assume that  $D_{\partial Y}$ is $L^2$-invertible and consider the higher 
index $\langle \Phi^P_{Y,g}, \Ind_\infty (D)\rangle$.
The following formula holds:
\begin{equation}\label{formula-main}
\langle \Phi^P_{Y,g}, \Ind_\infty (D)\rangle=
\int_{(Y_0/AN)_g} c^g_{Y_0/AN} {\rm AS}(Y_0/AN)_g-\frac{1}{2} \eta_g (D_{\partial Y_{M}}) 
\end{equation}
with 
$$\eta_g (D_{\partial Y_{M}}) =\frac{1}{\sqrt{\pi}} \int_0^\infty \tau^{\partial Y_{M}}_g (D_{\partial Y_{M}} \exp (-tD^2_{\partial Y_{M}}) \frac{dt}{\sqrt{t}},$$
where the integrand $c^g_{Y_0/AN} {\rm AS}(Y_0/AN)_g$ is defined in the same way as that in Equation (\ref{eq:HST-formula}). We refer to Equation (\ref{eq:geom-form}) for its explicit expression. We regard $\eta_g (D_{\partial Y_{M}})$ as a higher delocalized eta invariant associated to $P$ and $g$.
\end{theorem}
\begin{remark}
We observe that in Theorem \ref{intro:main-theorem-higher}, when $P$ is not maximal cuspidal parabolic subgroup (see Definition \ref{maximal P}), the pairing $\langle \Phi^P_{Y,g}, \Ind_\infty (D)\rangle$ vanishes,  similarly to \cite[Theorem 2.1]{hst}; moreover,  each term on the right side  of Equation (\ref{formula-main}) vanishes because of the existence of extra symmetry (see Remark \ref{vanish rmk-2}).  
\end{remark}

In these two theorems the well-definedness of the (higher)
delocalized eta invariant for the boundary operator is a consequence of the delocalized Atiyah-Patodi-Singer
index theorem. In Section \ref{sect:eta-in-general} we pass to the general case of a cocompact $G$ proper manifold
without boundary (thus not necessarily equal to the
boundary of a cocompact $G$ proper manifold with boundary).  Some discussion about this question was given in \cite[Proposition 4.1 and Section 4.4]{HWW1}.
We make a very detailed study on the most general hypothesis under which the delocalized eta invariant is well defined 
and establish the following result:

\begin{theorem}\label{thm etadefine-intro}
Let $(X,\mathbf{h})$ be a cocompact $G$-proper manifold without boundary 
 and let $D$ be a $G$-equivariant Dirac-type operator (associated to a unitary
 Clifford action and a Clifford connection).  Let $g\in G$ be a semi-simple element. If $D$ is $L^2$-invertible, then the  integral 
\begin{equation}\label{intro:eta-large t}
\frac{1}{\sqrt{\pi}} \int_0^\infty \tau^{X}_g (D\exp (-tD^2) \frac{dt}{\sqrt{t}}
\end{equation}
converges. 
\end{theorem}
The proof, which is rather involved, is divided into two parts: the small $t$  and the large $t$ 
integrability. The former is based on a very detailed study of  the Schwartz kernel
of $D\exp (-t D^2)$ and the way it behaves near the fixed-point set of $g$; the latter is based on a detailed study of the large time behavior of the heat kernel.
Our analysis of the short time behaviour also gives, as a byproduct, a rigorous proof of the Hoch-Wang index formula in \cite{Hochs-Wang-HC}. 
%
%
We end the paper by defining (higher) rho-invariants associated to metrics of positive scalar curvature
 and by studying their properties.

\medskip
\noindent 
 This article uses quite some background knowledge from diverse fields of mathematics such as noncommutative geometry, analysis on manifolds with boundaries and representation theory. For the benefit of the reader we give a quick, but certainly incomplete,  guide to the literature for the necessary background in these subjects

\begin{itemize}
\item The basic tools from noncommutative geometry needed for this paper are cyclic cohomology and its pairing with $K$-theory, as can be found in \cite[Ch. III]{bookconnes}. For manifolds with boundary the formalism of {\em relative} cyclic cohomology is particularly useful, for which we refer to \cite{LMP,moriyoshi-piazza}
\item The basic ingredient from the theory of representations of reductive Lie groups for this paper is the so-called Harish—Chandra algebra of functions on such groups \cite{Knapp}. The connection between representation theory and noncommutative geometry is rooted in the Connes—Kasparov conjecture proved by Wassermann in \cite{Wa:Connes-Kasparov}, see also \cite{Lafforgue}.
\item A standard reference for index theory on manifolds with boundary using the so-called $b$-calculus is the book by Melrose \cite{Melrose-Book}. Important for the construction of the relative cyclic cocycles of this paper is the construction of the $b$-trace and the so-called ``defect-formula’’; these two topics are discussed in detail
in \cite{Melrose-Book}. 
Quick introductions to the $b$-calculus are given in the Appendix of \cite{Mel-P1} and also in the surveys \cite{MP,Grieser,Loya}. This material is adapted to the present context of G-proper manifolds in Section 4B and Section 6 of \cite{PP2}.

\end{itemize}

\bigskip
\noindent
{\bf The paper is organized as follows.} In Section \ref{sect:preliminaries} we give a few geometric preliminaries; these  will play a major role throughout the paper.
Section \ref{sect:aps} is devoted to a proof of the 0-degree delocalized Atiyah-Patodi-Singer, using relative cyclic 0-cocycles associated to orbital integrals.
In Section \ref{sect:eta-in-general}  we prove Theorem \ref{thm etadefine-intro} above (existence of the delocalized eta invariant in the
non-bounding case); crucial in the proof is the study of the large time behaviour of the heat kernel, as an element in the
algebra of integral operators associated to the Lafforgue algebra, and of the short time behaviour near a fixed point-set; the latter can also be used 
 to give a detailed proof of the index formula \eqref{HC}.
In Section \ref{sect:higher-cocycles} we recall the higher cyclic cocycles $\Phi^P_g$ introduced in \cite{st} and we introduce
the cyclic cocycles 
they define on $G$-proper manifolds; we also introduce the relative version of these cocycles on manifolds with boundary. 
In Section  \ref{sect:toward}  and Section \ref{sect:reduction} we explain an approach
to a  general
higher delocalized APS index theorem, using again the interplay between cyclic cocycles and relative cyclic cocycles 
associated to $\Phi^P_g$. In Section \ref{sect:reduction} we do prove such a theorem in the slice compatible case
by adapting to manifolds with boundary 
 a reduction procedure due to Hochs, Song and Tang  \cite{hst}.
 Finally in Section \ref{sect:numeric} we introduce (higher) rho numbers and study their
 properties.

\bigskip
\noindent
{\bf Acknowledgements.}  We would like to thank Pierre Albin, Jean-Michel Bismut, Peter Hochs, Yuri Kordyukov, Xiaonan Ma, Shu Shen, and Weiping Zhang for  inspiring discussions. Piazza was partially supported by 
the PRIN {\it Moduli spaces and Lie Theory} of MIUR ({\it Ministero Istruzione Universit\`a Ricerca}); Song was partially supported by NSF Grant DMS-1800667 and DMS-1952557; Tang was partially supported by NSF Grant DMS-1800666 and DMS-1952551. Part of the research was carried out within the online Research Community on Representation Theory and Noncommutative Geometry sponsored by the American Institute of Mathematics.

\section{Notations}


We use the following notations throughout this paper:
\begin{itemize}
\item $G$ = connected, linear real  reductive group with maximal compact subgroup $K$;
	\item $Y_0$ = cocompact proper $G$-manifold with boundary;
	\item $\mathbf{h}_0$ a $G$-invariant metric on $Y_0$, product-type near the boundary;
	\item $Z_0$ = a $K$-slice of $Y_0$ so that $Y_0$ is diffeomorphic to $G\times_K Z_0$; and abusing notation, we will write $Y_0=G\times_K Z_0$; 	
	\item $Y$ = the $G$-proper $b$-manifold associated to $Y_0$;
	\item  $\mathbf{h}$ the $G$-invariant $b$-metric on $Y$ associated to $\mathbf{h}_0$;
	\item $Z$ = the $b$-manifold associated to the slice $Z_0$; $Z$ is a $K$-slice of $Y$ so that $Y = G\times_K Z$;
	\item $X$ = cocompact proper $G$-manifold without boundary;
	\item $S$ = a $K$-slice of $X$ so that $X = G \times_K S$; 
	\item $E$ = $G$-equivariant twisted spinor bundle so that $E= \mathcal{E} \otimes W$, where $\mathcal{E}$ is the spinor bundle associated Spin$^c$-structure and $W$ is an auxiliary  $G$-equivariant vector bundle;
	\item $C^\infty_c(G)$ = space of compactly supported smooth functions on $G$;
	\item $\mathcal{L}_s(G)$ = space of Lafforgue's Schwartz functions on $G$ associated to the norm $\nu_s$;
	\item $\mathcal{L}_{G}^c(X)$ = the space of $G$-equivariant $G$-compactly supported smoothing operators;
	\item ${}^b\mathcal{L}_G^c(Y)$ = the space of $G$-equivariant  $G$-compactly supported $b$-smoothing operators in the $b$-calculus with $\epsilon$-bounds;
	\item $\mathcal{L}_{G,s}^\infty(X)$ = the space of $G$-Lafforgue Schwartz smoothing operators associated to the norm $\nu_s$;
	\item $\mathcal{L}_{G,s}^\infty(Y)$ = the space of $G$-Lafforgue Schwartz residual smoothing operators in the  $b$-calculus associated to $\nu_s$
	with $\epsilon$-bounds ;
	\item ${}^b\mathcal{L}_{G,s}^\infty(Y)$= the space of $G$-Lafforgue Schwartz $b$-smoothing operators  in the $b$-calculus associated to $\nu_s$ with $\epsilon$-bounds;
	\item $C^*_r(G)$=reduced group $C^*$-algebra of $G$;
	\item $\tau_g$ = orbital integral on $\mathcal{L}_s(G)$ associated to the conjugacy class of $g$;
	\item $\tau_g^X$ = trace on $\mathcal{L}_{G,s}^\infty(X)$ associated to the orbital integral $\tau_g$;
	\item $\tau_g^Y$ = trace on $\mathcal{L}_{G,s}^\infty(Y)$ associated to the orbital integral $\tau_g$.
\end{itemize}

\section{Geometric preliminaries.}\label{sect:preliminaries}

Let $Y_0$ be a  manifold with boundary, $G$ a {connected linear real reductive Lie group  acting properly
and cocompactly on $Y_0$. We denote by $X$ the boundary of $Y_0$.
There exists a collar neighbourhood $U$ of the boundary $\pa Y_0$, $U\cong [0,2]\times \partial Y_0$,
 which is $G$-invariant and such that the action of $G$ on $U$ is of product type.
We assume that $Y_0$ is endowed with a $G$-invariant metric $\mathbf{h}_0$ which is of product type near the boundary. We let
$(Y_0,\mathbf{h}_0)$ be the resulting Riemannian  manifold with boundary;  in the collar
neighborhood $U\cong [0,2]\times \partial Y_0$ the metric  $\mathbf{h}_0$  can be written, through the above isomorphism,
 as $dt^2 + \mathbf{h}_X$,
with $\mathbf{h}_X $ a $G$-invariant Riemannian metric on  $X=\partial Y_0$. We denote by $c_0$ 
a cut-off function for the action of $G$ on $Y_0$; since the action is cocompact, this is a compactly supported smooth function.
We consider the associated manifold with cylindrical ends
$\widehat{Y}:= Y_0\cup_{\partial Y_0} \left(   (-\infty,0] \times \partial Y_0 \right)$,
endowed with the extended metric $\widehat{\mathbf{h}}$ and the extended $G$-action.
We denote by  $(Y,\mathbf{h})$ the $b$-manifold associated to  $(\widehat{Y},\widehat{\mathbf{h}})$. We shall often treat 
$(\widehat{Y},\widehat{\mathbf{h}})$ and $(Y,\mathbf{h})$ as the same object. We denote by $c$ the  obvious extension of the cut-off function
$c_0$ for the action of $G$ on $Y_0$ (constant along the cylindrical end); this is a cut-off function of the extended action of $G$ on $Y$.
If $x$ is a boundary defining function for the cocompact $G$-manifold $Y_0$, then the $b$-metric $\mathbf{h}$ has the following
product-structure near the boundary $X$:
 $$\frac{dx^2}{x^2}+ \mathbf{h}_X$$
We remark at this point that our arguments will actually apply to the more general case of {\it exact}
$b$-metrics, or, equivalently, manifolds with asymptotic cylindrical ends. We shall not insist on this point.

In this article we shall be interested in the case in which $Y_0$ admits a $G$-invariant Spin$^c$-structure.  Let $E_0$ be the $G$-equivariant spinor bundle associated to a Spin$^c$-structure on $Y_0$ twisted by an auxiliary  $G$-equivariant vector bundle. In particular, $E_0$ has a $G$-equivariant $\text{Cliff}(TY_0)$-module structure. Let $\nabla^{E_0}$ be a \emph{Clifford connection} on $E_0$, that is
\begin{equation}
\label{Clifford connection}
\left[\nabla^{E_0}_V, c(W) \right] = c(\nabla^{TY_0}_VW), \quad V, W \in C^\infty(Y_0, TY_0)
\end{equation}
where $c$ denotes the Clifford action and $\nabla^{TY_0}$ is the Levi-Civita connection. The Dirac operator associated to the Clifford connection is given by the following composition
\begin{equation}
\label{def Dirac}
D_{Y_0} \colon C^\infty(Y_0, E_0)\xlongrightarrow{\nabla^{E_0}} C^\infty(Y_0, T^*Y_0 \otimes E_0)  \cong C^\infty(Y_0, TY_0\otimes E_0) \xlongrightarrow{c}C^\infty(Y_0, E_0). 	
\end{equation}

\begin{assumption}\label{dim ass}
Suppose that $G$ is a connected reductive Lie group with maximal compact subgroup $K$. We denote by 
\[
\mathfrak{r}  \colon = \text{dim} \ G/K  \ (\text{mod} \ 2), \quad \mathbf{t}  \colon = \text{dim} \ Y_0  \ (\text{mod} \ 2).
\]
The index class
\[
\Ind_\infty (D_{Y_0}) \in K_{\mathbf{t}}\left(\mathcal{L}_{G,s}^\infty \left(Y,E\right)\right) \cong K_\mathbf{t}(C^*_rG).
\]
Recall that  $K_{\mathfrak{r}+1}(C^*_rG) = 0$, see \cite{Lafforgue}. We assume that 
\[
 \text{dim} \ G/K  = \text{dim} \ Y_0 \ (\text{mod} \ 2).
\]
Otherwise, the index class is written as 
\[
\Ind_\infty (D_{Y_0}) \in K_\mathbf{t}(C^*_rG) = K_{\mathfrak{r}+1}(C^*_rG) = 0.
\]
Without loss of generality, we can further assume 
\[
 \text{dim} \ G/K  = \text{dim} \ Y_0 = 0 \ (\text{mod} \ 2);
\] 
otherwise, we can simply replace $G$ (and $Y_0$) by $G \times \mathbb{R}$ (respectively $Y_0 \times \mathbb{R}$). To sum up, we assume that 
\begin{enumerate}
	\item the symmetric space $G/K$;
	\item the $G$-manifold $Y_0$.
\end{enumerate}
are all even dimensional.
\end{assumption}

\medskip
\noindent
For any cuspidal parabolic subgroup $P = MAN$ with $m = \dim(A)$,  the $m$-cyclic cocycle introduced below in (\ref{eq:PhiPg}), $\Phi^P_g$, is an $m$-cyclic cocycle on $\mathcal{L}_s(G)$. The pairing 
\[
\langle \Phi^P_{Y,g}, \Ind_\infty (D)\rangle
\]
is automatically zero if $m$ is odd. Thus it suffices to consider those cuspidal parabolic subgroups such that $m = \dim(A)$ is even. In particular, for the maximal parabolic subgroup $P = MAN$, it is always this case under our assumption.

\subsection{Slice compatible metric on $Y_0$}
Let us  fix a slice $Z_0$ for the $G$ action on $Y_0$; thus
$$Y_0\cong G\times_K Z_0$$
with $K$ a maximal compact subgroup of $G$ and $Z_0$ a smooth compact manifold  with boundary, denoted by $S = \partial Z_0$, endowed with a $K$-action. Consequently, $Y\cong G\times_K Z$
with $Z$ the $b$-manifold associated to $Z_0$ and the boundary
\[
X = \partial Y_0 \cong G \times_K S. 
\] 

Choose a $K$-invariant inner product on the Lie algebra $\mathfrak{g}$ of $G$, so that we have the so called Cartan decomposition 
$\mathfrak{g}=\mathfrak{k}\oplus\mathfrak{p}$ where $\mathfrak{k}$ is the Lie algebra of $K$ and 
$\mathfrak{p}$ its orthogonal complement. 
We have an  isomorphism
\begin{equation}
\label{tb-induced}
T Y_0\cong G\times_K(\mathfrak{p}\oplus TZ_0).
\end{equation}
Here we abuse the notation $\mathfrak{p}$ to denote the trivial vector bundle $Z_0\times \mathfrak{p} \to Z_0$. 
\begin{definition}\label{def:slice-metric} (slice compatible metrics) Given a slice $Z_0$, we shall say that a $G$-invariant metric on $Y_0$ is slice compatible with $Z_0$ if it is constructed from a $K$-invariant metric on $Z_0$ and a $K$-invariant metric on $\mathfrak{p}$ via the above isomorphism
$T Y_0\cong G\times_K(\mathfrak{p}\oplus TZ_0)$. We shall say that 
the  $G$-invariant metric $\mathbf{h}_0$ on $Y_0$ is slice-compatible if there is a slice $Z_0$ 
 such that it is slice compatible with $Z_0$.
\end{definition}}

When $Y_0$ has a $G$-equivariant Spin$^c$-structure, we can construct its spinor bundle as follow. We can assume, up to the passage to  double covers, 
 that
the adjoint representation ${\rm Ad}: K\to SO (\mathfrak{p})$ admits a lift $\widetilde{{\rm Ad}}: K\to {\rm Spin}(\mathfrak{p})$. We then obtain 
a $G$-invariant Spin$^c$-structure $P^{G/K}:= G\times_K {\rm Spin}(\mathfrak{p})\to G/K$. Assume now that $Z_0$ admits a $K$-invariant
Spin$^c$-structure. Then, proceeding as in \cite{Hochs2009,Hochs-Mathai}, we obtain a $G$-invariant Spin$^c$-structure on $M$. Because
\[
0 \to G \times_K \mathfrak{p} \to TY_0 \to G \times_K TZ_0 \to 0, 
\] 
the two out of three lemma of Spin$^c$-structure shows that  every $G$-invariant Spin$^c$-structure on $Y_0$ is induced from a $K$-invariant Spin$^c$-structure on $Z_0$. The vector bundle $E_0$ induces a $K$-equivariant vector bundle $E_{Z_0}$ such that 
\[
E_0 \cong G \times_K( S_\mathfrak{p} \otimes E_{Z_0})
\]
and $E_{Z_0}$ admits a $K$-equivariant $\text{Cliff}(TZ_0)$-module structure and $S_\mathfrak{p}$ is the spinor bundle along the $\mathfrak{p}$ direction. We call the above a \emph{slice compatible Spin$^c$-structure}. Throughout this paper, we only need to assume that the Spin$^c$-structure on $Y_0$ is slice compatible in section \ref{sect:reduction}.

 We decompose 
\[
L^2(Y_0, E_0) \cong \left[L^2(G) \otimes S_\mathfrak{p} \otimes L^2(Z_0, E_{Z_0}) \right]^K. 
\]
We can define a \emph{split Dirac operator} $D_{\text{split}}$ by the following formula
\begin{equation}\label{dirac-split}
D_{\text{split}} = D_{G,K} \hat{\otimes} 1 + 1\hat{\otimes} D_{Z_0},  
\end{equation}
where $D_{G,K}$ is the Spin$^c$-Dirac operator on $L^2(G)\otimes S_\mathfrak{p}$, and $D_{Z_0}$ is a $K$-equivariant Dirac operator on $E_{Z_0}$, and $\hat{\otimes}$ means the graded tensor product. We point out here that the two Dirac operators $D_{Y_0}$ and $D_{\text{split}}$ are different in general. We would like to thank Jean-Michel Bismut and Xiaonan Ma for pointing out this difference to us. It was incorrectly stated
 in \cite{HS} and used in the  articles\footnote{The list of references using this wrong property might be incomplete.} \cite{MGW, hst, Hochs-Wang-HC, HWW1,HWW2} that these two Dirac operators are identical. In the appendix, we consider the example where $Y_0= G$
 and give an explicit formula for the two Dirac operators, showing in particular that their difference 
 is different from zero. Nevertheless, as the operators have the same principal symbol they give the same index class provided the $G$ proper manifold has no boundary; thus in this case
\[
\Ind_\infty (D_{Y_0}) = \Ind_\infty (D_{\text{split}})  \in K_{0}(C^*_rG).
\]
The split Dirac operator $D_{\text{split}}$ has been extensively studied in \cite{hst, HWW1, HWW2}. It is important to point out that the connection used in the definition of $D_{\text{split}}$ might not be the Clifford connection (see (\ref{Clifford connection})); consequently,
the short time behaviour of the heat kernel is ill behaved for $D^2_{\text{split}}$. The main goal of this paper is to study delocalized APS index theory, where the choice of connection is even more  crucial, given that it affects the invertibility properties of the boundary operator. Hence, in this paper it is crucial that we work with the Dirac operator introduced in Equation (\ref{def Dirac}).

 \section{Delocalized traces and the  APS index formula}\label{sect:aps}
 
 In this section we want to tackle the case of a delocalized APS index theorem on $G$-proper manifolds
 with boundary, that is, 0-degree delocalized cyclic cocycles. Similar results have been discussed in \cite{HWW1, HWW2}. 
 Our treatment is different, centred around the interplay between absolute and relative cyclic cohomology and the $b$-calculus; moreover our treatment
 allows us to  get sharper results compared to \cite{HWW2} in the case of a connected linear real reductive group\footnote{In \cite{HWW2}, the authors consider general locally compact topological groups.} $G$. More precisely, in Theorem \ref{intro:0-delocalized-aps} we only assume that $g$ is a semisimple element of $G$ to obtain the index formula (\ref{main-0-degree}), while in \cite[Theorem 2.1]{HWW2}, the authors require that $G/Z_g$ is compact\footnote{Notice that for the numeric $g$-index associated to $D$
  a formula is proved in \cite{HWW1} under  the same hypothesis given here; on the other hand, the $g$-index is proved 
 to be equal to the pairing of an index class with a 0-degree cyclic cocycle (this is the number we consider in the present article)
 only under the additional assumption that  $G/Z_g$ is compact.}.
\\
 
 \subsection{Orbital integrals and associated cyclic $0$-cocycles in the closed case}
 \label{oiclosed}
   Consider first a co-compact $G$-proper manifold {\it without} boundary $X$.
We know that 
there exists a compact submanifold
$S\subset X$ on which the $G$-action restricts to an action of a maximal compact subgroup $K\subset G$, so that 
the natural map
\[
G\times_K S\to X,\quad [g,x]\mapsto g\cdot x,
\]
is a diffeomorphism. This decomposition of $X$ induces an isomorphism
\begin{equation}
\label{algebra-slice}
\mathcal{L}_G^c (X)\cong \left(C^\infty_c(G)\hat{\otimes}\Psi^{-\infty}(S)\right)^{K\times K}
\end{equation}
and more generally
\begin{equation}
\label{algebra-slice-bis}
\mathcal{L}_G^c (X,E)\cong \left(C^\infty_c(G)\hat{\otimes}\Psi^{-\infty}(S,E|_S)\right)^{K\times K},
\end{equation}
in the presence of a $G$-equivariant vector bundle $E$.
(We shall often expunge the vector bundle $E$ from the notation.)
On the left hand side we have the
smoothing operators defined by $G$-invariant smooth kernels of $G$-compact support;
on the right hand side we use the (unique) completion of the algebraic tensor product between the two algebras given that they are both nuclear.  \\

\begin{definition}\label{defn:laffargue}
For $s \in [0, \infty)$, define the Lafforgue algebra $\mathcal{L}_s(G)$ to be the completion of $C_c(G)$ with respect to the norm $\nu_s$ defined as follows
\[
\nu_s(f) \colon = \sup_{g\in G} \left\{(1 + \|g\|)^s \cdot \Xi^{-1}(g)\cdot |f(g)|\right\},
\]	
where $\Xi(g)$ denotes Harish-Chandra's spherical function. 
\end{definition}
The family of Banach algebras $\{\mathcal{L}_s(G)\}_{s\geq 0}$ satisfies the following properties \cite{Lafforgue}:
\begin{enumerate}
	\item For every $s \in [0, \infty)$, $\mathcal{L}_s(G)$ is a dense subalgebra of $C^*_r(G)$ stable under holomorphic calculus.
	\item For $0 \leq s_1 < s_2$, $\|f\|_{s_1} \leq\|f\|_{s_2}$,  $\forall f \in  \mathcal{L}_{s_2}(G)$. Hence, $\mathcal{L}_{s_2}(G) \subset  \mathcal{L}_{s_1}(G)$.  Define
	\[
	\mathcal{L}(G)\colon = \cap_{s \geq 0} \mathcal{L}_{s}(G). 
	\]
	\item For any semisimple element $x \in G$, there exists $d_0>0$, such that  $\forall s>d_0$, the orbital integral
	\[
	\int_{G/Z_x} f(gxg^{-1}) \; dgZ_x
	\]
	is convergent for all $ f \in \mathcal{L}_s(G)$.
\end{enumerate}
We observe that Harish-Chandra's Schwartz algebra is contained in the algebra $\mathcal{L}_s(G)$ for any $s>0$.\\ 

Throughout the paper, for a fixed $x\in G$, we work with a Lafforgue Schwartz algebra $\mathcal{L}_s(G)$ for a sufficiently large $s$ so that the orbital integral is a well defined continuous linear functional on $\mathcal{L}_s(G)$. 

\medskip
Let us consider now a cocompact $G$-proper manifold without boundary $X\cong G\times_K S$.  
The algebras  $\mathcal{L}_{G,s}^{\infty}(X)$ and $\mathcal{L}^\infty_G (X)$ are defined as
	\[
	\mathcal{L}^\infty_{G,s} (X):= \left(\mathcal{L}_s(G)\hat{\otimes}\Psi^{-\infty}(S)\right)^{K\times K},\ \mathcal{L}^\infty_G (X):= \left(\mathcal{L}(G)\hat{\otimes}\Psi^{-\infty}(S)\right)^{K\times K}.
	\]  
	There are similar algebras when we consider a $G$-equivariant vector bundle $E$ on $X$; however,
	for notational simplicity, we shall expunge the vector bundle $E$ from the notation.
	 For  
$$\widetilde{k}\in \mathcal{L}^\infty_G (X):=\left(\mathcal{L}(G)\hat{\otimes}\Psi^{-\infty}(S)\right)^{K\times K}
$$ 
or 
$$\widetilde{k}\in \mathcal{L}^\infty_{G,s} (X):=\left(\mathcal{L}_s (G)\hat{\otimes}\Psi^{-\infty}(S)\right)^{K\times K}
$$ 
we consider the bounded operator $T_{\widetilde{k}}$
on $L^2 (X)$ given by
\begin{equation}\label{correspondence-slice}
(T_{\widetilde{k}} e)(gs)=\int_G \int_S g \widetilde{k}(g^{-1}g^\prime,s,s^\prime) g^{\prime\,-1} e (g^\prime s^\prime)ds^\prime dg^\prime.\end{equation}
The operator $T_{\widetilde{k}}$ is an integral operator with  $G$-equivariant Schwartz kernel
$\kappa$ given by $$\kappa (gs,g^\prime s^\prime)=g \widetilde{k} (g^{-1} g^\prime, s, s^\prime) g^{\prime\,-1}\,.$$
The map  $\widetilde{k}\to T_{\widetilde{k}}$ is injective; moreover
$$T_{\widetilde{k}} \circ T_{\widetilde{k}^\prime} = T_{\widetilde{k}* \widetilde{k}^\prime }$$ 
so that its image is a subalgebra of the $G$-equivariant bounded operators on $L^2 (X)$. Following
an established abuse of notation we shall not distinguish between these two algebras, thus identifying
a smooth kernel with the bounded operators it defines.

Following \cite{PP2}, there are isomorphisms that associate to a smooth $G$-equivariant kernel $A$ on $X\times X$  a map $\Phi_A: G\to \Psi^{-\infty}(S)$,
with $\Phi_A$
equivariant with respect to the natural $K\times K$ action on $\Psi^{-\infty}(S)$ and the action $\alpha(k_1,k_2)(g):=k_1g k_2^{-1}$ on $G$. More precisely,
by  \cite[Prop. 1.7]{PP2}, there are  isomorphisms
\begin{align}\label{PHI}
\mathcal{L}^c_G(X)&\cong\left\{\Phi:G\to \Psi^{-\infty}(S),~\mbox{smooth, compactly supported and }
~K\times K~\mbox{invariant}\right\},\\
\mathcal{L}^\infty_{G,s}(X)&\cong\left\{\Phi:G\to \Psi^{-\infty}(S),~K\times K~\mbox{invariant and}~g\mapsto 
v_{s} (\| \Phi(g)\|_\alpha)\;\;\text{bounded}\right\}
\end{align}
with $\alpha$ a multi-index indexing derivatives with respect to the spacial variables of $S$,  $\| \;\|_\alpha$ denoting the 
associated well-known seminorm on $\Psi^{-\infty} (S)$  and $v_{s}(\,\,)$,
denoting the Lafforgue norm.
The image of the product by convolution of $A$ and $B$ on the left hand side
is equal to $\Phi_A * \Phi_B$ on the right hand side,
with
\begin{equation}\label{convolution-Phi}
(\Phi_A * \Phi_B) (g)=\int_G \Phi_A (gh^{-1})\circ \Phi_B (h) dh\,.
\end{equation}

Following again \cite{Hochs-Wang-KT} we define for $T= T_{\widetilde{k}}$
\begin{equation}\label{tr-g-closed}
\tau_g^X (T):= \int_{G/Z_g}\int_X c(hgh^{-1}x) {\rm tr} (hgh^{-1}\kappa (hg^{-1}h^{-1}x,x))dx \,d(hZ)
\end{equation}
with $c$ a cut-off function for the action of $G$ on $X$ and ${\rm tr}$ denoting the vector-bundle fiberwise
trace. There are equivalent ways to write the right-hand-side,
see \cite[Lemma 3.2]{Hochs-Wang-KT}; for example if $c_G$ is a cutoff function for the action of $Z_g$
on $G$ by right multiplication and
$c^g (x)= \int_G c_G (h) c(hgx)dh$
then
the right hand side of \eqref{tr-g-closed}
can be written as 
\begin{equation}\label{tr-g-closed-bis}
\int_X c^g (x) {\rm tr} (\kappa (x,gx)g) dx.
\end{equation}
 It is proved in \cite[Lemma 3.4]{Hochs-Wang-KT} that $\tau_g^{X}$ defines a continuous 
 trace
 \begin{equation}\label{delocalized-trace-on-closed}
\tau_g^{X}: \mathcal{L}^\infty_{G,s} (X)\to\CC\,.
\end{equation} 
In fact, the two traces \eqref{tr-g-closed} and \eqref{delocalized-trace-on-closed} are related by a homomorphism
of integration along the slice $S$,
${\rm Tr}_S: \mathcal{L}^\infty_{G,s} (X)\to \mathcal{L}_s(G)$, such that 
$$\tau^X_g=\tau_g \circ {\rm Tr}_S$$
(see \cite[Section 3.3]{Hochs-Wang-KT}). We see that   $\Tr_S$ associates  to $\Phi$ the function $$G\ni g\to \Tr(\Phi (g))\,.$$
Using  \cite[Lemma 1.24]{PP2} and the well known inequality $|\Tr (T)|\leq \| T \|_1$  for a smoothing operator 
on a smooth compact manifold\footnote{with $\|\;\;\|_1$
denoting the trace norm}, we see that 
 $\Tr_S: \mathcal{L}^\infty_{G,s} (X)\to \mathcal{L}_s(G)$ so defined is a  {\it continuous} map.
  (Even though \cite[Lemma 1.24]{PP2} uses a slightly different algebra on $G$ instead of $\mathcal{L}_s(G)$, the proof showing that the trace norm $A\mapsto ||A||_1,~A\in\Psi^{-\infty}(S)$ is continuous for the Fr\'ech\`et topology on $\Psi^{-\infty}(S)$ applies verbatim to show continuity in our case.) Exactly the same results hold for the algebra $\mathcal{L}^\infty_{G,s} (X)$. \\

\medskip
Let now $D$ be an equivariant Dirac operator, of product type near the boundary.
We shall make  the following assumption:

\begin{equation}\label{assumption}
 \text{the boundary operator}\;\;
D_{\partial Y}\;\; \text{is}\;\;L^2\text{-invertible}.
\end{equation}

\medskip
The following
Theorems sharpen the corresponding results in  \cite{PP2}, (see in particular Subsection 5.3, Proposition 5.27, Proposition 5.33 
and Theorem 5.42 there). 

\begin{theorem}\label{theo:smooth-index-CS}
Let $D$ be as above and let $Q^\sigma$ be a symbolic $b$-parametrix for $D$. The Connes-Skandalis projector
\begin{equation}\label{CS-projector-true}
P^b_{Q}:= \left(\begin{array}{cc} {}^b S_{+}^2 & {}^b S_{+}  (I+{}^b S_{+}) Q^b \\ {}^b S_{-} D^+ &
I-{}^b S_{-}^2 \end{array} \right)
\end{equation}
associated to a true $b$-parametrix $Q^b=Q^\sigma-Q^\prime$ with remainders ${}^b S^\pm$ in $\mathcal{L}^\infty_{G,s} (Y_0)$ is  a $2\times 2$ matrix with entries in $\mathcal{L}^\infty_{G,s} (Y)$
\footnote{really in a slightly extended algebra because of the identity appearing in the right lower corner of the matrix}. We thus have a well-defined {\it smooth} index class
\begin{equation}\label{CS-class-bis}
\Ind_\infty (D):=[P^b_{Q}] - [e_1]\in K_0(\mathcal{L}^\infty_{G,s} (Y))\equiv 
K_0 (C^*(Y_0\subset Y)^G)
\;\;\;\text{with}\;\;\;e_1:=\left( \begin{array}{cc} 0 & 0 \\ 0&1
\end{array} \right).
\end{equation}
\end{theorem}

\begin{theorem}\label{theo:smooth-index-CM}
{\ }
\begin{itemize}
\item[1)] The Connes-Moscovici projector $V(D)$,  
\[
V(D):=\left( \begin{array}{cc} e^{-D^- D^+} & e^{-\frac{1}{2}D^- D^+}
\left( \frac{I- e^{-D^- D^+}}{D^- D^+} \right) D^-\\
e^{-\frac{1}{2}D^+ D^-}D^+& I- e^{-D^+ D^-}
\end{array} \right),
\]
is a $2\times 2$ matrix with entries in ${}^b\mathcal{L}^\infty_{G,s} (Y_0)$;
the Connes-Moscovici projector $V(D^{\cyl})$ is a $2\times 2$ matrix with entries in ${}^b\mathcal{L}^\infty_{G,s,\RR} ({\rm cyl}(\partial Y))$. These two projectors define a smooth relative index class  $\Ind_\infty (D,D_\partial )\in 
K_0({}^b \mathcal{L}^\infty_{G,s} (Y_0),{}^b \mathcal{L}^\infty_{G,s,\RR} ({\rm cyl}(\partial Y)))$.\\
\item[2)] The projector $V^b (D)$ obtained by improving the parametrix $ Q
:= \frac{I-\exp(-\frac{1}{2} D^- D^+)}{D^- D^+} D^-$ defining $V(D)$ to a true $b$-parametrix $Q^b$
is a $2\times 2$ matrix with entries
in $\mathcal{L}^\infty_{G,s} (Y)$ and defines the same smooth index class in  $K_0(\mathcal{L}^\infty_{G,s} (Y))$
  as the Connes-Skandalis
projector of Theorem \ref{theo:smooth-index-CS}\\
\item[3)] The class $\Ind_\infty (D)$ is sent to the class $\Ind_\infty (D,D_\partial )$ through the excision isomorphism
$\alpha_{{\rm exc}}$.
\end{itemize}
\end{theorem}

\noindent
Summarizing: using the Connes-Moscovici projector(s) we have smooth index classes 
  $\Ind_\infty (D)\in K_0(\mathcal{L}^\infty_{G,s} (Y))$ and $\Ind_\infty (D,D_\partial )\in 
K_0 ({}^b \mathcal{L}^\infty_{G,s} (Y),{}^b \mathcal{L}^\infty_{G,s,\RR} ({\rm cyl}(\partial Y)))$,
with the first one sent to the second one by the excision isomorphism
$\alpha_{{\rm exc}}$.

\noindent
\begin{proof}
The proof proceeds as in \cite{PP2}. Recall, in particular, that the relative index class is defined by the triple
\begin{equation}\label{pre-wassermann-triple}
(V(D), e_1, q_t)
\,, \;\;t\in [1,+\infty]\,,\;\;\text{ with }
q_t:= \begin{cases} V(t D_{\cyl})
\;\;\quad\text{if}
\;\;\;t\in [1,+\infty)\\
e_1 \;\;\;\;\;\;\;\;\;\;\;\;\;\,\text{ if }
\;\;t=\infty
 \end{cases}
\end{equation}
 and with $e_1:=\begin{pmatrix} 0&0\\0&1 \end{pmatrix}$.  The well-definedness of this class is a consequence
of the large time behaviour of the heat kernel on $b$-manifolds for invertible operators, treated in detail in the next section.
Notice that most of the arguments given in \cite{PP2} use the closure under holomorphic calculus of certain algebras. 
This applies unchanged
whether we use pseudodifferential operators based on the Lafforgue algebra $\mathcal{L}_s(G)$
or the rapid decay algebra $H^\infty_L (G)$.
The only exception is the analogue of Lemma 2.7 in  \cite{PP2} that we state now explicitly, in the present case, for the benefit of the reader:

\begin{lemma}\label{lemma:composition} Let $X$ be a cocompact $G$-proper manifold without boundary with a slice $S$. Consider
$ \Psi^{-\infty}_{G,c}(X)$
and its extension $\mathcal{L}^\infty_{G,s} (X):=(\mathcal{L}_s(G)\hat{\otimes}\Psi^{-\infty}_c(S))^{K\times K}$,
an algebra of smoothing operators.
Then the composition $\Psi^0_{G,c}(X)\times \Psi^{-\infty}_{G,c}(X)\to  \Psi^{-\infty}_{G,c}(X)$
extends to a continuous map
\[
\Psi^0_{G,c}(X)\times \mathcal{L}^\infty_{G, s}(X)\to \mathcal{L}^\infty_{G,s}(X).
\]
\end{lemma}
\medskip
\noindent
{\it Proof of the Lemma}\\
The proof  is a variation of the corresponding one for  $(H^\infty_L (G)\hat{\otimes}\Psi_c^{-\infty}(S))^{K\times K}$ given in \cite[Lemma 2.7]{PP2}. Recall that $\Psi^0_{G,c}(X)=\big( \Psi^0_{G,c}(G)\hat{\otimes} \Psi^0_c(S) \big) ^{K\times K}$. As $\Psi^{-\infty}_c(S)$ is an ideal of $\Psi^{0}_c(S)$, we are reduced to prove that the product 
\[
\Psi^0_c(G)\times \mathcal{L}_c(G)\to \mathcal{L}_c(G),\ (A, f)\mapsto A\ast f
\]
is well-defined and continuous with respect to $f$. 

A general element $A$ of $\Psi^0_c(G)$ can be written as 
\[
A=\operatorname{Op}(a)+K,
\]
where $K\in \Psi^{-\infty}_{G,c}(G)\cong C^\infty_c(G)$ and $\operatorname{Op}(a)$ is the operator corresponding to a symbol $a\in S^0(\mathfrak{g}^*)$ of order zero with property 
\[
|D_\xi^\alpha(a)|\leq C_\alpha (1+|\xi|)^{-|\alpha|}.
\]

It follows from the inclusion $C^\infty_c(G)\subset \mathcal{L}_s(G)$ that the product $(K, f)\mapsto K\ast f$ is well-defined and continuous for $K\in \Psi^{-\infty}_{G, c}(G)$.  We are left to show that the composition $(\operatorname{Op}(a), f)\mapsto \operatorname{Op}(a)(f)$ belongs to $\mathcal{L}(G)$ and is continuous.  Recall that $\operatorname{Op}(a)(f)$ has the following expression
\[
\operatorname{Op}(a)(f)(g):=\int _G \int_{\mathfrak{g}^*}\chi(h^{-1}) e^{i\langle \xi, \exp^{-1}(h^{-1})\rangle} a(\xi)f(hg)d\xi dh. 
\]
Let $\Xi(g)$ be Harish-Chandra's spherical function as before. We have
\[
\begin{split}
&\left(1+\|g\|\right)^s \Xi(g)^{-1}\big(\operatorname{Op}(a)(f)(g) \big)\\
=& \int _G \int_{\mathfrak{g}^*}\chi(h^{-1}) e^{i\langle \xi, \exp^{-1}(h^{-1})\rangle} a(\xi)f(hg) \left(1+\|g\|\right)^s \Xi(g)^{-1} d\xi dh\\
=& \int _G \int_{\mathfrak{g}^*}\chi(h^{-1}) e^{i\langle \xi, \exp^{-1}(h^{-1})\rangle} a(\xi) \frac{\Xi (hg)}{\Xi(g)} \Xi(hg)^{-1} f(hg) \left(1+\|g\|\right)^s  d\xi dh
\end{split}
\]

By \cite[Lemma 12.5]{Knapp}, for $h^{-1}$ in the support of the function $\chi$, the support of which is compact, there is a constant $C>0$ such that 
\[
\left|\frac{\Xi(hg)}{\Xi(g)}\right|<C,\ \forall g\in G, h\in \operatorname{supp}(\chi). 
\]
Peetre's inequality gives,
\[
 \left(1+\|g\|\right)^s\leq \left(1+\|h^{-1}\|\right)^{s} \cdot \left(1+\|hg\|\right)^s. 
\]
It follows from the above two inequalities that we have that 
\[
\lVert  \left(1+\|g\|\right)^s\Xi(g)^{-1} \big(\operatorname{Op}(a)(f)(g) \big)\rVert
\] 
is bounded by 
\[
C  \Big\lVert \int _G \int_{\mathfrak{g}^*}\chi(h^{-1}) )e^{i\langle \xi, \exp^{-1}(h^{-1})\rangle} a(\xi) (1+\|h\|))^{s}d \xi dh \Big\rVert\times \operatorname{sup} (1+\|hg\|)^s  \big\lVert\Xi(hg)^{-1} f(hg) \big\rVert.
\]

The integral 
\[
\Big\lVert \int _G \int_{\mathfrak{g}^*}\chi(h^{-1}) )e^{i\langle \xi, \exp^{-1}(h^{-1})\rangle} a(\xi) (1+\|h\|)^{s}d \xi dh \Big\rVert
\]
is finite as the integration of $h$ is over the support of $\chi$, which is compact. 

Summarizing the above estimates, we have proved that there is a constant $\widetilde{C}>0$ such that 
\[
\lVert |1+\|g\||^s \Xi(g)^{-1} \big(\operatorname{Op}(a)(f)(g) \big)\rVert\leq \widetilde{C} \operatorname{sup} (1+\|hg\|)^s  \big\lVert\Xi(hg)^{-1} f(hg) \big\rVert.
\]
This proves that $\operatorname{Op}(a)(f)$ belongs to the Lafforgue algebra $\mathcal{L}_s(G)$, and the map $(\operatorname{Op}(a), f)\mapsto \operatorname{Op}(a)(f)$ is continuous with respect to $f$. This proves the Lemma. 
\end{proof}

\subsection{$0$-degree (relative) cyclic cocycles associated to orbital integrals}
\label{section:0cocycle}
Let us now pass to cyclic cocycles associated to orbital integrals. Consider for the time being a compact $b$-manifold $Y$ endowed with a $b$-metric $\mathbf{h}_Y$ which
has product structure $\mathbf{h}_Y=dx^2/x^2 + \mathbf{h}_{\partial Y}$ near the boundary.
We use the associated volume form
in order to trivialize the relevant half-density bundles. In the $b$-calculus with $\epsilon$-bounds we consider
$$\Ker\left( {}^b \Psi^{-\infty,\epsilon}(Y)\xrightarrow{I_Y} {}^b \Psi^{-\infty,\epsilon}_{\RR}(\partial Y\times\RR)\right)$$ the kernel
of the (surjective) indicial homomorphism.

It is well known, see Melrose's book \cite{Melrose-Book}, that if $\epsilon < 1$, then 
$$\Ker(I_Y)
\subset \Psi^{-\infty,\epsilon}(Y)$$
with $ \Psi^{-\infty,\epsilon}(Y)$ denoting the {\it residual operators}; these are smoothing kernels on the $b$-stretched product $Y^2_b$ that vanish at order $\epsilon$
on {\it all}
boundary faces, that is, the left boundary $lb(Y^2_b)$, the right boundary $rb(Y^2_b)$ and the front face $bf(Y^2_b)$.
See the figure below. In fact $\Ker(I_Y)=\rho_{bf(Y^2_b)}{}^b \Psi^{-\infty,\epsilon}(Y)$ with 
$\rho_{bf(Y^2_b)}$ a boundary defining function for the front face, and by definition 
$\rho_{bf(Y^2_b)}{}^b \Psi^{-\infty,\epsilon}(Y)\subset \Psi^{-\infty,\epsilon}(Y)$  if $\epsilon < 1$.

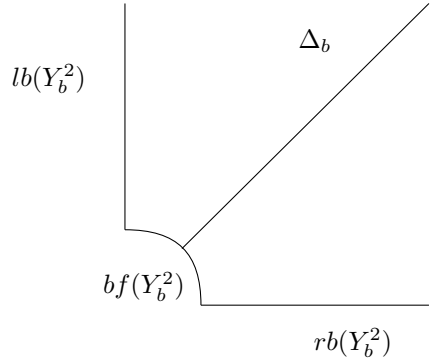
\begin{center}
\begin{tikzpicture}
		\node  (0) at (-6, 5) {};
		\node  (1) at (-6, 2) {};
		\node  (2) at (-5, 1) {};
		\node  (3) at (-2, 1) {};
		\node  (4) at (-7, 4) {$lb(Y^2_b)$};
		\node  (5) at (-3, 0.5) {$rb(Y^2_b)$};
		\node  (6) at (-5.25, 1.75) {};
		\node  (7) at (-2, 5) {};
		\node  (8) at (-3.5, 4.5) {$\Delta_b$};
		\node  (9) at (-5.75, 1.25) {$bf(Y^2_b)$};
	
		\draw (0.center) to (1.center);
		\draw [bend left=45, looseness=1.25] (1.center) to (2.center);
		\draw (2.center) to (3.center);
		\draw (6.center) to (7.center);
\end{tikzpicture}
\end{center}


Because of this vanishing, the residual  operators are trace class on $L^2_b$ and the trace
is obtained by integration over the lifted diagonal $\Delta_b$ (because of the extra vanishing the integral with respect to the $b$-volume
form, which near the boundary can be written as $\frac{dx}{x} dvol_{\partial X}$, is absolutely convergent).
Put it differently, this algebra of operators behaves very much as the smoothing operators on a smooth compact manifold
without boundary.

Consider now the $G$-proper case and 
\begin{equation}\label{true-ideal}
 \mathcal{L}^\infty _{G,s} (Y):=\Ker \left( I: {}^b \mathcal{L}^\infty _{G,s} (Y)\rightarrow {}^b \mathcal{L}^\infty _{G,s,\RR} (\cyl \pa Y_0) \right)\,.
 \end{equation}

For exactly the same reason as above this algebra behaves very much like $\left(\mathcal{L}_s(G)\hat{\otimes}\Psi^{-\infty}(S)\right)^{K\times K}$,
with $S$ now a slice for the action of $G$ on a cocompact $G$-proper manifold {\it without} boundary $X$.
This is in fact a great advantage of our method and allows us to avoid completely {\it ad hoc}
arguments.

We thus have a trace-homomorphism 
\begin{equation}\label{delocalized-trace-on-M}
\tau_g^Y: \mathcal{L}^\infty_{G,s} (Y)\to\CC\,.
\end{equation} 
exactly as in the closed case:
\begin{equation}\label{tr-g-b}
\tau_g^Y (T):= \int_{G/Z_g}\int_Y c(hgh^{-1}x) {\rm tr} (hgh^{-1}\kappa (hg^{-1}h^{-1}x,x))dx \,d(hZ)
\end{equation}
with $dx$ denoting now the $b$-volume form associated to the $b$-metric $\mathbf{h}$. 
The proof of the well-definedness of \eqref{tr-g-b} now proceeds as in the case of manifold without boundary, using the fact
that $\kappa$ is residual and therefore vanishing of order $\epsilon$ at all boundary faces.
In fact: there is a homomorphism of integration along the slice $S$, ${\rm Tr}_S: \mathcal{L}^\infty_{G,s} (Y,E)\to \mathcal{L}_s(G)$, which is continuous and such that 
$\tau^Y_g=\tau_g \circ {\rm Tr}_S\,.$

 The trace $\tau^Y_g$ defines a cyclic 0-cocycle
 on the algebra $ \mathcal{L}^\infty_{G,s} (Y)$. Using the pairing between $K$-theory and cyclic cohomology, denoted
 $\langle\,,\,\rangle$, we have in our case
 \begin{equation}\label{pairing}
 \langle\cdot\,,\cdot\,\rangle: HC^0 (\mathcal{L}^\infty_{G,s} (Y))\times K_0(\mathcal{L}^\infty_{G,s} (Y)) \to \CC
 \end{equation}
 and thus a homomorphism
 \begin{equation}\label{pairing-hom}
 \langle\tau^Y_g,\cdot\rangle: K_0(\mathcal{L}^\infty_{G,s} (Y))\to \CC.
 \end{equation}
\begin{definition}\label{def:g-index}
Let $D$ be a $G$-equivariant operator on $Y$ as in section \ref{sect:preliminaries}.
 Assume that the induced boundary operator is $L^2$-invertible, so that there is a well-defined
 index class $\Ind_\infty (D)\in K_0(\mathcal{L}^\infty_{G,s} (Y))$.
The $g$-index of $D$ is, by definition, the number 
 $\langle  \tau^Y_g, \Ind_\infty (D)\rangle$.
 \end{definition}

 \noindent
 {\it Our goal in this section  is to give a formula for $\langle \tau^Y_g,\Ind_\infty (D)\rangle$.}
 
  \medskip
 \noindent
 Following the relative cyclic cohomology approach in \cite{moriyoshi-piazza,GMPi,PP2} we want to find  a relative cyclic 0-cocycle $(\tau^{Y,r}_g,\sigma_g)$ verifying 
 \begin{equation}\label{basic-formula-bis}
 \langle\tau^Y_g,\Ind_\infty (D)\rangle=\langle(\tau^{Y,r}_g,\sigma_g), \Ind_\infty (D,D_\partial )\rangle.
 \end{equation}
(The $r$ on the right hand side stands for {\it regularized.})

\begin{proposition}\label{prop:1-eta-cocycle}
Let $X$ be a cocompact $G$-proper manifold without boundary.
Define the following 1-cochain on ${}^b \mathcal{L}^\infty_{G,s,\RR} ({\rm cyl}(X))$
\begin{equation}\label{1-eta}
\sigma^X_g (A_0,A_1)=\frac{i}{2\pi}\int_\RR\tau_g^X ( \partial_\lambda I(A_0,\lambda)\circ I(A_1, \lambda)) d\lambda\,,
\end{equation}
where the indicial family of $A\in {}^b \mathcal{L}^\infty_{G,s,\RR} ({\rm cyl}(X))$, denoted $I(A,\lambda)$, appears. Then $\sigma^X_g (\,,\,)$ is well-defined  and  a cyclic 1-cocycle.
\end{proposition}

\begin{proof}
Fourier transform identifies ${}^b \mathcal{L}^\infty_{G,s,\RR} ({\rm cyl}(X))$ with holomorphic families 
$$\{\RR\times i (-\epsilon,\epsilon)\ni\lambda \to \mathcal{L}^\infty_{G,s} (\partial X)\}$$
with values in the Fr\'echet algebra $\mathcal{L}^\infty_{G,s} (X)$, rapidly decreasing in ${\rm Re}\lambda$. (Recall that we do not
write $\epsilon$ in the notation, but elements in our algebras are built from $b$-operators of order $-\infty$ in the
calculus with $\epsilon$-bounds.)
It is then immediate that the integral is absolutely convergent and depends continuously
on $A_0$, $A_1$. The fact that it is a cyclic 1-cocycle follows 
from the tracial property of $\tau^X_g$ and integration by parts in $\lambda$.
\end{proof}

Let $Y$ be now a $b$-manifold and let  $\partial Y$ be its boundary.
Let $\tau^{Y,r}_g$ be the functional on ${}^b \mathcal{L}^c_{G} (Y)$:
$$\tau^{Y,r}_g (T):= 
 \int_{G/Z_g}\int^b_Y c(hgh^{-1}y) {\rm tr} (hgh^{-1}\kappa (hg^{-1}h^{-1}y,y))dy \,d(hZ)$$
where Melrose's $b$-integral has been used, $dy$ denotes the $b$-density associated to the $b$-metric $\mathbf{h}$ and where we recall
that  the cut-off function $c_0$ on $Y_0$ is extended constantly along
the cylinder to define $c$. This is the regularization of $\tau^Y_g$ on a $b$-manifold, for the time
being on kernels of $G$-compact support (we shall deal with the extension of $\tau^{Y,r}_g $
on all of ${}^b \mathcal{L}^\infty_{G,s} (Y)$ momentarily).
Observe that $$\tau^{Y,r}_g =\tau_g \circ {}^b {\rm Tr}_S$$ with ${}^b {\rm Tr}_S: {}^b \mathcal{L}^c_{G} (Y)\to C^\infty_c (G)$
denoting $b$-integration along the slice $S$. More precisely, as in the closed case, we have an isomorphism
$$ {}^b \mathcal{L}^c_G (Y)\cong \left\{\Phi:G\to {}^b \Psi^{-\infty,\epsilon}(S),~\mbox{smooth, compactly supported and }
~K\times K~\mbox{invariant}\right\}
$$
and ${}^b \Tr_S$ associates to $\Phi$ the function $G\ni \gamma\to {}^b \Tr (\Phi (\gamma))$.
The continuity of this map will be treated more generally in the proof of  Proposition \ref{prop:-0-relative-cocycle-bis} below.

\begin{proposition}\label{prop:-0-relative-cocycle}
The pair 
$(\tau^{Y,r}_g,\sigma^{\partial Y}_g)$ defines  a relative 0-cocycle for ${}^b \mathcal{L}^c_G (Y)\xrightarrow{I} {}^b \mathcal{L}^c_{G,\RR} ({\rm cyl}(\partial Y))$.
 \end{proposition}
 
 \begin{proof}
 As we are dealing with a 0-cochain, it suffices to show  that for the Hochschild $b$-differential of  $\tau^{Y,r}_g$ the following
formula holds:
$$(b\tau^{Y,r}_g) (A_0,A_1)= \sigma^{\partial Y}_g (I(A_0),I(A_1))$$
where we recall that
$$\sigma^{\partial Y}_g (I(A_0),I(A_1))=\frac{i}{2\pi}\int_\RR\tau_g^{\partial Y} ( \partial_\lambda I(A_0,\lambda)\circ I(A_1, \lambda)) d\lambda\,.
$$
The left hand side $(b\tau^{Y,r}_g) (A_0,A_1)$ is equal to $\tau^{Y,r}_g [A_0,A_1]$; if 
$A_i$, $i=0,1$,  corresponds to $\Phi_i:G\to {}^b \Psi^{-\infty}(S)$ then
$\tau^{Y,r}_g [A_0,A_1]$ is equal to
$$ \int_{G\slash Z_g}\int_G {}^b\Tr (\Phi_0 (h_1gh_1^{-1}h^{-1})\circ \Phi_1 (h) dhd(h_1Z_g)- \int_{G\slash Z_g}\int_G {}^b\Tr (\Phi_1 h_1gh_1^{-1} h^{-1})\circ \Phi_0 (h) dhd(h_1Z_g)\,.$$
Changing the order of integration and with a suitable change of coordinates in the second summand, using the unimodularity  of
$G$, we can rewrite the above expression as
$$ \int_{G\slash Z_g}\int_G {}^b\Tr [\Phi_0 (h_1gx^{-1} h_1^{-1}), \Phi_1 (h)] dh d(h_1Z_g)\,.$$
Now we can apply Melrose's formula for the $b$-trace of a commutator and get
$$\frac{i}{2\pi}\int_{G\slash Z_g}\int_G \int_\RR \Tr \left( \partial_\lambda I( \Phi_0 (h_1gh_1^{-1} h^{-1}),\lambda)\circ I(\Phi_1 (h),\lambda)
\right) d\lambda dhd(h_1Z_g)\,.$$
We can interchange the order of integration without problems here (rapid decay in $\lambda$ and compact support in $G$)
and so we conclude reverting to $I(A_0,\lambda)$ and $I(A_1,\lambda)$
that 
\begin{align*}\tau^{Y,r}_g [A_0,A_1] &
=\frac{i}{2\pi}\tau_g \int_\RR \Tr_{\partial S} ( \partial_\lambda I(A_0,\lambda)\circ I(A_1, \lambda)) d\lambda\\
& = \frac{i}{2\pi}\int_\RR\tau_g^{\partial Y} ( \partial_\lambda I(A_0,\lambda)\circ I(A_1, \lambda)) d\lambda
\end{align*}
This completes the proof.
 \end{proof}

 \begin{proposition}\label{prop:-0-relative-cocycle-bis}
The pair 
$(\tau^{Y,r}_g,\sigma^{\partial Y}_g)$ extends continuously to  a relative 0-cocycle for 
$${}^b \mathcal{L}^\infty_{G,s} (Y)\xrightarrow{I} {}^b \mathcal{L}^\infty_{G,s,\RR} ({\rm cyl}(\partial Y)).$$
Moreover, the following formula holds:
\begin{equation}\label{basic-formula-3} 
 \langle\tau^Y_g,\Ind_\infty (D)\rangle=\langle (\tau^{Y,r}_g,\sigma^{\partial Y}_g), \Ind_\infty (D,D_\partial )\rangle\,.
 \end{equation}
 \end{proposition}

\begin{proof}
We already know that $\sigma^{\partial Y}_g$ extends to a 1-cocycle on $\mathcal{L}^\infty_{G,s,\RR} ({\rm cyl}(\partial Y))$.
If we could show that $\tau^{Y,r}_g$ extends to  ${}^b \mathcal{L}^\infty_{G,s} (Y)$ then by density and continuity
we would get the first statement of the proposition. The second would then follow as usual
from the fact that 
$$\tau^{Y,r}_g |_{\mathcal{L}^\infty_{G,s} (Y)}= \tau^{Y}_g$$
given that on residual operators the $b$-integral equals the ordinary integral.\\ 

As in \cite{PP2} we have an isomorphism
$$ {}^b \mathcal{L}^\infty_{G,s} (Y)\cong \left\{\Phi:G\to {}^b\Psi^{-\infty,\epsilon}(S),~K\times K~\mbox{invariant and}~g\mapsto 
v_s (\| \Phi(g)\|_\alpha)\;\;\text{bounded}\right \}$$ where $\|\;\;\|_\alpha$ are now the $C^\infty$ seminorms of a smooth kernel on the $b$-stretched product.\\
We want to show that the map
\begin{equation}\label{continuity-b-trace}
{}^b \mathcal{L}^\infty_{G,s} (Y)\ni \Phi \longrightarrow {}^b\Tr (\Phi (\cdot))\in \mathcal{L}_s(G)\end{equation}
is continuous.
Here we can use a Proposition of Lesch-Moscovici-Pflaum \cite[Proposition 2.6]{LMP}, expressing the b-trace of a b-smoothing operator
on a compact $b$-manifold
in terms of the trace of two residual operators, see for example Proposition 7.1 in \cite{PP2}. Using this Proposition 
and the arguments in \cite[Lemma 7.8]{PP2}, directly inspired in turn on those of \cite{GMPi},
we prove the continuity of the map
$${}^b\Psi^{-\infty,\epsilon}(S)\ni T \longrightarrow {}^b \Tr (T)\in \CC$$
and thus of the map \eqref{continuity-b-trace}.
\end{proof}
 
 \subsection{The $0$-degree delocalized APS index theorem on $G$-proper manifolds}
 We now apply formula \eqref{basic-formula-3} to the index class associated to the Connes-Moscovici projector
  \begin{equation}
\label{cm-idempotent1}
V(D)=\left( \begin{array}{cc} e^{-D^- D^+} & e^{-\frac{1}{2}D^- D^+}
\left( \frac{I- e^{-D^- D^+}}{D^- D^+} \right) D^-\\
e^{-\frac{1}{2}D^+ D^-}D^+& I- e^{-D^+ D^-}
\end{array} \right)
\end{equation}

On the left hand-side of \eqref{basic-formula-3} we have $$\langle \tau^M_g,V^b (D)\rangle$$ with $V^b (D)$ the modified Connes-Moscovici projector. 
Indeed, recall that the Connes-Moscovici projector is simply the Connes-Skandalis projector
for the choice of parametrix 
\begin{equation*}
 Q
:= \frac{I-\exp(-\frac{1}{2} D^- D^+)}{D^- D^+} D^-
\end{equation*}
This produces the remainders
$I-Q D^+ = \exp(-\frac{1}{2} D^- D^+)$, $I-D^+ Q =  \exp(-\frac{1}{2} D^+ D^-)$
and $V^b (D)$ is obtained by improving this parametrix to a true parametrix in the $b$-calculus, that is,
an inverse modulo residual operators.\\

On the right hand-side we have the pairing of  the relative cocycle $(\tau^{Y,r}_g,\sigma^{\partial Y}_g)$ with the relative class
\begin{equation}\label{pre-wassermann-triple-1}
(V(D), e_1, q_t)
\,, \;\;t\in [1,+\infty]\,,\;\;\text{ with }
q_t:= \begin{cases} V(t D_{\cyl})
\;\;\quad\text{if}
\;\;\;t\in [1,+\infty)\\
e_1 \;\;\;\;\;\;\;\;\;\;\;\;\;\,\text{ if }
\;\;t=\infty
 \end{cases}
\end{equation}
and with $e_1:=\begin{pmatrix} 0&0\\0&1 \end{pmatrix}$.  
Let us concentrate first on the right hand-side. By definition of 
relative pairing we have:
\begin{equation}\label{relative-pairing}
\langle (\tau^{Y,r}_g,\sigma^{\partial Y}_g), (V(D), e_1, q_t)\rangle=
 \tau^{Y,r}_g (e^{-D^- D^+})- \tau^{Y,r}_g (e^{-D^+ D^-})
+ \int_1^\infty \sigma^{\partial Y}_g ([\dot{q}_t,q_t],q_t)dt\,.
\end{equation}
\begin{proposition}\label{prop:from-cocycle-to-eta}
 The term $\int_1^\infty \sigma^{\partial Y}_g ([\dot{q}_t,q_t],q_t)dt$, with 
 $q_t:= V(t D_{\cyl})$ is equal to
 $$
-\frac{1}{2} \left( \frac{1}{\sqrt{\pi}} \int_1^\infty \tau^{\partial Y}_g (D_{\partial Y} \exp (-tD^2_{\partial Y}) )\frac{dt}{\sqrt{t}}\right)
\,.$$
\end{proposition}

\begin{proof}
By definition
$$\sigma^{\partial Y}_g ([\dot{q}_t,q_t],q_t)dt=  \frac{i}{2\pi} \int_\RR  \tau^{\partial Y}_g (\partial_\lambda (I([\dot{q}_t,q_t],\lambda))\circ I(q_t,\lambda))d\lambda.$$
We can integrate by parts in $\lambda$ on the right hand side, given that the two terms are rapidly decreasing
in $\lambda$.
Thus, taking into account the multiplicative constants,  we want to prove that
\begin{equation}\label{reduction-to-eta}
 \int_\RR  \tau^{\partial Y}_g ( (I([\dot{q}_t,q_t],\lambda)\circ \partial_\lambda(I(q_t,\lambda))d\lambda
=
\frac{i\sqrt{\pi}}{\sqrt{t}}\, \tau^{\partial Y}_g (D_{\partial Y} \exp (-tD^2_{\partial Y}) )\,.
\end{equation}
We need to write down the indicial family of $q_t$. By definition
$$q_t= \left( \begin{array}{cc} e^{-t^2 D_{\cyl}^- D_{\cyl}^+} & e^{-\frac{t^2}{2}D_{\cyl}^- D_{\cyl}^+}
\left( \frac{I- e^{-t^2D_{\cyl}^- D_{\cyl}^+}}{t^2 D_{\cyl}^- D_{\cyl}^+} \right) t D_{\cyl}^-\\
e^{-\frac{t^2}{2}D_{\cyl}^+  D_{\cyl}^-} t D^+& I- e^{- t^2 D_{\cyl}^+ D_{\cyl}^-}
\end{array} \right)$$
and thus, denoting $D_{\partial Y}$ by $B$ we have 
$$I(q_t,\lambda)= \left( \begin{array}{cc} e^{-t^2 (\lambda^2 +B^2)} & e^{-\frac{t^2}{2}(\lambda^2 + B^2)}
\left( \frac{I- e^{-t^2 (\lambda^2 + B^2)}}{t^2 ( \lambda^2 + B^2)}\right) t (-i\lambda+B)\\
e^{-\frac{t^2}{2}(\lambda^2 +B^2)} t (i\lambda +B)& I- e^{- t^2 (\lambda^2 +B^2)}
\end{array} \right).$$
We set 
$$p(t,\lambda):= I(q_t,\lambda)\,.$$
We must compute 
$$\int_\RR \tau_g^{\partial Y} \left( \left( \partial_t p(t,\lambda)\circ p(t,\lambda)- p(t,\lambda)\circ  \partial_t p(t,\lambda)\right)\circ \partial_\lambda p(t,\lambda)\right) d\lambda$$
and show that it equals
$$\frac{i\sqrt{\pi}}{\sqrt{t}}\, \tau^{\partial Y}_g (D_{\partial Y} \exp (-tD^2_{\partial Y}) ).$$
This is a complicated computation; however, we remark that all operators appearing in the $2\times 2$ matrices
$$\partial_t p(t,\lambda)\,,\quad  p(t,\lambda)\quad  \text{and}\quad  \partial_\lambda p(t,\lambda)$$ are given by operators obtained by functional calculus
for $B$ and thus, in particular, they commute. This means that we can deal with this computation in a formal way and so,
with some help by {\it Mathematica}, formula \eqref{reduction-to-eta} does follow. We omit the elementary but lengthy details.
\end{proof}

\noindent
Thanks to this Proposition we have that 
\begin{equation}\label{relative-pairing-bis}
\langle (\tau^{Y,r}_g,\sigma^{\partial Y}_g), (V(D), e_1, q_t)\rangle=
 \tau^{Y,r}_g (e^{-D^- D^+})- \tau^{Y,r}_g (e^{-D^+ D^-})
-\frac{1}{2}  \int_1^\infty      \frac{1}{\sqrt{\pi}}\tau^{\partial Y}_g (D_{\partial Y} \exp (-tD^2_{\partial Y}) )\frac{dt}{\sqrt{t}}   \,.
\end{equation}
 As the last step we replace $D$ by  $sD$; in the equality
 $$ \langle\tau^Y_g,\Ind_\infty (D)\rangle=\langle (\tau^{Y,r}_g,\sigma^{\partial Y}_g), \Ind_\infty (D,D_{\partial Y} )\rangle$$
 the left hand side  $\langle \tau^Y_g, \Ind_\infty (D)\rangle$ remains unchanged
 whereas the right hand side 
  becomes
 $$\tau^{Y,r}_g (e^{-s^2 D^- D^+})- \tau^{Y,r}_g (e^{-s^2 D^+ D^-})-  \frac{1}{2}
 \int_s^\infty   \frac{1}{\sqrt{\pi}} \tau^{\partial Y}_g (D_{\partial Y} \exp (-tD^2_{\partial Y}) \frac{dt}{\sqrt{t}}  \,.$$
 Summarizing, for each $s>0$ we have 
  \begin{equation}\label{s-equality}
  \tau^{Y,r}_g (e^{-s^2 D^- D^+})- \tau^{Y,r}_g (e^{-s^2 D^+ D^-})= \langle \tau^Y_g, \Ind_\infty (D)\rangle +  \frac{1}{2}
 \int_s^\infty   \frac{1}{\sqrt{\pi}} \tau^{\partial Y}_g (D_{\partial Y} \exp (-tD^2_{\partial Y}) \frac{dt}{\sqrt{t}}  \,.
 \end{equation}
 Now we take the limit as $s\downarrow 0$. It is a general principle, explained in detail 
 in \cite{Melrose-Book}, that the Getzler rescaling applies to the heat kernel in the $b$-context; 
 needless to say, the geometry here is more complicated that in the case of a compact $b$-manifold endowed with
 a product $b$-metric.
Still, we shall prove in the next section the following proposition, where all structures are assumed to be product-like near
the boundary.
\begin{proposition}\label{prop:short-with-boundary}
Let $(Y_0,\mathbf{h}_0)$ be a cocompact $G$-proper manifold. Let $(Y,\mathbf{h})$ be the associated $b$-manifold.
Let $D_0$ be a G-equivariant Dirac operator defined in  (\ref{def Dirac}) and let $D$ be the associated $b$-differential operator.
Let $g$ be a semisimple element and let $Y_0^g$ the fixed point set of $g$. Then
the limit $\lim_{s\downarrow 0} \tau^{Y,r}_g (e^{-s^2 D^- D^+})- \tau^{Y,r}_g (e^{-s^2 D^+ D^-})$ exists and we have 
 $$\lim_{s\downarrow 0} \tau^{Y,r}_g (e^{-s^2 D^- D^+})- \tau^{Y,r}_g (e^{-s^2 D^+ D^-})= 
 \int_{Y_0^g} c^g_0 {\rm AS}_g (D_0)$$ 
 with $c^g_0 {\rm AS}_g(D_0)$ defined  in Equation (\ref{eq:X-geom-form}).
 \end{proposition}
 
 \smallskip
 \noindent
 Assuming the last Proposition we can infer  that the limit 
 $$ \lim_{s\downarrow 0} \frac{1}{2}
 \int_s^\infty   \frac{1}{\sqrt{\pi}} \tau^{\partial Y}_g (D_{\partial Y} \exp (-tD^2_{\partial Y}) \frac{dt}{\sqrt{t}}$$
 exists and  equals  
  $$\int_{Y_0^g} c^g {\rm AS}_g (D_0) - \langle \tau^Y_g, \Ind_\infty (D)\rangle\,.$$
 We conclude that we have proved the following 
 \begin{theorem}\label{theo:0-delocalized-aps} (0-degree delocalized APS)\\
 Let $G$ be connected, linear real reductive group. Let $g$ be a semisimple element. Let $Y_0$, $Y$, $D$, $D_{\partial Y}$ as above.
 Assume that $D_{\partial Y}$ is $L^2$-invertible.
 Then 
 $$\eta_g (D_{\partial Y}):=\frac{1}{\sqrt{\pi}} \int_0^\infty \tau^{\partial Y}_g (D_{\partial Y} \exp (-tD^2_{\partial Y}) \frac{dt}{\sqrt{t}}$$
 exists and for the pairing of the index class $\Ind_\infty(D)\in K_0(\mathcal{L}^\infty_{G,s} (Y))\equiv
 K_0 (C^*(Y_0\subset Y)^G)$ with the $0$-cocycle  $\tau^Y_g\in HC^0 ((\mathcal{L}^\infty_{G,s} (Y)))$ 
 the following delocalized 0-degree APS index formula holds:
 \[
 \langle \tau^Y_g,\Ind_\infty (D) \rangle= \int_{Y_0^g} c^g {\rm AS}_g (D_0) - \frac{1}{2} \eta_g (D_{\partial Y})\,,
 \]
where the integrand $c^g {\rm AS}_g(D_0)$ is defined in the same way as the one in Equation (\ref{eq:X-geom-form}).
 \end{theorem}
  
 \section{Delocalized eta invariants for $G$-proper manifolds}\label{sect:eta-in-general}
 
 In the previous section we have obtained the well-definedness  of $\eta_g (D_{\partial Y})$, with $D_{\partial Y}$ being $L^2$-invertible, as a byproduct of the proof
 of the delocalized APS index theorem for 0-degree cocycles. In fact, one can  show that  $\eta_g (D)$ is well defined
 on a cocompact
 $G$-proper manifold 
 even if $D$ does not arise  as a boundary operator.
 This is the content of the next 
 theorems, partially discussed also in \cite{HWW1,HWW2}.
The main result of this section is the following theorem. 
\begin{theorem}\label{thm etadefine}
Let $(X,\mathbf{h})$ be a cocompact $G$-proper manifold without boundary endowed with a $G$-equivariant Spin$^c$-structure and let $D$ be the Dirac-type operator defined in (\ref{def Dirac}). Let $g$ be a semi-simple element. If $D$ is $L^2$-invertible, then the integral 
\begin{equation}\label{thm:eta}
\frac{1}{\sqrt{\pi}} \int_0^\infty \tau^{X}_g (D\exp (-tD^2) \frac{dt}{\sqrt{t}}
\end{equation}
converges.
\end{theorem}

\begin{proof}We split the proof into two parts. 

In Proposition \ref{prop:eta large time integral}, we study the large time behavior of the heat kernel and prove the following integral
\[
\frac{1}{\sqrt{\pi}}\int_1^\infty \tau^X_g(D\exp(-tD^2))\frac{dt}{\sqrt{t}}
\] 
converges. 

In Proposition \ref{prop:eta-short time}, we study the small time behavior of the heat kernel and prove the following integral 
\[
\frac{1}{\sqrt{\pi}}\int^1_0  \tau^X_g(D\exp(-tD^2))\frac{dt}{\sqrt{t}}
\]
converges. 

We complete the proof of the theorem by combining the above two results. 
\end{proof}

Recall, following  \cite{PP2}, that for all $t>0$
$$\exp (-tD^2)=\frac{1}{2\pi i} \int_\gamma e^{-t\mu} (D^2-\mu)^{-1}d\mu$$
is an element in the Fr\'echet algebra $\mathcal{L}^\infty_{G, s} (X, E)$ with any $s > 0$. It was stated
in  \cite{PP2}, but without a proper proof,  that if $D$ is $L^2$-invertible then 
$\exp (-tD^2)$ converges exponentially  to $0$ in $\mathcal{L}^\infty_{G, s} (X, E)$. 
The corresponding statement for $b$-manifolds was also stated there.
The next two subsections provide detailed proofs of these two results.
 
 \subsection{Large time behaviour on manifolds without boundary}\label{subsect:large}
 \noindent 
Assume now that there exists $a>0$ such that
\begin{equation}\label{spectral-hyp}{\rm Spec}_{L^2} (D)\cap [-2a,2a]=\emptyset\,.
\end{equation}
We can and we shall choose $\gamma$ so that ${\rm Spec}_{L^2} (D)\cap \gamma(z)=\emptyset$ and ${\rm Re}\,\gamma (z)>a $ for every $z$.
We want to show that  $\exp (-tD^2)\in \mathcal{L}^\infty_{G, s} (X)$ is exponentially converging to $0$ in  $ \mathcal{L}^\infty_{G, s} (X)$ as $t\to +\infty$. \\

By \cite[Proposition 5.21]{PP2},
$$(D^2-\mu)^{-1}=B(\mu)+ C(\mu)$$
where
\begin{itemize}
\item $B(\mu)\in\Psi^{-2}_{G,c} (X)$ is a symbolic parametrix for $D^2-\mu$ and defines a pseudodifferential operator with parameter
of order $-2$;
\item $C(\mu) \in  \mathcal{L}^\infty_{G, s} (X)$, and $C(\mu)$ goes to  0 in  the Fr\'echet topology of $\mathcal{L}^\infty_{G, s} (X)$, 
as $|\mu|\to +\infty$.
\end{itemize}
In fact, we can improve this result and see easily that $C(\mu)$ is rapidly decreasing in $|\mu|$, with values
in $\mathcal{L}^\infty_{G, s} (X)$. Indeed, $(D^2-\mu)^{-1} = B(\mu)(1 + F(\mu))$ with $1+F(\mu)= (1+R(\mu))^{-1}$
and $R(\mu)$ the remainder of the symbolic parametrix $B(\mu)$. By the pseudodifferential calculus with parameter we know that 
$R(\mu)$ is rapidly decreasing in $|\mu|$ (it is a smoothing operator with parameter). Using the elementary identity
$$1- (1+R(\mu))^{-1}= R(\mu) [(1+R(\mu))^{-1}]$$
we understand that $F(\mu)$ is in fact rapidly decreasing in $|\mu|$ with values in  $\mathcal{L}^\infty_{G, s} (X)$. 
Thus $C(\mu):=  B(\mu)\circ F(\mu)$
is also rapidly decreasing in $|\mu|$, with values in $\mathcal{L}^\infty_{G, s} (X)$; here Lemma \ref{lemma:composition}
has been used.
By the same analysis carried out in \cite{PP2} and above we also know that 
for all $\mu$ and $k\geq 1$ 
$$(D^2-\mu)^{-k}=F(\mu)+G(\mu)$$
with
$F(\mu)\in\Psi^{-2k}_{G,c}(X,\Lambda)$, of $G$-compact support uniformly in $\mu$, and $G(\mu)\in \mathcal{L}^\infty_{G, s} (X)$
and rapidly decreasing in $\mathcal{L}^\infty_{G, s} (X)$ as $| \mu | \to +\infty$.\\ 
Let $p_{s,\alpha}$  be a seminorm  on the Fr\'echet algebra $\mathcal{L}^\infty_{G, s} (X):= 
(\mathcal{L}_{s}(G)\hat{\otimes}\Psi^{-\infty}(S))^{K\times K}$. 
This depends on $s$ through the Banach norm on $\mathcal{L}_s (G)$ and on the multiindex $\alpha$ 
through a seminorm $\| \cdot \|_\alpha$ on the 
smoothing operators on the slice. Let $|\alpha|=\ell$. Consider now $(D^2-\mu)^{-k}$ with $k>{\rm max}(\dim X,\ell/2)$.
Observe that
$$\exp (-tD^2)=\frac{(k-1)!}{2\pi i} t^{1-k}\int_\gamma e^{-t\mu} (D^2-\mu)^{-k}d\mu$$
We want to show that
$$p_{s,\alpha} \left( \int_{\gamma} e^{-t\mu}(D^2-\mu)^{-k}d\mu \right) \rightarrow 0 \quad \text{as} \quad t\to +\infty\,.$$ 
We thus consider the decomposition
$$(D^2-\mu)^{-k}=F(\mu)+G(\mu)$$
and  
$$ \int_{\gamma} e^{-t\mu} F(\mu)d\mu, \quad \int_{\gamma} e^{-t\mu} G(\mu)d\mu.  $$ 
Consider the first integral. Using the arguments employed  by Shubin  in \cite[Ch. XII, Sect. 11]{Shubin-Book}
in order to discuss the properties of complex powers, see also Vassout \cite[Ch. 4]{Vassout} and Gilkey\cite[Lemma 1.8.1]{gilkey-book}, we see that this first integral is in $\mathcal{L}^c_{G} (X)$, which is 
a subalgebra of $\mathcal{L}^\infty_{G,s} (X)$; the integrand, on the other hand, is in $\Psi^{-2k}_{G,c}(X,\Lambda)$ and to 
such an integrand we can apply the seminorm $p_{s,\alpha}$, given that it is contained in
$(\mathcal{L}_s (G)\hat{\otimes} C^{2k}(S\times S))^{K\times K}$.
From the properties of the Bochner integral we thus have that
$$p_{s,\alpha} \left( \int_{\gamma} e^{-t\mu} F(\mu))d\mu \right) \leq \int_{\gamma} e^{-t\mu} p_{s,\alpha} (F(\mu))d\mu
$$
Now, $F(\mu)$ is a pseudodifferential operator with parameter of order $(-2k)$ of $G$-compact support uniformly in $\mu$. 
We claim that  $p_{m,\alpha} (F(\mu))$ can be bounded by a negative power of $|\mu|$. To see this we take the associated element
$\Phi_{F(\cdot )}:G\to \Psi^{-2k}(S,\Lambda)$ under  the isomorphism \eqref{PHI}. 
We can equivalently show that the seminorm $p^{\Phi}_{s,\alpha}$ defined by $v_s$ and $\|\cdot\|_{\alpha}$ can be bounded by a negative power of $\mu$ when applied to the element $\Phi_{F(\cdot )}$.
Now, $\Phi_{F(\cdot )}$
is a  compactly supported function on $G$ with values 
in $\Psi^{-2k}(S,\Lambda)$. 
Because of the compactness of the support in $G$ we only need to understand  the statement 
for the elements in  $\Psi^{-2k}(S,\Lambda)$
for $k$ large; however this is well known, see for example \cite[Lemma 1.7.4]{gilkey-book}.
It is at this point clear that 
\begin{equation}\label{A(t)}  A(t):= \int_{\gamma} e^{-t\mu} p_{m,\alpha} (F(\mu))d\mu
\end{equation}
goes to $0$ as $t\to +\infty$. Indeed, as we can choose the path $\gamma$ so that ${\rm Re}\gamma(z)>a$ for all $z$,
 the result 
is immediate from the stated properties of $p_{m,\alpha} (F(\mu))$.\\
Consider now the second term, $$p_{s,\alpha}\left( \int_{\gamma} e^{-t\mu}G(\mu)d\mu \right)$$ 
Here the integrand is already an element in $\mathcal{L}^\infty_{G,s} (X)$, going to $0$ as $|\mu |$ goes to $+\infty$. In this case
we can bring the seminorm under the sign of integral.  Thus we are considering
\begin{equation}\label{B(t)}B(t):= \int_{\gamma} e^{-t\mu} p_{s,\alpha}(G(\mu))d\mu 
\end{equation}
Since we know that $p_{s,\alpha}(G(\mu))$ goes to $0$ as  $|\mu|$ goes to $+\infty$ (in fact, rapidly), using once again in a crucial way the fact that
 ${\rm Re}\gamma(z)>a$, we can conclude that 
$B(t)$ goes to $0$ as  $t\to +\infty$.\\
Notice that in fact, by elementary manipulations, writing $e^{-t\mu}$ as $e^{-t\mu/2}\circ e^{-t\mu/2}$
we can prove that the convergence is weighted exponential, with weight $a$ (and $a$ as in 
\eqref{spectral-hyp}). See  \cite[Remark 2.11]{PP2}.

\medskip
\noindent
We summarize our discussion in the following proposition:

\begin{proposition}\label{prop:large-yes-boundary-unperturbed}
If $D$ is $L^2$-invertible then
\begin{equation}\label{large-yes-boundary-unperturbed}
 \exp (-tD^2)\rightarrow 0 \quad\text{weighted exponentially in}\quad   \mathcal{L}^\infty_{G,s} (Y)\quad\text{as}\quad t\to +\infty.
 \end{equation}
 In general for a Schwartz function $f$ on $\mathbb{R}$, the  operators $f(tD^2)$and $Df(tD^2)$ both converge to 0 weighted exponentially in $\mathcal{L}^\infty_{G, s}(Y)$ as $t\to +\infty$. In particular, the Connes-Moscovici projector $V(tD)-e_1$ converges to 0 weighted exponentially in $\mathcal{L}^\infty_{G,s} (Y)$ as $t\to +\infty$.
 \end{proposition}
\begin{proof}
For the heat kernel we have already given a very detailed proof. For a general Schwartz function $f$ on $\mathbb{R}$, we can directly generalize the above analysis on $\exp(-tD^2)$ via a similar estimate for $D(D^2-\lambda)^{-2k}$ to show that $Df(-tD^2)$ converges to 0 weighted exponentially in $\mathcal{L}^\infty_{G,s} (Y)$ using the following integral formulas
\[
f(tD^2)=\frac{1}{2\pi i} \int_\gamma f(t\mu) (D^2-\mu)^{-1}d\mu,\ \ \ \ Df(tD^2)=\frac{1}{2\pi i} \int_\gamma f(t\mu) D(D^2-\mu)^{-1}d\mu. 
\] 
The property about the Connes-Moscovici projector follows from the observation that every component of the projector is of the form $f(tD^2)$ or $Df(tD^2)$ for some Schwartz function $f$ on $\mathbb{R}$. 
\end{proof}

\begin{proposition}\label{prop:eta large time integral}Let $(X,\mathbf{h})$ be a cocompact $G$-proper manifold without boundary endowed with a $G$-equivariant Spin$^c$-structure and let $D$ be the Dirac-type operator defined in (\ref{def Dirac}). Let $g$ be a semi-simple element. If $D$ is $L^2$-invertible, then the integral 
\begin{equation}\label{thm:eta-large t}
\frac{1}{\sqrt{\pi}} \int_1^\infty \tau^{X}_g (D\exp (-tD^2) \frac{dt}{\sqrt{t}}
\end{equation}
converges.
\end{proposition}
\begin{proof}
Consider the Schwartz function $f(x)=\exp^{-|x|}$ on $\mathbb{R}$. It follows from Proposition \ref{prop:large-yes-boundary-unperturbed} that the operator $Df(-tD^2)=D\exp(-tD^2)$ converges weighted exponentially in $\mathcal{L}^\infty_{G,s} (Y)\quad\text{as}\quad t\to +\infty$. Accordingly, the orbital integral $\tau^X_g(D\exp(-tD^2))$ converges to 0 weighted exponentially as $t\to +\infty$. This decay property assures that integral (\ref{thm:eta-large t}) is well defined. 
\end{proof}

\subsection{Large time behaviour on manifolds with boundary}
 \noindent We now discuss the case of manifolds with boundary. We consider as before  $(Y_0,\mathbf{h}_0)$, a cocompact $G$ proper manifold with boundary $X$ and denote 
by  $(Y,\mathbf{h})$  the associated $b$-manifold. We denote by $Z_0$ a slice for $Y_0$
and by $Z$ the associated $b$-manifold.
Let $D$ be a Dirac operator on $(Y,\mathbf{h})$ and assume that $D_\partial$ is $L^2$-invertible.
The same method as in \cite{PP2} can be directly generalized to establish that 
$$\exp (-tD^2)=\frac{1}{2\pi i} \int_\gamma e^{-t\mu} (D^2-\mu)^{-1}d\mu$$
is an element in ${}^b \mathcal{L}^\infty_{G,s} (Y)$ for every $t>0$ and for every $s>0$.\\

Assume now that there exists $a>0$ such that
$${\rm Spec}_{L^2} (D)\cap [-2a,2a]=\emptyset\,.$$
We can and we shall choose $\gamma$ so that ${\rm Spec}_{L^2} (D)\cap \gamma(z)=\emptyset$ and ${\rm Re}\,\gamma (z)>a $ for every $z$.\\

We want to show that  $\exp (-tD^2)\in {}^b \mathcal{L}^\infty_{G,s} (Y)$ is exponentially converging to $0$ in  $  {}^b\mathcal{L}^\infty_{G,s} (Y)$ as $t\to +\infty$. 
By \cite[Proposition 5.21]{PP2},
$$(D^2-\mu)^{-1}=B(\mu)+B(\mu)\circ L(\mu)$$
with 
$$B(\mu)=B^\sigma (\mu) - \varphi ((I(D^2-\mu)^{-1} I(R^\sigma (\mu)))$$
 In these formulae
\begin{itemize}
\item 
$\varphi: {}^b \mathcal{L}^\infty_{G, s, \RR} (\cyl (\partial Y))) \to {}^b \mathcal{L}^\infty_{G, s} (Y)$ is a section of the indicial homomorphism 
$I : {}^b \mathcal{L}^\infty_{G,s} (Y)\to {}^b \mathcal{L}^\infty_{G, s, \RR} (\cyl (\partial Y))$; $\varphi$ is defined by using a suitable cut-off function equal to 1 
on the boundary;
\item $B^\sigma (\mu)\in{}^b \Psi^{-2}_{G,c} (X)$ is a symbolic parametrix for $D^2-\mu$ and defines a $b$-pseudodifferential operator with parameter
of order $-2$ (in the small $b$-calculus);
\item $R^\sigma (\mu)\in{}^b \Psi^{-\infty}_{G,c} (X)$ is the remainder of a symbolic parametrix and defines a $b$-pseudodifferential operator with parameter
of order $-\infty$ (always in the small calculus);
\item $B(\mu)$ is a true parametrix for $D^2$ in the calculus with bounds.
\end{itemize}
Moreover, as proved in \cite{PP2},  $\mu\to   \varphi ((I(D^2-\mu)^{-1} I(R^\sigma (\mu)))$ is rapidly decreasing when
${\rm Re} (\mu) \to +\infty$ as a map with values in  ${}^b \mathcal{L}^\infty_{G,s} (Y)$  
and $L(\mu)\in {}^b \mathcal{L}^\infty_{G,s} (Y)$ goes to  0, always in Fr\'echet topology, as ${\rm Re}(\mu)\to +\infty$.\\
A similar formula can be written for $(D^2-\mu)^{-k}$ for any $k\geq 1$. We consider, as for the case of closed manifolds, $k$ very large; we  fix such a $k$ and adopt the same notation as above.
Write now 
$$\exp(-tD^2)=\frac{(k-1) !}{2\pi i} t^{1-k}  \int_{\gamma} e^{-t\mu} (D^2-\mu)^{-k}d\mu.$$
We can express  the integral on the right hand side initially as the sum of 2 terms
$$\int_\gamma e^{-t\mu} B_\mu   d\mu\;+\; \int_\gamma e^{-t\mu} B_\mu  \circ L_\mu  d\mu$$
and then, using the analogue of the above expression for $(D^2-\mu)^{-k}$, as the sum of 4 terms:
\begin{equation}
\begin{split}
\int_\gamma e^{-t\mu}  B^\sigma(\mu) d\mu - \int_\gamma e^{-t\mu}  \varphi (I(D^2-\mu)^{-k} I(R^\sigma (\mu)))d\mu
+ \int_\gamma e^{-t\mu}B^\sigma (\mu)\circ L (\mu) d\mu-\\  \int_\gamma e^{-t\mu} 
\varphi (I(D^2-\mu)^{-k} I(R^\sigma (\mu)))\circ L(\mu) d\mu
\end{split}
\end{equation}
where now $B^\sigma (\mu)\in {}^b \Psi^{-2k}_{G,c} (X)$ and defines a  $b$-pseudodifferential operator with parameter of order $-2k$.
The large-time behaviour of these four integrals can be treated as in the closed case. Indeed, for the first term 
we know that the Schwartz kernel of $B^\sigma (\mu)$ will be $C^{2k}$ across the $b$-diagonal and so, keeping in mind
that this term if of order $(-2k)$ as a pseudodifferential operator with parameter, we can directly
estimate the integral of the 
 relevant seminorm applied to $B^\sigma (\mu)$, as in the closed case, c.f. \eqref{A(t)}.
For the second summand it suffices to recall 
that $\mu\to \varphi (I(D^2-\mu)^{-1} I(R^\sigma_\mu))$ is a rapidly decreasing ${}^b \mathcal{L}^{\infty}_{G, s} (Y)$-valued map.
Similarly, the operators in the third integral define a ${}^b \mathcal{L}^{\infty}_{G, s} (Y)$-valued map going to $0$ in the Fr\'echet topology
as ${\rm Re}(\mu)\to +\infty$, whereas the operators in the fourth integral define a ${}^b \mathcal{L}^{\infty}_{G, s} (Y)$-valued map
rapidly decreasing 
as ${\rm Re}(\mu)\to +\infty$.
When we apply the seminorms to these families of operators inside the integral and we perform the integration,
we obtain expressions in $t$ 
that clearly go to 0 when $t\to +\infty$;
here we use again, crucially, that  ${\rm Re} \gamma(z)>a>0$.  As this is precisely as in \eqref{B(t)} we omit the details. 
 In fact, by a simple trick with the countour,
we see also in this case that the convergence is weighted exponential.

\medskip
\noindent
We summarize our discussion in the following proposition:

\begin{proposition}\label{prop:large-yes-boundary-unperturbed-b}
If $D$ is $L^2$-invertible then
\begin{equation}\label{large-yes-boundary-unperturbed-1}
 \exp (-tD^2)\rightarrow 0 \quad\text{weighted exponentially in}\quad   {}^b \mathcal{L}^\infty_{G,s} (Y)\quad\text{as}\quad t\to +\infty.
 \end{equation}
 In general for a Schwartz function $f$ on $\mathbb{R}$, the operators $f(tD^2)$ and $Df(tD^2)$ converge to 0 weighted exponentially in ${}^b\mathcal{L}^\infty_{G, s} (Y)$ as $t\to +\infty$. In particular the Connes-Moscovici projector $V(tD)-e_1$ converges to 0 weighted exponentially in ${}^b\mathcal{L}^\infty_{G, s} (Y)$ as $t\to +\infty$.
\end{proposition}

\subsection{Small time behavior}\label{subsec:small time}
To study the small $t$ convergence of the eta integral, we will need the following properties. 
\begin{lemma}\label{lem:action property} Let $X$ be a cocompact $G$-proper manifold without boundary with a  $G$-invariant complete Riemannian metric $\mathbf{h} _X$. Let $d_X(-,-)$ be the associated distance function, and $d_G(-,-)$ be the distance function on $G/K$. Suppose that $V$ is a compact set on $X$ and $g$ is any element in $G$. There are  constants $C_0>0, C_1>0$ such that 
\begin{equation}\label{inequality-1}
d_X(hgh^{-1}x, x)\geq C_1 d_G(hgh^{-1}e,e)-C_0,\ \forall x\in V, h\in G .\quad 
\end{equation}
\end{lemma}
\begin{proof}
The slice theorem gives a fibration structure on $X=(G\times S)/K$ over $G/K$, i.e. $\pi: X\to G/K$. For every $x\in X$, let $T_x X$ be the tangent space of $X$ at $x$, and $V_x$ be the kernel of the map $\pi_*: T_x X\to T_{\pi(x)}G/K$, and $H_x$ be the orthogonal complement of $V_x$ in $T_x X$ with respect to the metric $\mathbf{h}_X$. The induced map $\pi_*: H_x\to T_{\pi(x)} G/K$ is a linear isomorphism. Though the map $\pi_*: H_x \to T_{\pi(x)}G/K$ might not be an isometry with respect to the restricted metric $\mathbf{h}_X|_{H_x}$ and the metric $\mathbf{h}_{G/K}$ on $T_{\pi(x)}G/K$, there is a positive constant $c_x$ such that 
\[
\mathbf{h}_{X}|_{H_x} \geq c_x (\pi_*)^{-1}\big(\mathbf{h}_{G/K} \big). 
\]
Running $x$ over every point in the fiber $\pi^{-1}(\pi(x))$, as the fiber is compact, there is a constant $C_1>0$ such that 
\[
\mathbf{h}_{X}|_{H_y} \geq C_1 (\pi_*)^{-1}\big(\mathbf{h}_{G/K} \big), \ \forall  y\in \pi^{-1}(\pi(x)). 
\]
As $G$ acts on $G/K$ transitively and the metric $\mathbf{h}_X$ and $\mathbf{h}_{G/K}$ are $G$-invariant, the above estimate holds for all $x\in X$, i.e. 
\[
\mathbf{h}_{X}|_{H_x} \geq C_1 (\pi_*)^{-1}\big(\mathbf{h}_{G/K} \big), \ \forall  x\in X.  
\]
The above comparison of the Riemannian metrics gives the following property on the distance functions by the standard argument on path lengths 
\[
d_X(hgh^{-1}x, x)\geq C_1d_{G}(hgh^{-1}\pi(x), \pi(x)). 
\]
Choose $x_0$ to be a point in $\pi^{-1}([e])$, where $[e]$ the point in $G/K$ associated to the coset of the identity element. The above inequality shows
\begin{equation}\label{eq:inequalitydX}
d_X(hgh^{-1}x_0, x_0)\geq C_1 d_G(hgh^{-1}e, e).
\end{equation}
We apply the triangle inequality to obtain
\begin{equation}\label{eq:distance}
d_X(hgh^{-1}x_0, hgh^{-1}x)+d_X(hgh^{-1}x, x)+d_X(x,x_0)\geq d_X(hgh^{-1}x_0,x_0).
\end{equation}
As $V$ is compact, there is a finite upper bound $C/2$ such that $d(x,x_0)\leq C/2$ for all $x\in V$. Combining this with Inequalities (\ref{eq:distance}) and (\ref{eq:inequalitydX}), we conclude with the desired inequality, i.e.
\[
d_X(hgh^{-1}x,x)\geq C_1d_G(hgh^{-1}e, e)-C_0. 
\]
\end{proof}

\begin{remark}\label{remark:estimate-also-b-1}
It is obvious that Lemma \ref{lem:action property} also holds if instead of $(X,\mathbf{h})$, $V\subset X$ and, for (ii),  $C(g)$ we consider:\\
$\quad$ $-$ $(Y_0,\mathbf{h}_0)$, a cocompact $G$-proper manifold with boundary, with 
slice-compatible metric, product-type\\$\quad$  near the boundary;\\
$\quad$ $-$ $V_0\subset Y_0$ a compact set in $Y_0$.

\smallskip
\noindent
We also consider $(Y,\mathbf{h})$, the $G$-manifold with cylindrical ends associated to $(Y_0,\mathbf{h}_0)$.
Here we recall that there is a collar neighbourhood of $\partial Y_0$ in $Y_0$ where
the action of $G$ is of product type (and thus extendable to a $G$-proper action on $(Y,\mathbf{h})$). Notice
that, consequently, the action of $G$ preserves the decomposition $Y:= (-\infty, 0] \times \partial Y_{0} \cup_{\partial Y_{0}} Y_0$.
Consider the compact subset $V_0\subset Y_0$ and $V_0\cap \partial Y_0$, denoted $\partial V_0$ if non-empty. Consider
finally $V:= (-\infty, 0] \times \partial V_{0} \cup_{\partial V_{0}} V_0$. 
Lemma \ref{lem:action property}  also holds if we consider now $(Y,\mathbf{h})$, $V$.
Indeed \eqref{inequality-1} does hold if $x\in V_0\subset Y_0$, as we have just observed, and also holds 
if $x\in  (-\infty, 0] \times \partial V_{0}\subset V\subset Y$ because the action of $G$ is of product type along the cylinder.
\end{remark}

The following property follows from \cite[Proposition 4.2, (i)]{CareyGRS} because a cocompact $G$-proper manifold without boundary has bounded geometry.
\begin{lemma}\label{lem:Cheng-Li-Yau}
 Let $X$ be a cocompact $G$-proper manifold without boundary and let $D$ be a Dirac-type operator (associated to a unitary
 Clifford action and a Clifford connection) on a spinor bundle $\mathcal{E}$. Let $\kappa_t(x,y)$ be the kernel function of the operator $D\exp(-tD^2)$. There are constant $\alpha,\beta>0$ such that
 \[
 \vert \vert \kappa_t(x,y)\vert\vert\leq \beta t^{-\frac{n+1}{2}} \exp\left(-\alpha \frac{d_X(x,y)^2}{t}\right).
 \]
where $\vert\vert \cdot \vert\vert$ is the operator norm from $\mathcal{E}_x$ to $\mathcal{E}_y$. 
\end{lemma}

\begin{remark}\label{remark:estimate-also-b-2}
Consider  the heat kernel $k_t$ associated to a Dirac-Laplacian  on  $(Y,\mathbf{h})$, the $G$-proper manifold with cylindrical ends associated to
a cocompact $G$-proper manifold with boundary $(Y_0,\mathbf{h}_0)$.
As $(Y,\mathbf{h})$ is a complete manifold with bounded geometry we can apply again  \cite[Proposition 4.2, (i)]{CareyGRS}
and conclude that 
  there are constant $\alpha_0, \beta_0>0$ such that 
\begin{equation}\label{estimate-carey-Y}
||k_t(x,y)||\leq \beta_0 t^{-\frac{n}{2}}\exp\left(-\alpha \frac{d_Y (x,y)^2}{t}\right). 
\end{equation}
\end{remark}

The following estimate is proved by Harish-Chandra  \cite[Theorem 6]{Harish-Chandra-dis} (see also \cite[Lemma 4.4]{Hochs-Wang-HC}).
\begin{lemma}\label{lem:integral} For a semisimple element $g$, when $t$ is close to 0, the integral 
\[
\int_{G/Z_{g}}\exp\left(-\alpha \frac{d_G(hg^{-1}h^{-1}e, e)^2}{t}\right)
\]
is bounded. We will sometimes use $d_G(g)$ for $d_G(ge,e)$  in the following of this article for abbreviation. 
\end{lemma}

\begin{lemma}\label{lem:compact-quotient}
For a semisimple $g\in G$, the quotient $X^g/Z_g$ is compact.
\end{lemma}

\begin{proof}
If $g$ is not conjugate to an element in $K$ (that is, $g$ is not elliptic) then $X^g$ is empty (See Proposition \ref{prop:fixed point}). 
It suffices to work with $g=k_0\in K$.  By the slice theorem, we assume that $X=G\times_K S$. A point
$[(h,x)]$ belongs to $X^{k_0}$ if there is $k\in K$ such that $k_0h=hk^{-1}$ and $kx=x$. Therefore, 
\[
X^{k_0}=\{[(h, x)]\in G\times _K S| h^{-1}k_0^{-1}h=k, kx=x\}. 
\]
Let $\pi$ be the quotient map from $G\times S$ to $G\times_K S$. The space $\tilde{X}^{k_0}:=\{(h,x)\in G\times S| h^{-1}k_0^{-1}h=k, kx=x\}$ as a subset of $G\times S$ is the preimage $\pi^{-1}(X^{k_0})$. As $\pi$ is a principal $K$-bundle, $\tilde{X}^{k_0}$ is a manifold. As the $Z_{k_0}$ action on $X\times G$ commutes with the $K$ action, it is sufficient to prove that $Z_{k_0}\backslash \tilde{X}^{k_0}$ is compact to conclude that $Z_{k_0}\backslash X^{k_0}$ is compact.

Consider the right conjugacy action of $G$ on the conjugacy class $C(k_0^{-1})=\{hk_0^{-1}h^{-1}|h\in G\}$. $C(k_0^{-1})$ is a closed submanifold of $G$, and the $G$ action on $C(k_0^{-1})$ is transitive with the stabilizer  group at $k_0^{-1}$ being $Z_{k_0}$. We have that $C(k_0^{-1})$ is diffeomorphic to $Z_{k_0}\backslash G$ and the map $f: G\to C(k_0^{-1})$ mapping $g$ to $h^{-1}k_0^{-1}h$ is a fibration.

Consider the intersection $C(k_0, K):=C(k_0^{-1})\cap K$, which is a compact subset. Let $F(k_0)$ be $f^{-1}\big(C(k_0^{-1})\cap K\big)$. By the fibration property of $f$, $Z_{k_0}\backslash F(k_0)$ is homeomorphic to $C(k_0, K):=C(k_0^{-1})\cap K$. For $C(k_0, K)$, we consider the closed space $I_{k_0, K}:=\{ (k, x)| k\in C(k_0, K), k_0x=x\}$ of $K\times S$. The map $\tilde{f}: I_{k_0, K}\to C(k_0, K)$ mapping $(k,x)\in I_{k_0, K}$ to $k\in C(k_0, K)$ is a continuous surjective map. 

Let $F(k_0)_{f}\times_{\tilde{f}} I_{k_0, K} $ be the fiber product of $F(k_0)$ and $I_{k_0, K}$ over the maps $f|_{F(k_0)}$ and $\tilde{f}$. As $f$ is a fibration, $F(k_0)_f\times _{\tilde{f}}I_{k_0, K}$ can be identified with $\tilde{X}^{k_0}=\{(g,x)| g\in C(k_0, K), x\in S^{f(g)}\}$. 

With the above identification, the quotient of $\tilde{X}^{k_0}$ by $Z_{k_0}$ is homeomorphic to 
\[
Z_{k_0}\backslash \left(F(k_0)_f \times _{\tilde{f}} I_{k_0, K}\right)=\left( Z_{k_0}\backslash F(k_0)\right)_f\times _{\tilde{f}}I_{k_0, K},
\]
which is homeomorphic to $I_{k_0, K}$ as $Z_{k_0}\backslash F(k_0)$ is homeomorphic to $C(k_0, K)$. As a closed subspace of $K\times S$, $I_{k_0, K}$ is compact. We conclude that $Z_{k_0}\backslash \tilde{X}^{k_0}$ is compact. 
\end{proof}

For an elliptic element $g$, the fixed point set could be non-empty. To tackle this situation, let $c_G$ be a cut-off function associated to the right $Z_{g}$ action on $G$. Following \cite[Lemma 3.2]{Hochs-Wang-HC}, we have the following equivalent expression  for $\tau^X_{g}(D\exp(-tD^2))$,  

\begin{equation}\label{eq:eta integral expression}
\begin{split}
\tau^X_{g}(D\exp(-tD^2))&=\int_{G/Z_{g}} \int_X c(y){\rm tr}(\kappa_t(y, hgh^{-1}y)hgh^{-1})dy dh(Z_{g})\\
&=\int_{G}c_G(h)\int_X c(y){\rm tr}(\kappa_t(y, hgh^{-1}y)hgh^{-1})dy dh\\
&=\int_X\int_G c_G(h)c(y){\rm tr}(h\kappa_t(h^{-1}y, g h^{-1}y) g h^{-1})dhdy\\
&=\int_X\int_G c_G(h)c(hy) {\rm tr}(\kappa_t(y, gy)g)dhdy. 
\end{split}
\end{equation}

Recall that $V$ is the support of the cut-off function $c$. We apply  Lemma \ref{lem:action property} to $V$ and conclude that there is a constant $C_0>0$ such that 
\[
d_X(hgh^{-1}x, x)\geq C_1d_G(hgh^{-1}e, e)-C_0, \forall x\in V, h\in G. 
\]

We introduce two subsets of the conjugacy class $C(g)$, i.e.
\[
K(g):=\{g\in C(g) | d_G(ge,e)\leq 2C_0/C_1 \},\ K^c(g):=\{g\in C(g)| d_G(ge, e)>2C_0/C_1 \}. 
\]
Define a map $\chi: G\to C(g)$ by $\chi(h):=hgh^{-1}$. Let $U(g)$ be the subset of $G$ defined to be the preimage of $K(g)$. And  let $U^c(g)$  be the complement of $U(g)$ in $G$.  As $g$ is semisimple, $C(g)$ is a closed subset of $G$, and therefore the image of $K(g)$ in $G/K$ is a bounded closed subset. Therefore, the image of $K(g)$ in $G/K$ is compact, and $K(g)$ is compact. Observe that $U(g)$ is invariant under the right $Z_g$ action and the quotient $U(g)/Z_g$ is diffeomorphic to $K(g)$. So $U(g)/Z_g$ is  also compact. We split the integral (\ref{eq:eta integral expression}) into two parts,
\[
\begin{split}
\tau^X_{g}(D\exp(-tD^2))&=\int_X\int_G c_G(h)c(hy) {\rm tr}(\kappa_t(y, gy)g)dhdy\\
&=\int_X \int_{U(g)}  c_G(h)c(hy){\rm tr}(\kappa_t(y, gy)g)dhdy+\int_X \int_{U^c(g)} c_G(h) c(hy){\rm tr}(\kappa_t(y, gy)) dhdy.
\end{split}
\]

\begin{lemma}\label{lem:uc} Let $\widetilde{W}$ be a closed subset of $X$. 
The integral 
\begin{equation}\label{eq:integral-uc}
I^t_{\widetilde{W}, U^c}(g):=\int_{\widetilde{W}} \int_{U^c(g)} c_G(h) c(hy){\rm tr}(\kappa_t(y, gy)g)dhdy
\end{equation}
is of exponential decay as $t\to 0$. 
\end{lemma}

\begin{proof}
It is sufficient to prove the property for $\widetilde{W}=X$. We rewrite the integral (\ref{eq:integral-uc}) by applying the change of variable $x=hy$.
\[
\begin{split}
I^t_{U^c}(g)&:=\int_X \int_{U^c(g)} c_G(h)c(x){\rm tr}(\kappa_t(x, hgh^{-1}x)hgh^{-1})dhdx\\
&=\int_{U^c(g)} c_G(h) \int_V c(x) {\rm tr}(\kappa_t(x, hgh^{-1}x)hgh^{-1})dhdx.
\end{split}
\]

By Lemma \ref{lem:Cheng-Li-Yau}, we have the following estimate for $|{\rm tr}(\kappa_t(x, hgh^{-1}x)hgh^{-1})|$, 
\[
\begin{split}
\vert\vert\kappa_t(x, hgh^{-1}x)\vert\vert&\leq \beta t^{-\frac{n+1}{2}}\exp\left(-\alpha \frac{d_X^2(x, hgh^{-1}x)}{t}\right),\\
|{\rm tr}(\kappa_t(x, hgh^{-1}x)hgh^{-1})|&\leq \tilde{\beta}t^{-\frac{n+1}{2}}\exp\left(-\alpha \frac{d_X^2(x, hgh ^{-1}x)}{t}\right),\qquad \forall x\in X. 
\end{split}
\]
Applying Lemma \ref{lem:action property}, we get
\[
d_X(x, hgh^{-1}x)\geq C_1d_G(hgh^{-1}e, e)-C_0. 
\]
By the definition of $U^c(g)$, for $h\in U^c(g)$, we have
\[
d_X(x, hgh^{-1}x)\geq C_0,\ d_X(x,hgh^{-1}x)\geq \frac{C_1}{2}d_G(hgh ^{-1}e,e),
\] 
and accordingly
\[
d_X(x, hgh^{-1}x)\geq \frac{C_0}{2}+\frac{C_1d_G(hgh^{-1}e, e)}{4},\ d_X(x, hgh^{-1}x)^2\geq \frac{C^2_0}{4}+\frac{C^2_1d_G(hgh^{-1}e,e)^2}{16}.
\]
Hence, we have the estimate
\[
|{\rm tr}(\kappa_t(x, hgh^{-1}x)hgh^{-1})|\leq \tilde{\beta}t^{-\frac{n+1}{2}}\exp\left(-\frac{\alpha C_0^2}{4t}\right)\exp\left(-\alpha C_1^2\frac{d_G^2(hgh^{-1}e,e)}{16t}\right).
\]

As $V$ is compact, the integral 
\[
\int_V c(x)dx
\]
is finite. And $I^t_{U^c}(g)$ can be bounded by
\[
\begin{split}
|I^t_{U^c}(g)|\leq &\gamma t^{-\frac{n+1}{2}}\exp\left(-\frac{\alpha C_0^2}{4t}\right)\int_{U^{c}(g)}c(g) \exp\left(-\alpha C_1^2\frac{d^2_G(hgh^{-1}e,e)}{16t}\right)dh\\
\leq &\gamma t^{-\frac{n+1}{2}}\exp\left(-\frac{\alpha C_0^2}{4t}\right) \int_{G/Z_{g}} \left(-\alpha C_1^2 \frac{d^2_G(hgh^{-1}e,e)}{16t}\right)d(hZ_{g}).
\end{split}
\]
This implies the exponential decay property of $I^t_{U^c}(g)$. 
\end{proof}

Let $W$ be an open subset of $X$ containing the $g$ fixed point submanifold $X^{g}$. 
\begin{lemma}\label{lem:decay-outside w}
The integral 
\[
I^t_{W^c, U}(g):=\int_{W^c} \int_{U(g)}  c_G(h)c(hy){\rm tr}(\kappa_t(y, gy)g)dh dy
\]
converges exponentially to 0 as $t\to 0$. 
\end{lemma}
\begin{proof}
Because $U(g)/Z_g$ is compact and $c_G$ is the cut-off function for the $Z_g$-action on $G$, we know that $c_G$ is compactly supported on $U(g)$. 
Moreover, since $c$ is compactly supported, the properness of the $G$ action on $X$ implies that the function 
\[
\tilde{c}(y):=\int_{U(g)} c_G(h)c(hy)dh
\]
is also compactly supported. Let $\tilde{V}$ be the support of $\tilde{c}$. And $I^t_{W^c, U}(g)$  
can be expressed as 
\[
\int_{W^c\cap \tilde{V}} \tilde{c}(y){\rm tr}(\kappa_t(y, gy)gdy. 
\]

Using Lemma \ref{lem:Cheng-Li-Yau}, we have the estimate 
\[
|{\rm tr}(\kappa_t(y, gy)g)|\leq \tilde{\beta}t^{-\frac{n+1}{2}}\exp\left(-\alpha \frac{d^2_X(y, gy)}{t}\right) \forall y\in X. 
\]

As $W^c\cap \tilde{V}$ is compact and does not contain any $g$ fixed point, there is a positive $\epsilon_0$ such that 
\[
d_X(y, gy)\geq \epsilon_0,\ \forall y\in W^c \cap \tilde{V}.
\] 
We can bound $I^t_{W^c, U}(g)$ using the above estimate about $d(y, gy)$,
\[
|I^t_{W^c, U}(g)|\leq \tilde{\beta}t^{-\frac{n+1}{2}}\exp\left(-\alpha \frac{\epsilon_0^2}{t}\right) \int _{W^c\cap \tilde{V}} |\tilde{c}(y)|dy.
\]
This gives the exponential decay property of $I^t_{W^c, U}(g)$.
\end{proof}

We can summarize  Lemma \ref{lem:uc} and \ref{lem:decay-outside w} in the following result. 
\begin{proposition}\label{prop:eta-outside fixed point}
Given any open set $W$ containing the $g$ fixed point submanifold $X^{g}$, then the integral 
\[
I^t_{W^c}(g):=\int_{W^c} \int_G c_G(h)c(hy) {\rm tr}(\kappa_t(y, gy)g)dh dy =I^t_{W^c, U}(g)+I^t_{W^c, U^c}(g). 
\]
decay exponentially to zero at $t\to 0$. When $W$ is assumed to be $g$ invariant, $I^t_{W^c}(g)$ has the following expression under change of coordinates
\[
I^t_{W^c}(g)=\int_{W^c}\int_G c_G(h)c(hgx) {\rm tr}(\kappa_t(x, gx)g)dhdx. 
\]
\end{proposition}

Let $W$ be a sufficiently small neighborhood of $X^{g}$. Using the $Z_{g}$-invariant tubular neighborhood theorem, we can identify $W$ to be an $r$-ball bundle $N(r)$ in the normal bundle $N$ of $X^g$ in $X$ with a sufficiently small radius $r$. Define $c_{g}$, a smooth function on $X$, by
\[
c_{g}(x):=\int_G c_G(h)c(hg x)dh.
\]

\begin{lemma}
\label{lem-support}
The function $c_g$ is a cut-off function for the $Z_g$-action on $X$. In particular, for any $x \in X$, 
\begin{enumerate}
	\item  the integral 
\[
\int_{Z_g} c_g(hx) = 1;
\]
\item  Let $q: X\to X/Z_g$ be the quotient map. The intersection of the support of $c_g$ with $q^{-1}(V')$  is compact for any compact subset $V'$ in $X/Z_g$. 
\end{enumerate}
\end{lemma}
\begin{remark}
Part (1) of Lemma \ref{lem-support} is proved in \cite[Lemma 4.11]{Hochs-Wang-HC}.
\end{remark}
\begin{proof}
For part (1), compute that 
\[
\int_{Z_g} c_g(zx)dz = \int_{Z_g}\int_G c_G(h)c(hgzx)\; dhdz =  \int_{Z_g}\int_G c_G(h'z'^{-1})c(h'x) \; dhdz = 1
\]
where we  substitute $h' = hgz$ and $z' = gz$. 

For part (2), suppose that $V'$ is a compact subset in $X/Z_g$.  The intersection of the support of $c_g$ with $q^{-1}(V')$ consists is the closure of the following set,
\[
Q_g(V'):=\{x\in X: c_g(x)\ne 0, q(x)\in V'\}. 
\]
Recall that 
\[
c_g(x)=\int_G c_G(h)c(hgx) dh
\]
Hence, $c_g(x)\ne 0$ if and only if there is some $h$ such 
\[
c_G(h)\ne 0,\ {\text{and}}\ c(hgx)\ne 0. 
\]
Using this characterization of $c_g(x)\ne 0$, we have
\[
Q_g(V')=\{x\in X: \exists h\in G, c_G(h)\ne 0, c(hgx)\ne 0, q(x)\in V'\}. 
\]
Observe that $g$ commutes with $Z_g$. So $g$ acts on the quotient $X/Z_g$, and the map $q: X\to X/Z_g$ is $g$-equivariant. Hence, substituting $x'=gx$, we obtain the following expression, 
\[
Q_g(V')=g^{-1}\left(\{x'\in X:\exists h\in G, c_G(h)\ne 0, c(hx')\ne 0, q(x')\in g^{-1}(V')\}\right).
\]
As $V'$ is compact in $X/Z_g$, $g^{-1}(V')$ is also compact. We can assume without loss of generality that $q^{-1}\big(g^{-1}(V')\big)$ is of the form 
\[
q^{-1}\big(g^{-1}(V')\big)= Z_g V'',
\]
where $V''$ is a compact subset of $X$. 

Consider the set $G(V''):=\{ h|\exists x'\in V'',c(hx')\ne 0 \}$. Using the following properties
\begin{itemize} 
\item the set $V''$ is compact,
\item the support of $c$ is compact,
\item  the $G$ action on $X$ is proper,
\end{itemize}
we conclude that $G(V'')$ is relatively compact. 

Consider the set $P_g(V''):=\{h| c(hx')\ne 0, \exists x'\in Z_g V''\}$. We have the property that $P_g(V'')=G(V'')Z_g$. Hence, we conclude from the relative compactness of $G(V'')$ that $P_g(V'')/Z_g$ is relatively compact. 

Consider the set $R_g(V''):=\{ h| \exists x'\in Z_gV'', c_G(h)\ne 0, c(hx')\ne 0\} $. $R_g(V'')$ is the intersection of $P_g(V'')$ with the support of $c_G$. As $c_G$ is a cutoff function of the right $Z_g$ action on $G$, we conclude from the relative compactness of $P_g(V'')/Z_g$ that $R_g(V'')$ is compact. 

Observe that the set 
\[
g(Q_g(V))=\{x'\in X:\exists h\in G, c_G(h)\ne 0, c(hx')\ne 0, x'\in q{x'}\in g^{-1}(V')\}
\]
is a subset of 
\[
\{x'\in X: \exists h\in R_g(V''), c(gx')\ne 0\}.
\]
As $R_g(V'')$ is compact, we conclude from the compactness of the support of $c$ and properness of the $G$ action that $g(Q_g(V'))$ is relatively compact, and therefore $Q_g(V')$ is also relatively compact. 

\end{proof}

The integral 
\[
I^t_{N(r)}(g):=\int_{N(r)} \int_G c_G(h)c(hgy) {\rm tr}(\kappa_t(y, gy)g)dh dy
\]
can be expressed by
\[
\int_{N(r)} c_{g}(y) {\rm tr}(\kappa_t(y, gy)g) dy.
\]
By Lemma \ref{lem-support}, we know that the set $X^g/Z_g$ is compact. Thus $N(r)/Z_g$ is relatively compact. It follows from Lemma \ref{lem-support}  that $c_{g}$ has relative compact support in $N(r)$ and is a smooth function on $N(r)$ with bounded derivatives.  This allows us to use the method and estimates in the proof of \cite[Equation (2.2)]{Zhangwp} to prove the following lemma.

\begin{lemma}\label{lem:eta-near fixed point}
The integral $I^t_{N(r)}(g)$ satisfies
\[
\lim_{t\to 0} \frac{|I^t_{N(r)}(g)|}{\sqrt{t}} < +\infty. 
\]
\end{lemma}

\begin{proof}
We use the local asymptotic expansion \cite[Equation (2.16)]{Zhangwp}
\[
g\kappa_t(y, g)=\frac{e^{-\frac{d_X(y,gy)^2}{4t}}}{(4\pi t)^{\frac{n}{2}}} dg\left[
\sum_i \left( ((dg -I)y)_i e_i\right)
\left( \sum_{i=0}^{[n/2]+2}U_i t^i+o(t^{[n/2]+2})\right)
+
\left(\sum_{i=0}^{[n/2]+2}V_it^i+o(t^{[n/2]+2})\right)
\right].
\]
In \cite{Zhangwp}, the author assumed to work with an isometry group action on a closed manifold $M$. For our case, as the support $c_g$ in $N(r)$ is relatively compact and $g$ is contained in the group $Z_{g}$ which acts properly on $N(r)$, the author's analysis  in \cite[Section 2]{Zhangwp} in the tubular neighborhood $N(r)$ is also valid in our case near ${\rm supp}(c_g)\cap N(r)$. 
We refer to \cite[Section 2]{Zhangwp} for the explanation of the above formula.

We use the above formula to compute $c_{g}(y){\rm tr}(\kappa_t(y, gy)g)$, 
\[
\begin{split}
&
c_{g}(y){\rm tr}(\kappa_t(y, gy)g) =
\\
&
\frac{e^{-\frac{d(y,gy)^2}{4t}}}{(4\pi t)^{\frac{n}{2}}} dg\left[
\sum_i \left( ((dg -I)y)_i e_i\right)
 \left( \sum_{i=0}^{[n/2]+2}c_{g}(y)U_i t^i+o(t^{[n/2]+2})\right)
+
\left(\sum_{i=0}^{[n/2]+2}c_{g}(y)V_it^i+o(t^{[n/2]+2})\right)
\right].
\end{split}
\]

Apply \cite[Lemma 2.17]{Zhangwp} to the integral of each term in the above expansion of $c_{g}(y){\rm tr}(\kappa_t(y, gy)g)$. We can conclude the estimate in the lemma following the same argument as the one for \cite[Equation (2.2)]{Zhangwp}, that is, 
\[
\lim_{t \to 0^+}\frac{1}{\sqrt{t}}\left|c_{g}(y){\rm tr}(\kappa_t(y, gy)g) \right| < C
\] 
for some constant $C>0$. 
\end{proof}

Combing Proposition \ref{prop:eta-outside fixed point} and Lemma \ref{lem:eta-near fixed point}, we have reached the following proposition.
\begin{proposition}\label{prop:eta-short time}  Let $X$ be a cocompact $G$-proper manifold without boundary and let $D$ be a Dirac-type operator defined in (\ref{def Dirac}).  Let $g$ be a semi-simple element in $G$.  The integral 
\[
\frac{1}{\sqrt{\pi}} \int_0^1 \tau^{X}_{g} (D\exp (-tD^2) \frac{dt}{\sqrt{t}}
\]
converges.
\end{proposition}

\subsection{Short time limit and the Atiyah-Segal integrand}
If we replace $D\exp(-sD^2)$ by $\exp(-sD^2)$, a similar argument as in Theorem \ref{thm etadefine} proves the topological 
formula stated in  \cite[Theorem 2.8]{Hochs-Wang-HC}  for the  pairing between $\tau_g$ and the index element $\operatorname{ind}_G(D)$ for a Dirac type operator $D$ on a $G$-proper cocompact manifold $X$ without boundary, as we shall now explain.  We know that 
$$\langle \Ind_G (D),\tau_g\rangle= \langle \Ind (D),\tau_g^X\rangle=\tau^X_g (e^{-s^2 D^- D^+})- \tau^X_g (e^{-s^2 D^+ D^-})\,.$$
We compute the left hand side by taking the limit as $s\downarrow 0$ 
and using the following 
\begin{theorem}\label{short-no-boundary}
Let $(X,h)$ be a cocompact $G$-proper manifold without boundary. Assume that $X$, $G/K$, and the slice $S$ are all even dimensional. Let $D$ 
be a G-equivariant Dirac operator defined in  (\ref{def Dirac}). 
Let $g$ be a semi-simple element and let $X^g$ be  the fixed point set of $g$. Then
the limit $\lim_{s\downarrow 0} \tau_g (e^{-s^2 D^- D^+})- \tau_g (e^{-s^2 D^+ D^-})$ exists and we have 
 $$\lim_{s\downarrow 0} \tau^X_g (e^{-s^2 D^- D^+})- \tau^X_g (e^{-s^2 D^+ D^-})= 
 \int_{X^g} c^g {\rm AS}_g (D)$$ 
 with $c^g {\rm AS}_g(D)$ defined  in  Equation (\ref{eq:X-geom-form}). 
 Consequently we obtain the following formula
 $$  \langle \Ind (D),\tau_g^X\rangle= \int_{X^g} c^g {\rm AS}_g (D)\,.$$ 
\end{theorem}

\begin{proof}
Let $\kappa_s$ be the kernel of the heat operator $\exp(-sD^2)$. Observe that in term of supertrace, $ \tau^X_g (e^{-s^2 D^- D^+})- \tau^X_g (e^{-s^2 D^+ D^-})$ can be written as the integral 
\begin{equation}\label{eq:super trace}
J(s):=\int_{G/Z_g}\int_X c(hgh^{-1}x) \operatorname{str}(hgh^{-1}\kappa_{s^2}(hg^{-1}h^{-1}x,x))dx d(hZ). 
\end{equation}

By a similar computation as in Equation (\ref{eq:eta integral expression}), we can express $J(s)$ defined  by Equation (\ref{eq:super trace}) as follows,
\[
J(s)=\int_X\int_G c_G(h)c(hy) {\rm str}(\kappa_{s^2}(y, gy)g)dydh,
\]
where $c_G$ is the cut-off function associated to the right $Z_{g}$ action on $G$. 
We can follow the exactly same strategy as in Section \ref{subsec:small time} to study the limit of $J(s)$ as $s\to 0$. Below is a brief outline of steps. 
\begin{enumerate}
\item We observe that the analysis on the small time behavior the operator $D\exp(-sD^2)$ in Section \ref{subsec:small time} also holds for the operator $\exp(-sD^2)$. In particular, similar to Lemma \ref{lem:Cheng-Li-Yau}, there are constant $\alpha_0, \beta_0>0$ such that 
\[
||\kappa_s(x,y)||\leq \beta_0 s^{-\frac{n}{2}}\exp\left(-\alpha \frac{d_X(x,y)^2}{s}\right). 
\]
\item Similar to Proposition \ref{prop:eta-outside fixed point}, using the above Gaussian estimate of the heat kernel $\kappa_t(x,y)$, we prove that given any $Z_g$-invariant open set $W$ containing the $g$ fixed point submanifold $X^{g}$, the integral 
\[
J_{W^c}(s):=\int_{W^c} \int_G c_G(h)c(hy)\operatorname{str}(\kappa_{s^2}(y,gy)g)dh dy=\int_{W^c}\int_G c_G(h)c(hgx)\operatorname{str}(\kappa_{s^2}(x,gx)g)dhdx
\]
decays exponentially to zero as $s\to 0$. The last equality holds because of the invariance property of $W$.
\item Following Lemma \ref{lem:eta-near fixed point}, we choose $W$ to be an $r$-ball bundle $N(r)$ in the normal bundle $N$ of $X^g$ in $X$ with a sufficiently small radius $r$.  For the integral
\[
J_{N(r)}(s):=\int_{N(r)}\int_G c_G(h)c(hy)\operatorname{str}(\kappa_{s^2}(y,gy)g)dhdy,
\]
we follow the local analysis as in the proofs \cite[Theorem 6.16]{BGV} and \cite[Proposition 4.12]{Hochs-Wang-HC} to get 
\[
\lim_{s\to 0} J_{N(r)}(s)=\int_{X^g} c^g \operatorname{AS}_g(D). 
\]
\end{enumerate}

Combining the above (1)-(3), we reach the desired equality
\[
\lim_{s\downarrow 0} \tau^X_g (e^{-s^2 D^- D^+})- \tau^X_g (e^{-s^2 D^+ D^-})=
 \int_{X^g} c^g {\rm AS}_g (D).
\]
\end{proof}

\begin{remark}Theorem  \ref{short-no-boundary} was originally stated by Hochs and Wang \cite[Theorem 2.8]{Hochs-Wang-HC} with a stronger assumption on the group element $g$.  Our strategy of proof is similar to the one in \cite{Hochs-Wang-HC}.  However, our proof does not use the inequality (4.6) in \cite{Hochs-Wang-HC} for the heat kernel $\widetilde{\kappa}_t^{G,K}$, the proof of which is not clear to us. Instead, we use the global Gaussian estimates for the heat kernel $e^{-tD^2}$, c.f. Lemma \ref{lem:Cheng-Li-Yau},
in order to establish, as a crucial step,
the exponential decay property of $\tau_g\big(\exp(-tD^2)\big)$ outside a neighborhood $W$ of the fixed point submanifold $X^g$ as in Proposition \ref{prop:eta-outside fixed point}.  Notice also that in \cite{Hochs-Wang-HC} the split Dirac operator \eqref{dirac-split}
is used throughout and we know that this operator is {\it not} associated to a Clifford connection.\\

\end{remark}

We are now in the position to prove 
Proposition \ref{prop:short-with-boundary} that we restate here for the benefit of the reader:

\medskip
\noindent
{\it Let $(Y_0,\mathbf{h}_0)$ be a cocompact $G$-proper Riemannian manifold. Let $(Y,\mathbf{h})$ be the associated $b$-manifold.
Let $D_0$ be a G-equivariant Dirac operator defined in  (\ref{def Dirac}) and let $D$ be the associated $b$-differential operator.
Let $g$ be a semisimple element and let $Y_0^g$ the fixed point set of $g$. Then
the limit $\lim_{s\downarrow 0} \tau^{Y,r}_g (e^{-s^2 D^- D^+})- \tau^{Y,r}_g (e^{-s^2 D^+ D^-})$ exists and we have 
 $$\lim_{s\downarrow 0} \tau^{Y,r}_g (e^{-s^2 D^- D^+})- \tau^{Y,r}_g (e^{-s^2 D^+ D^-})= 
 \int_{Y_0^g} c^g_0 {\rm AS}_g (D_0)$$ 
 with $c^g_0 {\rm AS}_g(D_0)$ defined  in  Equation (\ref{eq:X-geom-form}).
 }

\begin{proof}
Let $\kappa_t$ be the kernel of the heat operator $\exp(-tD^2)$. By definition
$$ \tau^{Y,r}_g (e^{-s^2 D^- D^+})-  \tau^{Y,r}_g (e^{-s^2 D^+ D^-})$$ is equal to 
\begin{equation}\label{eq:b-super trace}
{}^b J(s):=\int_{G/Z_g} \int^b _Y c(hgh^{-1}y) \operatorname{str}(hgh^{-1}\kappa_{s^2}(hg^{-1}h^{-1}y,y))dy d(hZ) 
\end{equation}
with  $dy$ denoting the b-density associated to the $b$-metric $\mathbf{h}$. We claim that the following  expression holds: 
\begin{equation}\label{eq:b-super trace-bis}
{}^b J(s)=\int^b_Y\int_G c_G(h)c(hy) {\rm str}(\kappa_{s^2}(y, gy)g) dh dy
\end{equation}
where $c_G$ is the cut-off function associated to the right $Z_{g}$ action on $G$.
In order to prove \eqref{eq:b-super trace-bis} we make a  preliminary remark. The spaces of $b$-pseudodifferential operators
${}^b \mathcal{L}_G^c (Y, E)$  and ${}^b \mathcal{L}_{G,s}^\infty (Y, E)$
have been defined in terms of the slice theorem, as  projective tensor products. We could have defined these operators directly,
in terms of a $b$-stretched product $Y\times_b Y$. Proceeding exactly as in \cite[Chapter 4]{LeichtnamPiazzaMemoires} we see that 
 $Y\times_b Y$ inherits an action of $G\times G$;
 in particular, it makes sense to consider $L^*_{(h,h')} \exp (-t D^2)$ which is a smoothing $b$-kernel (not $G$-invariant unless $h=h'$).
  Coming back to the claim, we can rewrite the right hand side of \eqref{eq:b-super trace} as
 $$\int_{G/Z_g} \int^b_Y c(y)\,{\rm str}(\kappa_{s^2}(y, hgh^{-1}y) hgh^{-1})dyd(hZ)$$ which is in turn equal to 
 \begin{equation}\label{int-with-b}
 \int_{G} c_G (h) \int^b _Y    c(y)\, {\rm str}(\kappa_{s^2} (y, hgh^{-1}y) hgh^{-1})dydh.
 \end{equation}
 Consider the smoothing $b$-kernel 
 $$\Theta_g (h):= L_{e, hgh^{-1}}^*( \exp (-(sD)^2));$$
 we have that
 \begin{equation}\label{int-with-b-2}
 \int_{G} c_G (h) \int^b _Y    c(y)\, {\rm str}(\kappa_{s^2} (y, hgh^{-1}y) hgh^{-1})dydh=
 \int_{G} c_G (h) \int^b _Y    c(y)\, {\rm str}(\kappa (\Theta_g (h)) (y,y) dydh.
 \end{equation}
 Write 
 $ \int^b _Y    c(y)\, {\rm str}(\kappa (\Theta_g (h)) (y,y) dy$ as
 $$ \int _{Y_0}    c_0(y_0)\, {\rm str}(\kappa (\Theta_g (h)) (y_0,y_0)dy_0
 + \int^b_{(-\infty,0]_t\times \partial Y_0} c_0 (x) \, {\rm str}(\kappa (\Theta_g (h)) (t,x,t,x)dtdx.
 $$ 
Consider the function $F_g (h)$ on $Y$,  so defined:

- $ F_g (h)(y) :=  c_0(y_0)\, {\rm str}(\kappa (\Theta_g (h)) (y_0,y_0)$ if $y=y_0\in Y_0$;

- $ F_g (h)(y) :=  -t \frac{d}{dt}\left( c_0( x)\, {\rm str}(\kappa (\Theta_g (h)) (t,x, t, x)\right)$ if $y=(t,x)\in (-\infty,0]_t \times \partial Y_0$.\\
The function $F_g (h)$ is discontinuous at $\partial Y_0$, a set of measure 0,  but it is otherwise smooth. By
 \cite[Prop. 2.6]{LMP} we have that 
  $$ \int_{G} c_G (h) \int^b _Y    c(y)\, {\rm str}(\kappa (\Theta_g (h)) (y,y) dydh
  = \int_{G} \int _Y   c_G (h)   F_g (h)dydh. 
  $$ 
We know that the double integral on the right hand side is absolutely convergent; we can now proceed as in \cite{Hochs-Wang-HC}, use Fubini's
theorem and interchange the two integrals in the right hand side of the above formula, obtaining
 $$ \int _Y  \int_{G} 
 c_G (h)    F_g (h)dydh\,.
  $$ 
Going back to the $b$-integral  we obtaining finally that \eqref{int-with-b} is equal to 
 $$\int^b _Y    \int_{G} c_G (h)  c(y)\, {\rm str}(\kappa_{s^2} (y, hgh^{-1}y) hgh^{-1})dydh$$ 
 which is easily seen to be equal to 
 $$\int^b_Y\int_G c_G(h)c(hy) {\rm str}(k_{s^2}(y, gy)g) dh dy$$
 as claimed. Summarizing:
 \begin{equation}\label{fubini-b}
 \tau^{Y,r}_g (e^{-s^2 D^- D^+})- \tau^{Y,r}_g (e^{-s^2 D^+ D^-})= \int^b_Y\int_G c_G(h)c(hy) {\rm str}(k_{s^2}(y, gy)g) dh dy\,.
 \end{equation}
We now consider the fixed point set $Y_0^g\subset Y_0$ and the associated $Y^g\subset Y$. We fix a $Z_g$-invariant open set $W_0$
containing $Y_0^g$ and the associated $W\subset Y$ containing $Y^g$.  We choose $W_0$ to be an $r$-ball bundle $N_0 (r)$ in the normal bundle $N_0$ of $Y^g_0$ in $Y_0$ with a sufficiently small radius $r$.  We can then decompose the right hand side of \eqref{fubini-b} as:
\begin{equation}\label{w-decomposition}
\int^b_{W^c} \int_{G} c_G(h)c(hy) {\rm str}(k_{s^2}(y, gy)g) dh dy + 
\int^b_{W} \int_{G} c_G(h)c(hy) {\rm str}(k_{s^2}(y, gy)g) dh dy 
\end{equation}
Recall that if $Z$ is a b-manifold obtained from a manifold with boundary $Z_0$ and $\phi$ is a $b$-density
obtained as an extension 
from $Z_0$ of a density  $\phi_0$ on $Z_0$, then 
$$\int^b_Z \phi = \int_{Z_0} \phi_0\,.$$
In particular if $c_0$ is a cut-off function   for the $G$ action on $Y_0$ and $c$
is obtained by extending constantly along the cylindrical end, then  
$$\int^b_Y c(y) d{\rm vol}_b  = \int_{Y_0} c_0 (y_0)  d{\rm vol} \,.$$
Thanks to Remark \ref{remark:estimate-also-b-2} we know that there are constant $\alpha_0, \beta_0>0$ such that 
\[
||k_t (x,y)||\leq \beta_0 t^{-\frac{n}{2}}\exp\left(-\alpha \frac{d_Y (x,y)^2}{t}\right). 
\]
Using this global Gaussian estimate and proceeding exactly as in the closed case we find that the first integral
in \eqref{w-decomposition} converges exponentially to $0$ (here we use the above remark about the cut-off function $c$;
it is not compactly supported but its $b$-integral is equal to the integral of a compactly supported function).
We are left with the second integral in \eqref{w-decomposition}. 
In this case $Y^g_0$ is compact and $W_0:= N_0 (r)$ is relatively compact;
put it differently, $Y^g$ and $W$ are obtained attaching cylinders to compact (respectively relatively compact)
spaces. We can thus analyze 
$$\int^b_{N(r)} \int_{G} c_G(h)c(hy) {\rm str}(k_{s^2}(y, gy)g) dh dy $$
using Getzler rescaling and find that the limit as $s\downarrow 0$ of this double integral  is equal to 
$$\int^b_{Y^g} c^g {\rm AS}_g (D).$$ 
 As all the structures are product-like near the boundary we see that
$$\int^b_{Y^g} c^g {\rm AS}_g (D)= \int_{Y_0^g} c_0^g {\rm AS}_g (D_0)$$
which is what we wanted to prove. 
\end{proof}

 \section{Higher delocalized  cyclic cocycles}\label{sect:higher-cocycles} 

In this section we shall introduce  higher delocalized cyclic cochains.\\

\noindent{Let} $K<G$ be a maximal compact subgroup and let $P<G$, $P=MAN$, be a cuspidal parabolic subgroup of $G$. By the Iwasawa decomposition $G=KMAN$ we can write an element $g\in G$ as 
\[
g = \kappa (g)\mu(g) e^{H(g)}n \in KM A N = G. 
\]
Let $\operatorname{dim}(A)=m$.  By choosing coordinates of the Lie algebra $\mathfrak{a}$ of $A$,  consider the function 
\[
H=(H_1, \dots, H_m): G\to \mathfrak{a}. 
\] 
We define the following cyclic cochain on the algebra  $\mathcal{L}_s(G)$, the Lafforgue Schwartz algebra. For $f_0, ..., f_m \in \mathcal{L}_s(G)$ and a semi-simple element $g\in M$,  define $\Phi^{P}_{g}$ by the following integral, 
\begin{equation}
\label{eq:PhiPg}
\begin{aligned}
\Phi^{P}_{g}(f_0, f_1, &\dots, f_m) \\
:=&\int_{h \in M/Z_M(g)} \int_{K N} \int_{G^{\times m}} \sum_{\tau \in S_m}\mathrm{sgn}(\tau)\cdot H_{\tau(1)}(g_1...g_mk) H_{\tau(2)}(g_2...g_mk ) \dots  H_{\tau(m)}(g_m, k) \\
& f_0 \big (kh g h^{-1}nk^{-1} (g_1\dots g_m)^{-1}\big)f_1(g_1) \dots f_m(g_m)dg_1\cdots dg_m dk dn dh,  
\end{aligned}
\end{equation}
where $Z_M(g)$ is the centralizer of $g$ in $M$. The following result is proved in the work of Song-Tang, \cite[Theorem 3.5]{st}. 

\begin{proposition}\label{prop:extend-group}$\Phi^{P}_{g}$ is a cyclic cocycle on the Harish-Chandra Schwartz algebra $\mathcal{L}_s(G)$.
\end{proposition}

\begin{definition}\label{maximal P}
We say that a cuspidal parabolic subgroup $P = MAN$ is \emph{maximal} if one of the following equivalent conditions holds
\begin{itemize}
\item 	the dimension of $A$ is minimal;
\item  the rank of $M$ is maximal;
\item the rank of $M$ equals the rank of $K$. 
\end{itemize}
In \cite[Corollary 6.3]{hst}, the authors showed that $\Phi^{P}_{g}$ is trivial in the cyclic cohomology of $\mathcal{L}_s(G)$ unless $P$ is maximal. 
\end{definition}

\medskip
\noindent
Let $Y_0$ be an {\bf even dimensional} cocompact $G$-proper manifold with boundary and let
$Y$ be the associated $b$-manifold. We shall now  use $\Phi^P_g$ in order to define a cyclic cocycle $\Phi^P_{Y,g}$ on the algebra
$\mathcal{L}^\infty_{G,s} (Y)$; subsequently we shall use $\Phi^P_{Y,g}$ in order to define a relative
cyclic cocycle 
$ (\Phi^{r,P}_{Y,g},\sigma^P_{\partial Y,g})$ for the indicial homomorphism
 ${}^b \mathcal{L}^\infty_{G,s} (Y)\xrightarrow{I} {}^b \mathcal{L}^c_{G,s,\RR}(\cyl(\partial Y))$.
 These algebras involve the choice of an $\epsilon$ strictly smaller than half of the width of the spectral gap
 for $D_{\partial Y}$. We  fix such an $\epsilon$ and we choose it in any case smaller than $1$.

\medskip
\noindent
Recall, see \cite{PP2}, that if $S$ is a slice for the action of $G$ on $Y$, then we have an identification
$$\mathcal{L}^c_G (Y)\cong \left\{F:G\to \rho_{bf}{}^b\Psi^{-\infty,\epsilon}(S),~K\times K~\mbox{equivariant, continuous and
of compact support in}~ G\right\}.$$
Recall that $\rho_{bf}{}^b\Psi^{-\infty}(S)\subset \Psi^{-\infty,\epsilon}(S)$ if $\epsilon<1$,
which is in turn contained in the trace class
operators on $L^2_b (S)$.\\
We have
$$ {}^b \mathcal{L}^c_G (Y)\cong \left\{F:G\to {}^b \Psi^{-\infty,\epsilon}(S),~K\times K~\mbox{equivariant, continuous and
of compact support in}~ G\right\}.$$
Finally, by Fourier transform we have an injection 
$${}^b \mathcal{L}^c_{G,\RR}(\cyl(\partial Y))\hookrightarrow \mathcal{S}(\RR, \mathcal{L}^c_{G}(\partial Y))\,.$$

\begin{definition}\label{def:cocycle-on-residual}
For a semisimple element $g\in M$
 and $A_0, ..., A_m\in {\mathcal{L}}^c_G(Y)$, define a cochain $\Phi^P_{Y,g}$ on $\mathcal{L}^c_G(Y)$ by 
\[
\begin{aligned}
&\Phi^P_{Y,g}(A_0, A_1, \dots, A_m) \\
:=&\int_{h \in M/Z_M(g)} \int_{K N} \int_{G^{\times m}} \sum_{\tau \in S_m}\mathrm{sgn}(\tau)\cdot H_{\tau(1)}(g_1...g_mk) H_{\tau(2)}(g_2...g_mk ) \dots  H_{\tau(m)}(g_m, k) \\
& \operatorname{Tr}\Big( A_0 \big (kh g h^{-1}nk^{-1} (g_1\dots g_m)^{-1}\big)\circ A_1(g_1) \dots \circ A_m(g_m) \Big) dg_1\cdots dg_m dk dn dh.
\end{aligned}
\]
Proceeding as in Song-Tang one can prove that this is in fact a cyclic cocycle.
\end{definition}

\noindent
Using always $\Phi^P_g$ on $\mathcal{L}_s(G)$ we shall now define a {\it relative} cyclic cocycle $( \Phi^{r,P}_{Y,g},
\sigma^P_{\partial Y,g})$ for the homomorphism
 $ {^b \mathcal{L}}^c_G(Y)\xrightarrow{I}   {}^b \mathcal{L}^c_{G,\RR}(\cyl(\partial Y))$.

\begin{definition}\label{def:cocycle-on-b}
For a semisimple element $g\in M$
 and $A_0, ..., A_m\in {^b \mathcal{L}}^c_G(Y)$, define a cochain $ \Phi^{r,P}_{Y,g}$ on ${^b \mathcal{L}}^c_G(M)$ by 
\[
\begin{aligned}
 \Phi^{r,P}_{Y,g}&(A_0, A_1, \dots, A_m) \\
:=&\int_{h \in M/Z_M(g)} \int_{K N} \int_{G^{\times m}} \sum_{\tau \in S_m}\mathrm{sgn}(\tau)\cdot H_{\tau(1)}(g_1...g_mk) H_{\tau(2)}(g_2...g_mk ) \dots  H_{\tau(m)}(g_m, k) \\
& {^b \operatorname{Tr}}\Big( A_0 \big (kh g h^{-1}nk^{-1} (g_1\dots g_m)^{-1}\big)\circ A_1(g_1) \dots \circ A_m(g_m) \Big) dg_1\cdots dg_m dk dn dh.
\end{aligned}
\]
For $B_0,...,B_{m+1}\in {^b \mathcal{L}}^c_{G, \mathbb{R}}(\text{cyl}(Y))$, define a cochain $\sigma_{\partial Y,g}$ on ${^b \mathcal{L}}^c_{G, \mathbb{R}}(\text{cyl}(Y))$ by 
\[
\begin{aligned}
&\sigma^P_{\partial Y,g}(B_0, ..., B_{m+1})\\
:= &\int_{h \in M/Z_M(g)} \int_{K N} \int_{G^{\times {m+1}}} \int_{\mathbb{R}}\sum_{\tau \in S_m}\mathrm{sgn}(\tau)\cdot H_{\tau(1)}(g_1...g_mk) H_{\tau(2)}(g_2...g_mk ) \dots  H_{\tau(m)}(g_m, k) \\
& {\operatorname{Tr}}\Big( \hat{B}_0 \big (kh g h^{-1}nk^{-1} (g_1\dots g_mg_{m+1})^{-1}, \lambda \big)\circ \hat{B}_1(g_1, \lambda)\circ \dots \circ \hat{B}_m(g_m, \lambda)\circ \frac{\partial \hat{B}_{m+1}(g_{m+1}, \lambda)}{\partial \lambda} \Big) \\
&\qquad \qquad dg_1\cdots dg_{m+1} dk dn dhd\lambda,
\end{aligned}
\]
where we have used the Fourier transform
\[
{}^b \mathcal{L}_{G,s,\RR}^c  (\cyl (Y))\ni A \longrightarrow  \widehat{A} \in \mathscr{S}(\RR,\mathcal{L}^c_{G}(Y)).
\]
\end{definition}

\begin{proposition}\label{prop:rel-cocycle}
Let $Y_0$ be a proper $G$ manifold with boundary and $Y$ be the associated $b$-manifold. For a semisimple element $g\in M$, we have the following identities:
\[
\left(\begin{array}{cc}(b+B)&-I^*\\ 0&-(b+B)\end{array}\right)\left(\begin{array}{c}  \Phi^{r,P}_{Y,g}\\ \sigma^P_{\partial Y,g}\end{array}\right)=\left(\begin{array}{c}0\\ 0\end{array}\right). 
\]
Consequently, the pair $(\Phi^{r,P}_{Y,g}, \sigma^P_{\partial Y,g})$ defines a relative cyclic cocycle for the homomorphism
 $ {^b \mathcal{L}}^c_G(Y)\xrightarrow{I}   {}^b \mathcal{L}^c_{G,\RR}(\cyl(\partial Y))$.
\end{proposition}

\begin{proof}
The proof proceeds analogous to the proof of \cite[Prop. 6.7.]{PP2}. For the Hochschild differential, we compute   
\begin{align*}
b&\Phi^{r,P}_{Y,g}(A_0,\ldots,A_{m+1})\\
&=\int_{h \in M/Z_M(x)} \int_{K N} \int_{G^{\times (m+1)}}\Bigg(
{^b \operatorname{Tr}}\Big( A_0 \big (kh g h^{-1}nk^{-1} (g_1\dots g_m)^{-1}(g')^{-1}\big)\circ A_1(g')\circ A_2(g_1)\circ \dots \circ A_{m+1}(g_m) \Big) 
\\
&+
\sum_{i=1}^m(-1)^i\left[{^b \operatorname{Tr}}\Big( A_0 \big (kh g h^{-1}nk^{-1} (g_1\dots g_m)^{-1}\big)\circ A_1(g_1)\circ\ldots\circ A_i(g_i(g')^{-1})\circ A_{i+1}(g')\circ \dots \circ A_{m+1}(g_m) \Big) \right]
\\
&+(-1)^{m+1}\left[{^b \operatorname{Tr}}\Big( A_{m+1} \big (kh g h^{-1}nk^{-1} (g_1\dots g_m)^{-1}(g')^{-1}\big)\circ A_0(g')\circ A_2(g_1)\circ \dots \circ A_{m}(g_m) \Big)\right] \Bigg)
\\
&\hspace{3cm}
\times \sum_{\tau \in S_m}\mathrm{sgn}(\tau)\cdot H_{\tau(1)}(g_1...g_mk) H_{\tau(2)}(g_2...g_mk ) \dots  H_{\tau(m)}(g_m, k)dg_1\cdots dg_m dk dn dh dg'
\\
&=(-1)^{m+1}\int_{h \in M/Z_M(x)} \int_{K N} \int_{G^{\times (m+1)}}
{^b \operatorname{Tr}}\Big( \left[A_1(g_1)\circ \cdots \circ A_{m+1} (g_{m+1}), A_0 \big (kh g h^{-1}nk^{-1} (g_1\dots g_{m+1})^{-1}\big)\right]\\
&\hspace{3cm}
\times \sum_{\tau \in S_m}\mathrm{sgn}(\tau)\cdot H_{\tau(1)}(g_1...g_mk) H_{\tau(2)}(g_2...g_mk ) \dots  H_{\tau(m)}(g_m, k)dg_1\cdots dg_{m+1} dk dn dh 
\\
&= I^*\sigma^P_{\partial Y,g} (A_0,\ldots,A_{m+1}).
\end{align*}
In this computation we have used the fact that $\Phi^P_{g}$ is a cyclic cocycle and Melrose's formula for the $b$-trace.  
To show that $B\Phi^{r,P}_{Y,g}=0$ we write out the differential
\[
B \Phi^{r,P}_{Y,g}(A_0,\ldots,A_{m-1})=\sum_{i=0}^{m-1}(-1)^{m-1)i}\Phi^{r,P}_{Y,g}(1,A_i,\ldots,A_{m-1},A_0,\ldots,A_{i-1}),
\]
where $1$ is the delta function at the unit. Ignoring the sign, we can write the $i$'th term of this expression as
\begin{align*}
&\int_{h \in M/Z_M(x)} \int_{K N} \int_{G^{\times m}} \sum_{\tau \in S_m}\mathrm{sgn}(\tau)\cdot H_{\tau(1)}(g_1...g_mk) H_{\tau(2)}(g_2...g_mk ) \dots  H_{\tau(m)}(g_m k) \\
& \delta_e \big(kh g h^{-1}nk^{-1} (g_1\dots g_m)^{-1}\big)
^{b}\operatorname{Tr}\Big(  A_i(g_1) \dots \circ A_m(g_{m-i+1})\circ A_0(g_{m-i})\circ A_{i-1}(g_{m}) \Big) dg_1\cdots dg_m dk dn dh\\
&=\int_{h \in M/Z_M(x)} \int_{K N} \int_{G^{\times m}} \sum_{\tau \in S_m}\mathrm{sgn}(\tau)\cdot H_{\tau(1)}(khgh^{-1}n) H_{\tau(2)}(g_2...g_mk ) \dots  H_{\tau(m)}(g_m, k) \\
&^{b}\operatorname{Tr}\Big(  A_i(khgh^{-1}nk(g_2\cdots g_m)^{-1}) \dots \circ A_m(g_{m-i+1})\circ A_0(g_{m-i})\circ A_{i-1}(g_{m}) \Big)
dg_2\cdots dg_m dk dn dh\\
&=0,
\end{align*}
because $H(khgh^{-1}n)=0$. 
This shows that $B\Phi^{r, P}_{Y,g}=0$ and injectivity of $I^*$ show that $(b+B)\sigma^P_{\partial Y,g}=0$.
\end{proof}

\medskip
\noindent
{\bf Notation:} we shall often omit the parabolic subgroup $P$ from the 
notation, thus denoting by $\Phi_g$ the cyclic cocycle on $\mathcal{L}_s(G)$, by  
$\Phi_{Y,g}$ the cyclic cocycle on $\mathcal{L}^c_G (Y)$ and by 
$( \Phi^r_{Y,g}, \sigma_{\partial Y,g})$ the relative cyclic cocycle 
for  $ {^b \mathcal{L}}^c_G(Y)\xrightarrow{I}   {}^b \mathcal{L}^c_{G,\RR}(\cyl(\partial Y))$.

\section{Toward a  general higher APS index formula }\label{sect:toward}

\begin{proposition}\label{prop:extension} Assume that $G$ is a connected, linear real reductive group. Then:
\begin{itemize}
\item[1]  the   cyclic cocycle 
$\Phi^P_{Y,g}$ extends continuously from $\mathcal{L}^c_G(Y)$ to  $\mathcal{L}^\infty_{G,s}(Y)$; 
\item[2] the relative cyclic cocycle  $(\Phi^{r,P}_{Y,g}, \sigma_{\partial Y,g})$   extends 
continuously from the pair  $ {^b \mathcal{L}}^c_G(Y)\xrightarrow{I}   {}^b \mathcal{L}^c_{G,\RR}(\cyl(\partial Y))$
to the pair  $ {^b \mathcal{L}}^\infty_{G,s}(Y)\xrightarrow{I}   {}^b \mathcal{L}^\infty_{G,s,\RR}(\cyl(\partial Y))$.  
\end{itemize}
\end{proposition}
\begin{proof}
The proof of the first statement is analogous to the considerations in \cite{hst} depending crucially on the inequality proved in \cite[Thm A.5]{st}:
\[
\left|\Phi^P_g (f_0,\ldots,f_m)\right|\leq C\nu_{d_0+T_0+1}(f_0)\cdots \nu_{d_0+T_0+1}(f_m),\qquad f_0,\ldots,f_m\in \mathcal{L}_s(G),
\]
where
\[
\nu_t(f):=\sup_{g\in G}|(1+||g||)^t\Xi(g)^{-1}f(g)|,
\]
and $T_0$ and $d_0$ as in \cite{st}. Given $A\in \mathcal{L}^c_G(Y)$, we introduce the norm
\[
|||A|||_t:=\sup_{g\in G}|(1+||g||)^t\Xi(g)^{-1}||A(g)||_b,
\]
where$ || P ||_b^2 := \| \chi P \|^2_1 + \| \phi [\mathcal{V},P] \|^2_1 + \| [\mathcal{V},P] \|^2_1
 + \| [\phi ,P] \|^2 + \| P \|^2$ for $P\in {}^b\Psi^{-\infty,\epsilon} (S) + \Psi^{-\infty,\epsilon} (S) $ (c.f. \cite[Definition 7.3]{PP2}). Because $\mathcal{L}^\infty_{G,s}(Y)$ lies inside the norm-completion of $\mathcal{L}^c_G(Y)$
 with respect to $|||~|||_s$, it suffices to estimate the cocycle $\Phi^P_{Y,g}$ in one of these norms. For this we use the inequality
 \[
 |   {}^b\Tr  (P_0 P_1 \cdots P_k)| \leq  C || P_0 ||_b \cdots || P_k ||_b,
 \]
 c.f.\ \cite[Lemma 6.4.]{GMPi} to obtain
 \[
 \left|\Phi^P_{Y,g}(A_0,\ldots,A_m)\right|\leq C |||A_0|||_{d_0+T_0+1}\cdots |||A_m|||_{d_0+T_0+1}.
 \]
This proves the first claim.  For the second claim we proceed as in \cite[Proposition 7.12]{PP2}: we use
the usual trace inequality $|\Tr(AB)|\leq ||A||_1||B||_1$ together with the estimate $|H_i(gk)|\leq C_iL(g)$ of \cite[Proposition A.2]{st} to find 
\begin{align*}
|&\sigma^P_{\partial Y,g}(B_0, ..., B_{m+1})| \\
&\leq C \int_{h \in M/Z_M(g)} \int_{K N} \int_{G^{\times {(m+1)}}} \tilde{f}_0(kh g h^{-1}nk^{-1} (g_1\dots g_mg_{m+1})^{-1}, \lambda)\tilde{f}_1(g_1,\lambda)\cdots\\
&\qquad\qquad\qquad\qquad\qquad \qquad \qquad \tilde{f}_{m+1}(g_{m+1},\lambda) dg_1\cdots dg_{m+1} dk dn dhd\lambda,
\end{align*}
where 
\begin{align*}
\tilde{f}_0(g,\lambda)&:=||\hat{B}_0(g,\lambda)||_1,\\
\tilde{f}_i(g,\lambda)&:=||\hat{B}_i(g,\lambda)||_1(1+L(g))^i,\qquad i=1,\ldots,m,\\
\tilde{f}_{m+1}(g,\lambda)&:=||\frac{\partial \hat{B}_{m+1}(g,\lambda)}{\partial\lambda}||_1
\end{align*}
By continuity of the map $||~||_1:\mathcal{L}^\infty_{G,s}(\partial Y)\to\mathcal{L}_s(G)$, c.f. \S \ref{oiclosed}, we see that $\tilde{f}_j\in \mathcal{L}_s(G)$, for all $j=1,\ldots,m+1$.
We can therefore rewrite the right hand side of the equality above as
\[
\int_{h \in M/Z_M(g)} \int_{K N}\int_\RR F(kh g h^{-1}nk^{-1},\lambda)dk dn dhd\lambda,
\]
with $F:=\tilde{f}_0*\ldots*\tilde{f}_{m+1}$. Convergence of this integral now follows as in \cite[Theorem A.5]{st}.
\end{proof}

\begin{definition}
\label{defn:higher-eta} 
Let $p_t=V(t D_{\text{cyl}})$ and $c_m=(-1)^{\frac{m}{2}}\frac{m!}{(\frac{m}{2})!}$. Fix a cuspidal parabolic subgroup $P=MAN<G$ and $g\in M$
a semisimple element.
Let 
$$\eta^P_g (t):=2 c_m  \sum_{i=0}^{m} \sigma^P_{\partial Y,g} (p_t, ..., [\dot{p}_t, p_t], ..., p_t)$$
We define the higher eta invariant associated to $\Phi^P_g$ and the boundary operator $D_\partial$ as
\begin{equation}\label{general-rho}
\eta_g^P (D_\partial):= \lim_{\epsilon\downarrow 0} \int_\epsilon^{1/\epsilon} \eta^P_g (t)dt \equiv 
\lim_{\epsilon\downarrow 0} \int_\epsilon^{1/\epsilon} 2 c_m  \sum_{i=0}^{m} \sigma^P_{\partial Y,g} (p_t, ..., [\dot{p}_t, p_t], ..., p_t) dt
\end{equation}
if this limit exists.\\
\end{definition}

\begin{theorem}\label{thm:splitting} 
Let $s\in (0,1]$. For the index pairing 
$\langle \Ind_\infty(D), [\Phi_{Y,g}]$ the following formula holds:

\[
c_m \langle \Ind_\infty(D), [\Phi^P_{Y,g}]\rangle=
\Phi^{r,P}_{Y,g} (V(sD), \dots , V(sD))- \frac{1}{2}\int_s^\infty \eta^P_g(t)dt
\]
where part of the statement is that the $t$-integral converges at $+\infty$.
\end{theorem}

\begin{proof}
One first establishes the equality
\begin{equation}\label{equal-relative-abs}
\langle \Ind_\infty(D), [\Phi^P_{Y,g}]\rangle = \langle \Ind_\infty(D,D_\partial), [\Phi^{r,P}_{Y,g},\sigma^{P}_{\partial Y,g}]\rangle\,,
\end{equation}
 exactly as in \cite[Theorem 9.7]{moriyoshi-piazza} and \cite[Theorem 7.17]{PP2}. By definition of relative pairing
 and by the very definition of our relative index class, that is
 $$\Ind_\infty(D,D_\partial):=
[V(D), e_1, q_t]
\,, \;\;t\in [1,+\infty]\,,\;\;\text{ with }
q_t:= \begin{cases} V(t D_{\cyl})
\;\;\quad\text{if}
\;\;\;t\in [1,+\infty)\\
e_1 \;\;\;\;\;\;\;\;\;\;\;\;\;\,\text{ if }
\;\;t=\infty
 \end{cases}
$$
we then obtain from \eqref{equal-relative-abs} the following formula
 \[
 c_m \langle \Ind_\infty(D), [\Phi^P_{Y,g}]\rangle=
\Phi^{r,P}_{Y,g} (V(D), \dots , V(D))- \frac{1}{2}\int_1^\infty \eta^P_g(t)dt
\]
The formula we want to prove is obtained by rescaling  $D$ to $sD$.
\end{proof}

\begin{remark}
We would like to take the limit as $s\downarrow 0$ in Theorem \ref{thm:splitting} and obtain directly a higher delocalized APS index formula
as the sum of a geometric term and the higher delocalized eta invariant. Unfortunately at the moment it is quite unclear
how to study the limit as $s\downarrow 0$ of $\Phi^{r,P}_{Y,g} (V(sD), \dots , V(sD))$
which is why in the next section we give a treatment of the higher APS index formula corresponding to $\Phi^P_g$
through reduction, as in the closed case treated by Hochs, Song and Tang in \cite{hst}.
Notice that even for a cocompact $G$-proper manifold {\it without} boundary $X$ it is a difficult problem to study the limit
$$\lim_{s\to 0}\Phi^{P}_{X,g} (V(sD), \dots , V(sD))$$
with $V(D)$ the (symmetrized) Connes-Moscovici projector.
\end{remark}

\section{Reduction}\label{sect:reduction}

Suppose that $Y$ is a smooth $G$-proper manifold with boundary. Let $P = MAN$ be a cuspidal parabolic subgroup. Since the nilpotent subgroup $N$ acts freely on $X$, we can consider $Y_{MA} := Y/N$ which is a smooth manifold with a proper $MA$-action. In addition, the abelian group $A$ also acts freely on $Y_{MA}$, we define 
\[
Y_M = Y/(AN) = Y_{MA}/A
\]
which is a smooth manifold with a proper $M$-action. The following equation of the index pairings was proved \cite[Proposition 4.8]{hst} for a smooth $G$-proper manifold $X$ without boundary, 
\[
 \langle \Phi_{X,g}, \Ind_\infty (D) \rangle = \langle \Phi_{X_{MA},g}, \Ind_\infty (D_{X_{MA}})  \rangle\,.
 \]
 Moreover, if the metric is slice compatible (Definition \ref{def:slice-metric}), we have \cite[Lemma 5.3]{hst}
\[
 \langle \Phi_{X_{MA},g}, \Ind_\infty (D_{X_{MA}})  \rangle= \langle \Phi_{X_M,g}, \Ind_\infty (D_{X_M}) \rangle.
 \]
 These two equations together give a cohomological formula for $\langle \Phi_{X,g}, \Ind_\infty (D) \rangle$, with geometric information coming from $X_{M,g}$. In this section, we generalize the above computation for  $\langle \Phi_{X,g}, \Ind_\infty (D) \rangle$ to a $G$-proper manifold $Y$ with boundary. In particular, we will need to study the invertibility of the Dirac operators on $\partial (Y/N)$ and $\partial (Y_{M})$.
\subsection{First reduction for manifolds without boundary}

Recall that 
\[
\mathcal{L}^\infty_{G,s}(X,E) \cong \left(\mathcal{L}_s(G)\hat{\otimes}\Psi^{-\infty}(S, E|_S)\right)^{K\times K}, \quad \mathcal{L}^\infty_{G,s}(X_{MA},E|_{X_{MA}}) \cong \left(\mathcal{L}_s(MA)\hat{\otimes}\Psi^{-\infty}(S, E|_S)\right)^{(K \cap M)\times (K \cap M)}.
\]
\begin{definition}
For any $f \in 	\mathcal{L}_s(G)$, Harish-Chandra  \cite{Harish-Chandra-dis} defines a  map $f^N$ on $MA$ by 
\[
f(ma) \colon = \int_N f(man) dn. 
\]
Then for any $k \in \mathcal{L}^\infty_{G,s}(X,E)$, and for all $m \in M, a \in A, y, y'\in S$, put
\begin{equation}\label{integration-on-N}
 k^N (ma, y, y'):= \int_N k(nma, y, y') dn \in \text{Hom}\left(E|_y, E|_{y'}\right).
\end{equation}
\end{definition}
\begin{lemma}
\label{integral N lem}
The map defined in (\ref{integration-on-N}) is a multiplicative map:
\[
\phi^N \colon {\mathcal{L}}^\infty_{G,s}(X,E)\ni k\longrightarrow k^N \in  {\mathcal{L}}^\infty_{G,s}(X_{MA},E|_{X_{MA}})
\]
Moreover, the map $\phi^N$ is continuous with respect to Fr\'echet topologies. 
\begin{proof}
Harish-Chandra proved in \cite{Harish-Chandra-dis}[Lemma 22] that (\ref{integration-on-N}) defines a continuous linear map  
\[
\mathcal{L}_s(G)\ni k\longrightarrow k^N \in  \mathcal{L}_s(MA).
\]
It remains to show that it is multiplicative. Suppose that 
\[
\kappa_1, \kappa_2 \in {\mathcal{L}}^\infty_{G,s}(X,E). 
\]
Because Lafforgue's Schwartz algebra is closed under convolution,
\[
\kappa_1^N \star \kappa_2^N \in {\mathcal{L}}^\infty_{G,s}(X_{MA},E|_{X_{MA}}). 
\]
By definition and the Iwasawa decomposition $G = KNMA$,
\begin{equation}
\label{kappa eq-1}
\begin{aligned}
\left(\kappa_1 \star \kappa_2\right)^N(ma) = &\int_G \int_N \kappa_1(nmag'^{-1}) \kappa_2(g') \; dndg'\\
=&\int_K\int_M \int_A\int_N \int_N \kappa_1(nmaa'^{-1}m'^{-1}n'^{-1}k'^{-1}) \kappa_2(k'n'm'a') \;dk' dm'da'dn' dn
\end{aligned}
\end{equation}
Since the kernels $\kappa_i, i=1,2$ are $K\times K$-equivariant, we have 
\[
\kappa_1(k'gk) = k' \cdot \kappa_1(g) \cdot  k,
\] 
where the $k, k'$ on the right-hand side denotes the $K$-action on $\Psi^{-\infty}(S)$. 
A similar equation holds for $\kappa_2$. The last equation in (\ref{kappa eq-1}) becomes
\[
\int_M \int_A\int_N \int_N \kappa_1(nmaa'^{-1}m'^{-1}n'^{-1}) \kappa_2(n'm'a') \; dm'da'dn'dn 
\]
Recall that $MA$ normalizes $N$, that is,
\[
nmaa'^{-1}m'^{-1}n'^{-1} =n''maa'^{-1}m'^{-1}
\]
for some $n'' \in N$. We conclude that 
\[
\begin{aligned}
\left(\kappa_1 \star \kappa_2\right)^N(ma) =&\int_M \int_A\int_N \int_N \kappa_1(n''maa'^{-1}m'^{-1}) \kappa_2(n'm'a') \; dm'da'dn' dn''\\
=&\int_M \int_A \kappa_1^N(maa'^{-1}m'^{-1}) \kappa_2^N(m'a') \;dm'da'\\
=&\left(\kappa_1^N \star \kappa_2^N\right)(ma).
\end{aligned}
\]
\end{proof}
\end{lemma}
\begin{remark}
By Lemma \ref{integral N lem}, the map $\phi^N$ induces a map 
\[
\phi^N_* \colon K^*\left(\mathcal{L}^\infty_{G,s}(X,E)\right) \cong K^*\left(C^*_r(G)\right) \to K^*\left(\mathcal{L}^\infty_{G,s}(X_{MA},E|_{X_{MA}})\right) \cong K^*\left(C^*_r(MA)\right).
\]	
\end{remark}

We consider $C^\infty (X,E)$ and $C^\infty (X,E)^{N,c}$, the smooth $N$-invariant sections of $E$ with compact support in $X/N\cong X_{MA}$.  Let $D$ be a Dirac operator on $X$ and $k_t  \in {\mathcal{L}}^\infty_{G,s}(X,E)$ be the Schwartz kernel of $\exp(-tD^2)$. Since $k_t^N \in {\mathcal{L}}^\infty_{G,s}(X_{MA},E|_{X_{MA}})$, we can regard it
as an operator on $C^\infty (X,E)^{N,c}$.

\begin{lemma}\label{lem:heatkernel}
For all $\sigma\in C^\infty_c (X, E)$ and $s\in C^\infty (X,E)^{N,c}$ we have
$$(k_t^*\sigma,s)_{L^2 (X,E)}= (\sigma, k_t^N s)_{L^2 (X,E)},$$
where $k_t^*$ is the adjoint kernel of $k_t$. Moreover, 
\[
\frac{d}{dt}\left(\sigma, k_t^N s \right)_{L^2 (X,E)} = \left(\sigma, -D^2 \exp(-tD^2) s \right)_{L^2 (X,E)}
\]
\begin{proof} This lemma is proved in 
\cite[Lemma 4.1 and Lemma 4.2]{hst}. We have included it here for the reader's convenience.
\end{proof}
\end{lemma}
\begin{lemma}
\label{lem-Schwartz}
Let $D_{X_{MA}}$ be the Dirac operator on $X_{MA}$ induced from $D$. Then the Schwartz kernel of $\exp(-tD_{X_{MA}}^2)$ is $k^N_t$. 
\begin{proof}
Since $X_{MA} = X/N$, we can identify $C^\infty (X,E)^{N,c}\cong  C^\infty (X_{MA},E|_{X_{MA}})$. Under such an identification, the restriction $D$ on $N$-invariant sections equals $D_{X_{MA}}$. By Lemma \ref{lem:heatkernel}, we have that
\[
\frac{d}{dt}\left(\sigma, k_t^N s \right)_{L^2 (X,E)} = \left(\sigma, -D_{X_{MA}}^2 k_t^N s \right)_{L^2 (X,E)} 
\] 
and
\[
\lim_{t \downarrow 0}\left(\sigma, k_t^N s \right)_{L^2 (X,E)}  = \left(\sigma, s \right)_{L^2 (X,E)}. 
\]
Using the uniqueness of the heat equation with usual initial data,
we conclude that $k_t^N$ is the Schwartz kernel of $\exp(-tD_{X_{MA}}^2)$. 
\end{proof}
\end{lemma}

\begin{remark} Lemma \ref{lem-Schwartz} is stated in \cite[Lemma 4.5]{hst} for the wrong operator $D_{X_{MA}}$ introduced in \cite[Eq. (4.10)]{hst}. It is the operator $D_{X_{MA}}$ introduced in Lemma \ref{lem-Schwartz}
that carries the right property for the development. It is not hard to see that $D_{X_{MA}}$ is a twisted Dirac operator on $X_{MA}$.
\end{remark}

\begin{remark}
As in \cite[Lemma 4.6]{hst}, one can similarly show that $V(tD)^N = V(tD_{X_{MA}})$, where $V$ denotes the Connes-Moscovici projection. 
\end{remark}

\begin{lemma}
\label{3-invertibility}
If the Dirac operator $D$ on $X$ is invertible, then the corresponding Dirac operator $D_{X_{MA}}$ is invertible as well. 
\begin{proof}
Suppose that 
\[
\text{Spec}(D) \cap (-\delta, \delta) = \emptyset
\]	
for some $\delta>0$. We have shown in Part II (we should move the proof to Part I later) that the Schwartz kernel 
\[
e^{\frac{t\delta}{2}} \cdot k_t \in \mathcal{L}^\infty_{G,s}(X,E)
\]
converges to zero in Frechet topology as $t \to \infty$. Now the key observation is that the map 
\[
 \mathcal{L}^\infty_{G,s}(X,E) \ni k_t \longrightarrow  k_t^N \in \mathcal{L}^\infty_{G,s}(X_{MA},E|_{X_{MA}})
\]
 is continuous by Lemma \ref{integral N lem}. Thus, $k_t^N$, which is the Schwartz kernel of $e^{-tD_{X_{MA}}^2}$ by Lemma \ref{lem-Schwartz}, converges to zero faster than $e^{\frac{-t\delta}{2}}$ as $t \to \infty$. It follows that the operator $e^{\frac{t\delta}{2} -tD_{X_{MA}}^2}$ is uniformly bounded. Thus,  we can find a constant $C$ such that for  $t \gg 0$, and any $s \in L^2(X_{MA}, E|_{X_{MA}})$, 
\begin{equation}
\label{exp strum}
\langle e^{-tD_{X_{MA}}^2}s, s \rangle   \leq  e^{-\frac{t\delta}{3}} \cdot \|s\|^2.	
\end{equation}
We denote by $E_\lambda$ the spectrum measure associated to the self-adjoint operator $D_{X_{MA}}^2$, that is 
\[
E_\lambda \colon \lambda \to \text{projection on }L^2(X_{MA}, E|_{X_{MA}}), \quad \lambda \in \text{Spec}(D_{X_{MA}}^2) \subseteq [0, +\infty). 
\]
For any $s \in L^2(X_{MA}, E|_{X_{MA}})$, 
\begin{align*}
\langle e^{-tD_{X_{MA}}^2} s, s\rangle =& \int_0^\infty e^{-t\lambda} \; d\langle E_\lambda s, s \rangle	\\
&\geq \int_0^{\frac{\delta}{4}} e^{-t\lambda} \; d\langle E_\lambda s, s \rangle \geq e^{-\frac{t\delta}{4}} \cdot \int_0^{\frac{\delta}{4}}  \; d\langle E_\lambda s, s \rangle
\end{align*}
If $\text{Spec}(D_{X_{MA}}^2) \cap \left[0, \frac{\delta}{4}\right) \neq \emptyset$, then 
\[
P \colon = \int_0^{\frac{\delta}{4}}  \; d  E_\lambda 
\]
defines a non-zero projection on $L^2(X_{MA}, E|_{X_{MA}})$. If we take $s_0 \in \text{Image}\left(P\right)$, then
\[
\langle e^{-tD_{X_{MA}}^2}s_0, s_0 \rangle  \geq e^{-\frac{t\delta}{4}} \cdot  \|s_0\|^2, 
\]
which contradicts to (\ref{exp strum}). This completes the proof.  
\end{proof}
\end{lemma}

\subsection{First reduction for manifold with boundary}
Suppose that $Y$ is a smooth $G$-proper manifold with boundary, denoted  $\partial Y$. Let 
$$
g \in \mathcal{L}^\infty_{G,s}(Y,E)\,;\quad k\in {^b \mathcal{L}}^\infty_{G,s}(Y,E)\,; \quad  k_\RR\in {}^b \mathcal{L}^\infty_{G,s,\RR}(\cyl(\partial Y),p^*E_{\partial Y}),$$
where for the sake of clarity we have now included the bundles in the notation; $p:\RR\times \partial Y\to \partial Y$ is 
the obvious projection.
We see these kernels as functions on $G$ with values in pseudodifferential operators on the slice $S$ satisfying a $K\times K$-equivariance.
We can then define 
$$g^N\in  \mathcal{L}^\infty_{G,s}(Y_{MA},E|_{Y_{MA}})\,;\quad k^N\in {^b \mathcal{L}}^\infty_{G,s}(Y_{MA},E|_{Y_{MA}})\,; 
\quad  k^N_\RR\in {}^b \mathcal{L}^\infty_{G,s,\RR}(\cyl(\partial Y_{MA}), p^* E|_{\partial Y_{MA}}).$$
These are functions on $MA$ with values in psudodifferential operators on the slice $Z$ satisfying a $(K\cap M)\times (K\cap M)$-equivariance.
See \cite{hst}, Section 4.1. We can of course extend this map to $M_{n\times n} (\mathcal{L}^\infty_{G,s}(Y,E))$, 
$M_{n\times n} ( {^b \mathcal{L}}^\infty_{G,s}(Y,E) )$ and  $M_{n\times n} ({}^b \mathcal{L}^\infty_{G,s,\RR}(\cyl(\partial Y),p^*E_{\partial Y}))$.
By an argument similar to that given  in Lemma \ref{integral N lem}, we have that the following three maps
\begin{enumerate}
\item ${\mathcal{L}}^\infty_{G,s}(Y,E)\ni g\longrightarrow g^N \in  {\mathcal{L}}^\infty_{G,s}(Y_{MA},E|_{Y_{MA}})$
	\item ${^b \mathcal{L}}^\infty_{G,s}(Y,E)\ni k\longrightarrow k^N \in  {^b \mathcal{L}}^\infty_{G,s}(Y_{MA},E|_{Y_{MA}})$
	\item $ {}^b \mathcal{L}^\infty_{G,s,\RR}(\cyl(\partial Y),p^*E_{\partial Y}) \ni k_\RR  \longrightarrow k_\RR^N \in {}^b \mathcal{L}^\infty_{G,s,\RR}(\cyl(\partial Y_{MA}), p^* E|_{\partial Y_{MA}})$
\end{enumerate}
are all multiplicative maps and continuous with respect to Fr\'echet topologies. As before, we have the following result:
\begin{proposition}\label{prop:reduction-of-CM}
We have that 
\begin{itemize}
\item $k^N_t$  is equal to the Schwartz kernel  of 
$\exp (-t D^2_{Y_{MA}})$; an analogous result  holds for $k^N_{\mathbb{R}}$.
\item $V(tD)^N= V(t D_{Y_{MA}})$.
\item If the boundary operator of $D$ is $L^2$-invertible, then  $({}^b V(tD))^N = {}^b V(D_{Y_{MA}})$.
\end{itemize}
\end{proposition}

We shall now consider various index classes. We consider $Y_{MA}$ and $Y_M$.
The manifold $Y_{MA}$ is a product, i.e. 
\[
Y_{MA}=Y_M\times A,
\]   
where $M$ acts properly on $Y_M$ and trivially on $A$, and $A$ acts properly and freely on $A$ and trivially on $Y_M$.
\begin{proposition}\label{prop:existence-index-reduced}
Assume  that the metric on $Y$ is slice compatible and that the boundary operator of $D$
is $L^2$-invertible.  Then the boundary operator of $D_{Y_{MA}}$ and of $D_{Y_M}$ are also $L^2$-invertible.
Consequently there are well defined index classes:
\begin{equation}\label{existence-index-reduced}
\Ind_\infty (D_{Y_{MA}})\in K_0 (\mathcal{L}^\infty_{G,s}(Y_{MA},E|_{Y_{MA}}))\,,\quad 
 \Ind_\infty (D_{Y_M}) \in K_0 (\mathcal{L}^\infty_{G,s}(Y_M,E_M))
\end{equation} 
with $E_M$ as in (\ref{EM}) below.
\end{proposition}

\begin{proof}
If $D_{\partial Y}$ is $L^2$-invertible, then we can use Lemma \ref{3-invertibility} in order to see
directly that the boundary operator of $D_{Y_{MA}}$ is also $L^2$-invertible.
Next, as the metric is slice compatible, we have that the induced metric on $Y_{MA}=Y_M\times A$ is a product metric. It follows that the Dirac operator $D_{Y_{MA}}$ decomposes as 
\begin{equation}\label{eq:Dirac-MA}
D_{Y_{MA}} = D_{Y_M} \hat{\otimes} 1 + 1 \hat{\otimes} D_A,
\end{equation}
where we use the graded tensor products. This means that 
$$D^2_{Y_{MA}} = D^2_{Y_M}\otimes 1 + 1 \otimes D^2_A$$
As $D^2_{Y_{MA}}$ is $L^2$-invertible and $D^2_A$ is {\it not} $L^2$-invertible ($A$ is isomorphic to
$\mathbb{R}^n$), we see that $D^2_{Y_M}$ and thus $D_{Y_{M}}$ must be $L^2$-invertible.
\end{proof}

Consider now the cyclic cocycle $\Phi_{MA,g}$ on $\mathcal{L}_s(MA)$, see \cite[Section 3.1]{hst}, and the associated cyclic cocycle
on  $ \mathcal{L}^\infty_{G,s}(Y_{MA},E|_{Y_{MA}})$, denoted  $\Phi_{Y_{MA},g}$. Consider the index class
$\Ind_\infty (D)\equiv [{}^b V(D)]$ and the index class $\Ind_\infty (D_{Y_{MA}})\equiv [{}^b V(D_{Y_{MA}})]$, which is well defined
because of Proposition \ref{prop:existence-index-reduced}.

\begin{proposition}\label{prop:reduction-1}
 The following equality holds
 \begin{equation}\label{reduction1}
 \langle \Phi_{Y,g}, \Ind_\infty (D) \rangle = \langle \Phi_{Y_{MA},g}, \Ind_\infty (D_{Y_{MA}})  \rangle\,.
 \end{equation}
 \end{proposition}
 
 \begin{proof}
  We must prove that 
  \begin{equation}\label{reduction1-bis}
 \langle \Phi_{Y,g}, [{}^b V(D)] \rangle = \langle \Phi_{Y_{MA},g},   [{}^b V(D_{Y_{MA}})] \rangle\,.
 \end{equation}
 This follows combining \cite[Proposition 3.2]{hst}, which clearly holds for the algebra 
 of residual operators
 $\mathcal{L}^\infty_{G,s}(Y,E)$,
 and Proposition \ref{prop:reduction-of-CM} above.
\end{proof}

\subsection{Second reduction}The Lie group $M$ is of equal rank and acts properly on $Y_M$. Consider the orbital integral  $\tau^M_{g}$ on $\mathcal{L}_s(M)$ associated to a semisimple element $g$ in $M$. As in subsection \ref{section:0cocycle}, we can associated the orbital integral $\tau^M_{g}$ a  $0$-cocycle
on $Y_M$, denoted  $\Phi_{Y_M,g}$. More precisely, $\Phi_{Y_M,g}$ is a cyclic 0-cocycle on the algebra
$\mathcal{L}^\infty_{G,s}(Y_M,E_M)$, with $E_M$ as in (\ref{EM}).

\begin{proposition}\label{label:reduction-final}
If the metric on $Y$ is slice compatible and the boundary operator on $Y$ is $L^2$-invertible, then the following equality holds:
\begin{equation}\label{reduction2}
 \langle \Phi_{Y_{MA},g}, \Ind_\infty (D_{Y_{MA}})  \rangle= \langle \Phi_{Y_M,g}, \Ind_\infty (D_{Y_M}) \rangle 
 \end{equation}
\end{proposition}

\begin{proof} 

We prove this identity in two steps. 

\noindent{\bf Step I}: Recall that $\Ind(D_{Y_{MA}})$ is an $K$-theory element of the algebra $\mathcal{L}_{MA,s}^{\infty}(Y_{MA}, E_{MA})$. Recall that the manifold $Y_{MA}$ is a product, i.e. 
$Y_{MA}=Y_M\times A$, 
where $M$ acts properly on $Y_M$ and trivially on $A$, and $A$ acts properly and freely on $A$ and trivially on $Y_M$. With this decomposition, we can write the algebra $\mathcal{L}_{MA,s}^{\infty}(Y_{MA}, E_{MA})$ as follows,  
\begin{equation}\label{eq:product-algebra}
\mathcal{L}_{MA,s}^{\infty}(Y_{MA}, E_{MA})=\mathcal{L}_{M,s}^\infty(Y_M, E_M)\hat{\otimes} \mathcal{L}_s(A, \operatorname{End}(\mathcal{S}_A)), 
\end{equation}
where $\mathcal{S}_A$ is the space of spinors on $\mathfrak{a}$, the Lie algebra of $A$, and $\operatorname{End}(\mathcal{S}_A)$ is the algebra of endomorphisms on $\mathcal{S}_A$. 
Now, we consider the index elements $\Ind_\infty (D_{Y_{MA}} )$, $\Ind_\infty(D_{Y_M})$, and $\Ind_\infty (D_A)$, i.e.
\[
\Ind_\infty (D_{Y_{MA}})\in K_*(\mathcal{L}_{MA,s}^\infty(Y_{MA}, E_{MA})),\ \Ind_\infty (D_{Y_M})\in K_*(\mathcal{L}^{\infty}_{M,s}(Y_M, E_M)),\ \Ind_\infty (D_A)\in K_*(\mathcal{L}_s(A, \operatorname{End}(\mathcal{S}_A))).
\]

We notice that the decomposition Eq. (\ref{eq:product-algebra}) of $\mathcal{L}_{MA,s}^\infty(Y_{MA}, E_{MA})$ defines an element 
\[
\Ind_\infty (D_{Y_M})\otimes \Ind_\infty (D_{A})
\]
through the external product 
\[
K_*(\mathcal{L}^{\infty}_{M,s}(Y_M, E_M))\otimes K_*(\mathcal{L}_s(A, \operatorname{End}(\mathcal{S}_A)))\to K_*(\mathcal{L}^{\infty}_{MA,s}(Y_{MA}, E_{MA})). 
\]

\noindent
We claim that 
\begin{equation}\label{index dec}
\Ind_\infty (D_{Y_{MA}})= \Ind_\infty (D_{Y_{M}})\otimes \Ind _\infty(D_{A}) \in K_*(\mathcal{L}^{\infty}_{MA,s}(Y_{MA}, E_{MA})). 
\end{equation}
To prove this formula we argue as follows.\\
Recall that $\mathcal{L}^\infty_{MA,s}(Y_{MA}, E_{MA})$  (respectively $\mathcal{L}^\infty_{M,s}(Y_M, E_M)$ and $\mathcal{L}_s(A, \operatorname{End}(\mathcal{S}_A))$) is a dense subalgebra of the Roe algebra $C^* (Y_{MA}, E_{MA})^{MA}$ (respectively $C^*(Y_M, E_M)^{M}$ and $C^*_r(A, \operatorname{End}(\mathcal{S}_A))$) closed under holomorphic functional calculus. 
The image of the smooth index classes in the K-theory of the respective Roe $C^*$-algebras
define the (isomorphic)  index classes
$$\Ind (D_{Y_{MA}})\in K_* (C^* (Y_{MA}, E_{MA})^{MA})\,,\;\;\; \Ind (D_{Y_{M}})\in K_* (  C^*(Y_M, E_M)^{M}),\  \Ind (D_{A})\in K_* ( C^*_r(A, \operatorname{End}(\mathcal{S}_A))\,.$$
Furthermore, the Roe algebra  $C^* (Y_{MA}, E_{MA})^{MA}$ (respectively $C^*(Y_M, E_M)^{M}$ and $C^*_r(A, \operatorname{End}(\mathcal{S}_A))$)  is strongly Morita equivalent to $C^*_r (MA)$ (respectively $C^*_r(M)$ and $C^*_r(A)$) and as explained in \cite{PP2} we have explicit representatives of the
(strongly) Morita equivalent  index classes 
\begin{equation}\label{3-index-classes}
\Ind_{C^*_r (MA) } (D_{Y_{MA}})  \in K_* (C^*_r (MA))\,,\quad \Ind_{C^*_r (M) }(D_{Y_{M}})\in K_* (C^*_r (M))\,\,\text{ and }\,\, \Ind_{C^*_r (A)} (D_{A})\in  K_* (C^*_r (A)).
\end{equation} Indeed, these index classes
are defined in terms of our operators acting on suitable $C^*_r H$-modules, with $H$ one of the above 3 groups. Now, following \cite{Wu, LP-JFA} and 
\cite{Wahl-product} we
can also express these index classes in terms of unbounded KK-classes associated to a  Atiyah-Patodi-Singer boundary condition, denoted here 
\begin{equation}\label{3-KK-classes}
\begin{split}
\Ind^{{\rm APS}}_{C^*_r (MA) } (D_{Y_{MA}}) & \in KK_* (\CC,C^*_r (MA))\ ,\ \  \Ind^{{\rm APS}}_{C^*_r (M) }(D_{Y_{M}})\in KK_* (\CC,C^*_r (M)),\\
& \Ind^{{\rm APS}}_{C^*_r (A)} (D_{A})\in  KK_* (\CC,C^*_r (A)).
 \end{split}
\end{equation}
Notice that our $b$-index classes do {\it not} define KK-elements (the resolvent is not $C^*$-compact); this is why we need to pass to the APS-index classes
defined through the well-known boundary condition.
Using  formula (\ref{eq:Dirac-MA}) of the Dirac operator $D_{Y_{MA}}$ and
proceeding exactly as in \cite[Theorem 2.2]{Wahl-product}
 we can prove the following identity 
\[
 \Ind^{{\rm APS}}_{C^*_r (MA) }  (D_{Y_{MA}}) =[ \Ind^{{\rm APS}}_{C^*_r (M) } (D_{Y_{M}})\otimes \Ind^{{\rm APS}}_{C^*_r (A)} (D_{A}) ]\in KK_*(\CC,C^*_r (MA))\equiv KK_* (C^*_r (MA))
\]
Following the stated  isomorphisms of  $K$-theory groups and the compatibility of the various index classes we obtain finally 
\begin{equation*}
\Ind_\infty (D_{Y_{MA}})= \Ind_\infty (D_{Y_{M}})\otimes \Ind _\infty(D_{A}) \in K_*(\mathcal{L}^{\infty}_{MA,s}(Y_{MA}, E_{MA})). 
\end{equation*}
This is  precisely the claim we wanted to prove.

\noindent{\bf Step II}: Using the product structure (\ref{eq:product-algebra}) of the algebra $\mathcal{L}^\infty_{MA,s}(Y_{MA}, E_{MA})$, we compute the cocycle $\Phi_{Y_{MA},g}$ as follows.  For any $f_i \otimes g_i \in \mathcal{L}^{\infty}_{M,s}(Y_{MA}) \otimes\mathcal{L}_s(A, \operatorname{End}(\mathcal{S}_A))$, 
\begin{equation}
\label{MA com}
\begin{aligned}
&\Phi_{Y_{MA},g}(f_0 \otimes g_0, \dots f_n \otimes g_n) \\
=&\int_{M/Z_{M,g}} \int_{(MA)^n} \text{det}(a_1, \dots, a_n) \cdot \operatorname{Tr}\Big( f_0(hgh^{-1}(m_1 \dots m_n)^{-1})g_0((a_1\cdot \dots \cdot a_n)^{-1})\\
&f_1(m_1) g_1(a_1) \cdot \dots \cdot  f_n(m_n)g_n(a_n) \Big)
\; dh da_1 \dots da_n dm_1 \dots dm_n\\
=&\left(\int_{M/Z_{M,g}} \int_{(M)^n} \operatorname{Tr}\Big(f_0(hgh^{-1}(m_1 \dots m_n)^{-1})f_1(m_1) \cdot \dots \cdot  f_n(m_n) \Big)
\; dh dm_1 \dots dm_n \right)\\
\times& \left(\int_{(A)^n} \text{det}(a_1, \dots, a_n) \cdot 
\operatorname{Tr}\Big(g_0((a_1\cdot \dots \cdot a_n)^{-1}) g_1(a_1) \cdot \dots \cdot g_n(a_n)\Big) \; da_1 \dots da_n  \right). 
\end{aligned}
\end{equation}
In the above equation, we denote 
\[
\widetilde{\Phi}_{Y_M,g}(f_0, \dots, f_n) = \int_{M/Z_{M,g}} \int_{(M)^n} \operatorname{Tr}\Big( f_0(hgh^{-1}(m_1 \dots m_n)^{-1})f_1(m_1) \cdot \dots \cdot  f_n(m_n) \Big)
\; dh dm_1 \dots dm_n 
\]
and
\[
\Phi_{A, e} (g_0, \dots, g_n) = \int_{(A)^n} \text{det}(a_1, \dots, a_n) \cdot 
\operatorname{Tr}\Big( g_0((a_1\cdot \dots \cdot a_n)^{-1}) g_1(a_1) \cdot \dots \cdot g_n(a_n)\Big) \; da_1 \dots da_n 
\]
By (\ref{index dec}) and (\ref{MA com}), we conclude that 
\[
 \langle \Phi_{Y_{MA},g}, \Ind_\infty (D_{Y_{MA}})  \rangle=
 \langle \widetilde{\Phi}_{Y_M,g}, \Ind_\infty (D_{Y_M}) \rangle  \cdot \langle \Phi_{A,e}, \Ind_\infty (D_{A}) \rangle.
\]
By a special case \cite[Theorem 4.6]{PPT} (the case for G = A), we know that 
\[
\langle \Phi_{A,e}, \Ind_\infty (D_{A}) \rangle =1. 
\]
On the other hand, we can directly  check that if $f \star f = f$, then 
\[
\begin{aligned}
\widetilde{\Phi}_{Y_M,g}(f, \dots, f) =&\int_{M/Z_{M, g}} \int_{M^{\times n}} \operatorname{Tr}\Big( f(hgh^{-1}(m_1 \dots m_n)^{-1})f(m_1) \cdot \dots \cdot  f(m_n) \Big)
\; dh dm_1 \dots dm_n \\
=& \int_{M/Z_{M,g}} \operatorname{Tr}\big( f(hgh^{-1}) \big)dh = \Phi_{Y_M,g}(f).
\end{aligned}
\]
Thus, 
\[
 \langle \widetilde{\Phi}_{Y_M,g}, \Ind_\infty (D_{Y_M}) \rangle  =  \langle \Phi_{Y_M,g}, \Ind_\infty (D_{Y_M}) \rangle . 
 \]
This completes the proof. 
\end{proof}

\section{An index theorem for higher orbital integral through reduction}\label{sec:reduction}
We shall now put things together and give a formula for the higher delocalized
Atiyah-Patodi-Singer index $\langle \Phi^P_{Y,g}, \Ind_\infty (D)\rangle$.

\medskip
\noindent
Let $G$ be connected, linear real reductive, $P=MAN$ a cuspidal parabolic subgroup  and $g\in M$ a semisimple element.
Let $(Y_0,\mathbf{h}_0)$ be a cocompact $G$-proper manifold with boundary. We fix a slice $Z_0$ (and $Z$) for the $G$ action on $Y_0$ (and $Y$), so that 
$$Y_0\cong G\times_K Z_0, \quad Y \cong G \times_K Z. $$
with $Z_0$ a smooth compact manifold {\bf with} boundary. 
We assume that $D_{\partial Y}$ is $L^2$-invertible. Consider $Y/AN$, an $M$-proper manifold. Recall that we have the following $K\cap M$-invariant decomposition
\[
\begin{aligned}
\mathfrak{p} =&(\mathfrak{p} \cap \mathfrak{m})\oplus \mathfrak{a} \oplus \left(\left(\mathfrak{p} \cap \mathfrak{m}\right)\oplus \mathfrak{a} \right)^\perp\\
\cong& (\mathfrak{p} \cap \mathfrak{m})\oplus \mathfrak{a} \oplus \left(\mathfrak{k}/(\mathfrak{k}\cap \mathfrak{m})\right).
\end{aligned}
\]
Accordingly, the spinor bundle
\[
S_\mathfrak{p} \cong S_{\mathfrak{p} \cap \mathfrak{m}} \otimes S_\mathfrak{a} \otimes S_{\mathfrak{k}/(\mathfrak{k}\cap \mathfrak{m})},
\]
where $S_\mathfrak{a}$ is a vector space of dimension $2^{\lceil\frac{\dim \mathfrak{a}}{2}\rceil}$ on which $M\cap K$ acts trivially. The following facts will play a central role in our study:
\begin{enumerate}
\item
We realize $Y/N$ in terms of a slice:
\[
Y_{MA} \colon = Y/N\cong AM\times_{K\cap M} Z. 
\]
Restricting to $Y_{MA}$, 
\[
E_{MA}\colon  =E|_{Y_{MA}} \cong AM \times_{M \cap K}\left(S_{\mathfrak{p} \cap \mathfrak{m}} \otimes S_\mathfrak{a} \otimes S_{\mathfrak{k}/(\mathfrak{k}\cap \mathfrak{m}) }\otimes E_Z \right).
\]
For convenience, we introduce the following
\begin{equation}
\label{tildeE}
\widetilde{E_Z} := E_Z \otimes S_{\mathfrak{k}/(\mathfrak{k}\cap \mathfrak{m})}.
\end{equation}
Thus
\begin{equation}
\label{EM}
E_M = M \times_{M \cap K}\left(S_{\mathfrak{p} \cap \mathfrak{m}} \otimes S_{\mathfrak{k}/(\mathfrak{k}\cap \mathfrak{m}) }\otimes E_Z \right) \cong M \times_{M \cap K}\left(S_{\mathfrak{p} \cap \mathfrak{m}} \otimes \widetilde{E_Z} \right).
\end{equation}
By the assumption (\ref{dim ass}) on the dimension of $Y$, we can see that $Y_M$ is even dimensional because both $Z$ and $M/M\cap K$ are even dimensional. 
The quotient Dirac operator $D_{Y_M}$ on $Y_M$ acts on
\[
L^2(Y_M, E_M) \cong \left[L^2(M) \otimes  S_{\mathfrak{p} \cap \mathfrak{m}} \otimes L^2(Z, \widetilde{E_Z}) \right]^{K\cap M}.
\] 

We have assumed that $E$ is a $G$-equivariant  twisted spinor bundle on $Y$. Accordingly, $E_M$ is an $M$-equivariant twisted spinor bundle on $Y_M$, and we can write 
\begin{equation}\label{eq WM}
E_M= \mathcal{E}_M \otimes W_M. 
\end{equation}
where $\mathcal{E}_M$ is the spinor bundle associated to the induced Spin$^c$-structure on $Y_M$ and $W_M=M\times_{M\cap M} \widetilde{E}_Z$. It is important to point out that since 
\[
S_{\mathfrak{k}/(\mathfrak{k}\cap \mathfrak{m}) } = S^+_{\mathfrak{k}/(\mathfrak{k}\cap \mathfrak{m}) } \oplus S^-_{\mathfrak{k}/(\mathfrak{k}\cap \mathfrak{m}) }
\]
has a $\mathbb{Z}_2$-grading, the auxiliary $M$-equivariant vector $W_M$ is equipped with a $\mathbb{Z}_2$-grading as well, denoted by 
\begin{equation}\label{WM grading}
W_M = W_M^+ \oplus W_M^-
\end{equation}
\end{enumerate}

\begin{remark}\label{vanish rmk}
Unless $P$ is maximal cuspidal parabolic, $K \cap M$ has a lower rank than $K$. In this case,  there exists a $K \cap M$-equivariant isomorphism $\lambda$ defined by the action of an element in a Cartan subgroup of $K$ but not in $K \cap M$,
\[
\lambda: S^+_{\mathfrak{k}/(\mathfrak{k}\cap \mathfrak{m})}\cong S^-_{\mathfrak{k}/(\mathfrak{k}\cap \mathfrak{m})}.
\]
Such an isomorphism $\lambda$ is compatible with the respective connections and metrics. In the constructions in (\ref{tildeE}), (\ref{EM}) and (\ref{eq WM}),  the tensor products are all graded and grading compatible. Hence, $\lambda$ gives an $M$-equivariant isomorphism between the two vector bundles 
\[
W_M^+ \cong W_M^-,
\] 
compatible with the respective connections and metrics.
\end{remark}

\begin{remark}
\label{rmk:supertrace} As $W_M$ in Equation (\ref{WM grading}) is $\mathbb{Z}_2$-graded, the vector bundle $E_M$ in Equation (\ref{eq WM})  can be written as follows with respect to the gradings on $\mathcal{E}_M$ and $W_M$, 
\[
E_M=\big(\mathcal{E}_M^+\oplus \mathcal{E}_M^-\big)\otimes W_M^+\oplus \big(\mathcal{E}_M^+ \oplus \mathcal{E}_M^-\big)\otimes W^-_{M}.
\]
Accordingly, the $\mathbb{Z}_2$-grading  on $E_M$ gives the following decomposition,
\begin{equation}\label{EM grading}
E_M=E_M^+\oplus E_M^-,\qquad 
E_M^+=\mathcal{E}_M^+\otimes W^+_M\oplus \mathcal{E}_M^-\otimes W^-_{M},\qquad E_M^-=\mathcal{E}_M^-\otimes W^+_{M}\oplus \mathcal{E}_M^+\otimes W_M^-. 
\end{equation}
Given a linear operator $T$ on $E_M$, we write $T$ into a $4\times 4$ block matrix, 
\begin{equation}\label{eq:T decomp}
T=\left[\begin{array}{cccc}T^{++}_{++}&T^{++}_{-+}&T^{++}_{+-}&T^{++}_{--}\\ 
T^{-+}_{++}& T^{-+}_{-+}& T^{-+}_{+-}&T^{-+}_{--}\\
T^{+-}_{++}&T^{+-}_{-+}&T^{+-}_{+-}&T^{+-}_{--}\\
T^{--}_{++}&T^{--}_{-+}&T^{--}_{+-}&T^{--}_{--}\\
\end{array}\right],
\end{equation}
where $T^{\alpha\beta}_{\gamma\delta}$ is a linear operator from $\mathcal{E}_M^\gamma\otimes W_M^{\delta}$ to $\mathcal{E}_M^{\alpha}\otimes W^{\beta}_M$, for $\alpha, \beta, \gamma, \delta=\pm$. The $\mathbb{Z}_2$-grading, c.f. Equation (\ref{EM grading}), on $E_M$ introduces a supertrace on $T$ as follows,
\begin{equation}\label{eq:supertrace}
\begin{split}
\operatorname{Str}(T):=&\operatorname{tr}(T^{++}_{++})+\operatorname{tr}(T^{--}_{--})-\big[\operatorname{tr}(T^{-+}_{-+})+\operatorname{tr}(T^{+-}_{+-})\big]\\
=&\operatorname{tr}(T^{++}_{++})-\operatorname{tr}(T^{-+}_{-+})-\big[\operatorname{tr}(T^{+-}_{+-})-\operatorname{tr}(T^{--}_{--})\big]. 
\end{split}
\end{equation}
\end{remark}

The following theorem is one of the main results of this paper:

\begin{theorem}
\label{main thm}
Suppose that the metric on $Y$ is slice compatible.
Assume that  $D_{\partial Y}$ is $L^2$-invertible and consider the higher 
index $\langle \Phi^P_{Y,g}, \Ind_\infty (D)\rangle$.
The following formula holds:
\begin{equation}\label{eq:higherdelocalizedaps}
\langle \Phi^P_{Y,g}, \Ind_\infty (D)\rangle=
\int_{(Y_0/AN)^g} c^g_{(Y_0/AN)^g} {\rm AS}(Y_0/AN)_g-\frac{1}{2} \eta_g (D_{\partial Y_{M}}) 
\end{equation}
with 
$$\eta_g (D_{\partial Y_{M}}) =\frac{1}{\sqrt{\pi}} \int_0^\infty \tau^{\partial Y_{M}}_g (D_{\partial Y_{M}} \exp (-tD^2_{\partial Y_{M}}) \frac{dt}{\sqrt{t}}.$$
Recall that $c^g_{(Y_0/AN)^g}$ is a compactly supported smooth cutoff function on $(Y_0/AN)^g$ associated to the $Z_{M, g}$ action on $(Y_0/AN)^g$. 
\end{theorem}
To explain the right-hand side of (\ref{formula-main}) more explicitly, we fix the following data.
\begin{itemize} 
\item In $(\ref{eq WM})$, we constructed a $M$-equivariant, $\mathbb{Z}_2$-graded vector bundle 
\[
W_M = W_M^+ \oplus W_M^-. 
\]  
We define
\[
R^{W^\pm_M} = \text{the curvature form of the Hermitian connection on } W_M^\pm. 
\]

\item  $R^\mathcal{N}$, the curvature form associated to the Hermitian connection on $\mathcal{N}_{(Y_0/AN)^g}\otimes \mathbb{C}$ ($\mathcal{N}_{(Y_0/AN)^g}$ is the normal bundle of the $g$-fixed point submanifold $(Y_0/AN)^g$ in $Y_0/AN$); 
\item  $R^{L}$, the curvature form associated to the Hermitian connection on $L_{\operatorname{det}}|_{(Y_0/AN)^g}$ ($L_{\operatorname{det}}$ is  the determinant line bundle of the Spin$^c$-structure on $Y_0/AN$ and $L_{\operatorname{det}}|_{(Y_0/AN)^g}$ is its restriction to $(Y_0/AN)^g$);
\item $R_{(Y_0/AN)^g}$, the Riemannian curvature form associated to the Levi-Civita connection on the tangent bundle of  $(Y_0/AN)^g$;
\item The form ${\rm AS}(Y_0/AN)_g$ is then given by  the following expression:

\begin{equation}\label{eq:geom-form}
\frac{\widehat{A}\left(\frac{R_{(Y_0/AN)^g}}{2\pi i}\right) \left[\operatorname{tr}\left(g \exp\left(\frac{R^{W^+_M}}{2\pi i}\right)\right) - \operatorname{tr}\left(g \exp\left(\frac{R^{W^-_M}}{2\pi i}\right)\right)\right]\exp(\operatorname{tr}(\frac{R^L}{2\pi i}))}
{\operatorname{det}\left(1-g \exp(-\frac{R^{\mathcal{N}}}{2\pi i})\right)^{\frac{1}{2}}}
\end{equation}

\item Since the auxiliary vector bundle $W_M$ on $Y_M$ is $\mathbb{Z}_2$-graded, the Dirac operator $D_{Y_M}$ is equipped with a $\mathbb{Z}_2$-grading as well. Following the decomposition (\ref{EM grading}) and (\ref{eq:T decomp}), we have the decomposition for $D_{Y_M}$,
\[
\begin{split}
&D_{Y_M}=\left[\begin{array}{cccc}0&D^{+}_{Y_M, +}&0&0\\
				D^{-}_{Y_M, +}&0&0&0\\
				0&0&0&D^{+}_{Y_{M},-}\\
				0&0&D^{-}_{Y_M,-}&0\\
	\end{array}\right]=\left[\begin{array}{cc}
				D_{Y_M,+}&0\\ 
				0& D_{Y_M,-}\end{array}
				\right] \\
&D_{Y_M, +}=\left[\begin{array}{cc}0&D^{+}_{Y_M, +}\\
				D^{-}_{Y_M, +}&0
				\end{array}\right], \qquad D_{Y_M, -}=\left[\begin{array}{cc}0&D^{+}_{Y_M, -}\\
				D^{-}_{Y_M, -}&0
				\end{array}\right],
\end{split}
\]
where
\[
D^\alpha_{Y_M, \beta} \colon L^2\left(Y_M, \mathcal{E}^\alpha_{Y_M} \otimes W^\beta_M\right) \to L^2\left(Y_M, \mathcal{E}^{-\alpha}_{Y_M} \otimes W^\beta_M\right),\quad \alpha, \beta = \pm. 
\]
Accordingly, since $\partial Y_M$ is odd dimensional, the Dirac operator $D_{\partial Y_M}$ is a self-adjoint {\bf even} differential operator on the $\mathbb{Z}_2$-graded spinor bundle $E_{\partial Y_M}$, which is defined as follows, 
\[
E_{\partial Y_M}:= \mathcal{E}_{\partial Y_M}\otimes W_{M}|_{\partial Y_M},\qquad E_{\partial Y_M}^+:= \mathcal{E}_{\partial Y_M}\otimes W^+_{M}|_{\partial Y_M},\qquad E^-_{\partial Y_M}:= \mathcal{E}_{\partial Y_M}\otimes W^-_{M}|_{\partial Y_M}.
\]
We can write $D_{\partial Y_M}$ into a $2\times 2$ block diagonal matrix with respect to the grading on $E|_{\partial Y_M}$, 
\[
\left[\begin{array}{cc}
	D_{\partial Y_M, +}&0\\ 0&D_{\partial Y_M, -}\\
	\end{array}
\right],
\] 
where $D_{\partial Y_M, +}$ (and $D_{\partial Y_M, -}$) is the Dirac operator on $E_{\partial Y_M}^+$ (and $E_{\partial Y_M }^-$).
Following the definition of the supertrace (\ref{eq:supertrace}) and the $\mathbb{Z}_2$-grading of $D_{\partial Y_M}$, 
 we obtain the following expression for the delocalized eta invariant of $D_{\partial Y_M}$:
\begin{equation}\label{eq:etasuper}
\eta_g(D_{\partial Y_M})=\eta_g(D_{\partial Y_M, +})-\eta_{g}(D_{\partial Y_M, -}). 
\end{equation}
This point is  slightly different from the 0-degree pairing case, where the operator $D_{\partial Y_M}$ is not graded. 
\end{itemize}

\begin{proof}
Denote $\Phi^P_{Y,g}$ briefly by $\Phi_{Y,g}$. By \eqref{reduction1} we know that 
\begin{equation*} \langle \Phi_{Y,g}, \Ind_\infty (D) \rangle = \langle \Phi_{Y_{MA},g}, \Ind_\infty (D_{Y_{MA}})  \rangle\,.
 \end{equation*}
On the other hand, by equation \eqref{reduction2} we have that 
\begin{equation*}
 \langle \Phi_{Y_{MA},g}, \Ind_\infty (D_{Y_{MA}})  \rangle= \langle \Phi_{Y_M,g}, \Ind_\infty (D_{Y_M}) \rangle 
 \end{equation*}
 so that 
 \begin{equation*} \langle \Phi_{Y,g}, \Ind_\infty (D) \rangle =  \langle \Phi_{Y_M,g}, \Ind_\infty (D_{Y_M}) \rangle 
 \end{equation*}
 It suffices to apply now the 0-degree delocalized APS index theorem, Theorem \ref{theo:0-delocalized-aps}, to the right hand side of the above equation
 and recall that  $Y/AN$ is diffeomorphic to  $Y_M:=M\times_{K\cap M} Z$ and that is obtained by addition of a cylindrical end
 to $Y_0/AN$. 
\end{proof}

When $P$ is not a maximal parabolic subgroup, we have the following vanishing result:
\begin{remark}\label{vanish rmk-2}
If $P$ is not a maximal parabolic subgroup, then $\Phi^P_{Y,g}$ is a trivial class in cyclic cohomology and 
\[
\langle \Phi^P_{Y,g}, \Ind_\infty (D)\rangle = 0.
\]
For the right-hand side of Equation (\ref{eq:higherdelocalizedaps}), we know from Remark \ref{vanish rmk} that 
\[
W_M^+ \cong W_M^-
\]
compatible with the respective connections and metrics. Thus,  $D_{\partial Y_M,+}$ is unitary equivalent to  $D_{\partial Y_M,-}$ under the above isomorphism.
By using the expression of $\eta_g(D_{\partial Y_M})$ as $\eta_g(D_{\partial Y_M, +})-\eta_{g}(D_{\partial Y_M, -})$,
 we obtain the following 
\[
\eta_g (D_{\partial Y_{M}})  = 0\,.
\]
Moreover, the integral
\[
\int_{(Y_0/AN)^g} c^g_{(Y_0/AN)^g} {\rm AS}(Y_0/AN)_g = 0
\]
as follows from the fact that
\[
\operatorname{tr}\left(g \exp\left(\frac{R^{W^+_M}}{2\pi i}\right)\right) - \operatorname{tr}\left(g \exp\left(\frac{R^{W^-_M}}{2\pi i}\right)\right) = 0,
\]
because of the the isomorphism $W^+_M\cong W^-_M$.
\end{remark}
\section{Numeric rho invariants on $G$-proper manifolds}\label{sect:numeric}

In this section we shall introduce (higher) rho numbers associated to positive scalar curvature (psc) metrics.

\subsection{Rho numbers associated to delocalized 0-cocycles}
We consider a closed $G$-proper manifold $X$ {\bf without} boundary, $G$ connected, linear real reductive, $g\in G$ a semisimple element,
$D_X$ a $G$-equivariant $L^2$-invertible Dirac operator defined in  (\ref{def Dirac}). We 
know that the following integral is convergent:
$$\eta_g (D_X):=\frac{1}{\sqrt{\pi}} \int_0^\infty \tau^{X}_{g} (D_{X} \exp (-tD^2_{X}) \frac{dt}{\sqrt{t}}.$$

\begin{definition}
Let $X$ be $G$-equivariantly spin  and $D_X\equiv D_{\mathbf{h}}$, the spin Dirac operator
 associated to a $G$-equivariant PSC metric $\mathbf{h}$.
  We know that $D_\mathbf{h}$ is $L^2$-invertible. We define 
$$\rho_g (\mathbf{h}):= \eta_g (D_{\mathbf{h}})\,.$$
\end{definition}

%
%

\subsection{Rho numbers associated to higher delocalized cocycles}
We can generalize the example of the previous subsection and define rho numbers associated
 the higher cocycles $\Phi^P_{g}$.

\begin{definition}
Let $P=MNA$ be a cuspidal parabolic subgroup, $g\in M$ as above and consider $\Phi_g^P$ and $\Phi_{X,g}^P$ (we recall that $X$ is without boundary). 
Assume that $\mathbf{h}$ is a $G$-invariant PSC metric on $X$. In this case we also assume that
$\mathbf{h}$ is slice-compatible.
Then  
\begin{equation}\label{rho-P}
\rho_g^P (\mathbf{h}):= \eta_g (D_{X_M})
\end{equation}
(with $X_M$ the reduced manifold associated to $X$)
 is well defined.
 \end{definition}
 
 \noindent
 Notice that as $D_X$ is invertible, we have that  $D_{X_M}$ is also invertible, see 
 Proposition \ref{prop:existence-index-reduced}.
 
\subsection{Bordism properties}
The APS index theorems proved in this article can be used in order to study the bordism properties 
of these rho invariants.

We assume, unless otherwise stated that we are on a $G$-proper manifold which is endowed 
with a $G$-invariant metric and a $G$-equivariant spin structure. If needed, we shall assume slice-compatibility
of the metric and of the spin structure. 
  \\

\begin{definition}
Consider a cocompact proper $G$-manifold $(W,\mathbf{h})$, possibly with boundary, endowed 
with an equivariant spin structure. Let $S$ be the spinor bundle.
Let $g\in G$ be semisimple. We shall say that $g$ is {\it geometrically-simple
on $W$}
if 
$$\int_{W^g} c^g {\rm AS}_g (W,S)=0\,.$$
with ${\rm AS}_g (W,S)$ the usual integrand appearing in the 0-degree delocalized APS-index theorem.
\end{definition}

\begin{example}
The following proposition provides many examples of geometrically-simple elements $g$ on an arbitrary
$G$-proper manifold $W$.

\begin{proposition}\label{prop:fixed point}If $g$ is non-elliptic, that is, does not conjugate to a compact element, then every element of the conjugacy class $C(g):=\{hgh^{-1}|h\in G\}$ in $G$ does not have any fixed point on $W$.
\end{proposition}
\begin{proof}
Without loss of generality, we assume that $W^g \neq \emptyset$. Take 
\[
x = (g_1, s) \in W = G \times_K S
\] 
such that 
\[
gx = (gg_1, s) = (g_1, s) = x. 
\] 
As $(g_1,s)$ and $(gg_1,s)$ correspond to the same point in $G\times_K S$, we can find some $k \in K$ such that 
\[
gg_1 = g_1 k, \quad k s = s. 
\]	
Thus, $g_1^{-1}g g_1 \in K$, that is, $g$ has to be elliptic, which contradicts to our assumption on $g$. 
\end{proof}
 \end{example}

\begin{definition}\label{def:concordant}
Let $(Y,\mathbf{h}_0)$ and $(Y,\mathbf{h}_1)$ be two psc metrics. We say that they are $G$-concordant if there exists a $G$-invariant metric
$\mathbf{h}$ on $Y\times [0,1]$ which is of psc, product-type near the boundary and restricts to $\mathbf{h}_0$ at $Y\times \{0\}$ and to
$\mathbf{h}_1$ at $Y\times \{1\}$. If $\mathbf{h}_0$ and $\mathbf{h}_1$ are slice compatible, we also require the metric $\mathbf{h}$ on $Y\times [0,1]$ to be slice compatible. 
\end{definition}

\begin{definition}\label{def:bordant}
Let $(Y_0,\mathbf{h}_0)$ and $(Y_1,\mathbf{h}_1)$ be two $G$-proper manifolds with psc metrics. 
We shall say that they are $G$-psc-bordant  if there exists a $G$-manifold with boundary $W$ with a $G$-invariant metric $\mathbf{h}$
such that:\\
(i) $\partial W=Y_0\sqcup Y_1$;\\
(ii) $\mathbf{h}$ is product-type near the boundary and of psc;\\
(iii) $\mathbf{h}$ restricts to $\mathbf{h}_0\sqcup \mathbf{h}_1$ on  $\partial W=Y_0\sqcup Y_1$.

\noindent{If} $\mathbf{h}_0$ and $\mathbf{h}_1$ are slice compatible, we also require the metric $\mathbf{h}$ to be slice compatible. 
\end{definition}

 \begin{proposition}\label{prop:bordism}$\;$\\
 1] Assume that the 
 psc metrics  $\mathbf{h}_0$ and $\mathbf{h}_1$ on $Y$  are $G$-concordant.
 Assume that $g$ is geometrically-simple on $Y\times [0,1]$, Then $\rho_g (\mathbf{h}_0)=\rho_g (\mathbf{h}_1)$.\\
 2] Let $P=MAN<G$ be a cuspidal parabolic subgroup and let $g\in M$ be a semisimple element.
 Assume that the slice-compatible psc metrics  $\mathbf{h}_0$ and $\mathbf{h}_1$ on $Y$  are $G$-concordant.
 Assume that $g$ is geometrically-simple on $Y_M\times [0,1]$. Then $\rho^P_g (\mathbf{h}_0)=\rho^P_g (\mathbf{h}_1)$.\\
 3] Let  $(Y_0,\mathbf{h}_0)$ and $(Y_1,\mathbf{h}_1)$ be $G$-psc-bordant through $(W,h)$.
 Assume that $g$ is geometrically-simple on $W$.  
 Then $\rho_g (\mathbf{h}_0)=\rho_g (\mathbf{h}_1)$.
  \end{proposition}

\begin{proof}
In both cases the proof is an immediate consequence of the geometric-simplicity of $g\in G$ and the relevant delocalized APS index theorems.
\end{proof}

\begin{remark}
Examples of geometrically-simple $g$  as in Proposition \ref{prop:bordism}
are given by elements in the conjugacy class
  of a non-elliptic element, c.f. Proposition \ref{prop:fixed point}.
  Notice that this applies to {\it any} bordism 
  entering into Definition \ref{def:concordant} and Definition \ref{def:bordant}.
  Thus  for such $g$ 
  our (higher) rho invariants are in fact {\it concordance and bordism invariants}.
\end{remark}

\begin{remark}
It is a challenge to understand whether these results can be employed in studying the spaces
$ \mathcal{R}^+_{G,{\rm slice}} (Y)$ of slice-compatible $G$-invariant  metrics of psc
and the larger space $ \mathcal{R}^+_{G} (Y)$ of  $G$-invariant  metrics of psc (if non-empty).
It would be also interesting to understand the relationship between the following 3 spaces, especially from the point of view of homotopy theory:\\
(i) $ \mathcal{R}^+_{G,{\rm slice}} (Y)$;\\(ii) $\mathcal{R}^+_G (Y)$;\\  (iii) $\mathcal{R}^+_K (S)$, the space of 
$K$-invariant metrics of psc on the slice $S$.\\ 
Clearly there is an inclusion $ \mathcal{R}^+_{G,{\rm slice}}(Y)\hookrightarrow \mathcal{R}^+_G (Y)$;
there is also a natural map   $\mathcal{R}^+_K (S) \to  \mathcal{R}^+_{G,{\rm slice}} (Y)$, see \cite{MGW}.
For all these questions it would be interesting to develop a $G$-equivariant Stolz' sequence and  investigate
 its basic properties. We leave this task to future research.
\end{remark}

\begin{remark}
If $Y_1$ and $Y_2$ are two oriented $G$-proper manifolds and $\mathbf{f}:Y_1\to Y_2$ is an oriented
 $G$-equivariant homotopy equivalence then we can also define rho invariants associated to $f$ and to a semisimple element $g$:   
$$\rho_g (\mathbf{f}):= \eta_g (D^{{\rm sign}}_X+ A({\mathbf{f}}))$$
with $D^{{\rm sign}}_X$ the signature operator on $X=Y_1\cup (-Y_2)$ and $A({\mathbf{f}})$ the Hilsum-Skandalis perturbation, as developed by Fukumoto in the $G$-proper
context \cite{Fukumoto}. Notice however that both the definition of this invariant and the proof of the associated APS index theorem
are not a routine extension of the case treated here. Especially, the large time behaviour of a perturbed
heat kernel and the proof of the analogue of Proposition \ref{prop:from-cocycle-to-eta} are particularly intricate.
We plan to report on our results in this direction in a separate article.
 \end{remark}
%
\appendix
\section{Dirac operator on $G$}

In this appendix, we discuss the difference between the Dirac operator, Eq. (\ref{def Dirac}), considered in this article and the split Dirac operator operator, Eq. (\ref{dirac-split}), on the manifold $G$ with the free proper left $G$-action. 

Suppose that $G$ is a connected reductive Lie group with maximal compact subgroup $K$.  We denote by $\mathfrak{g} = \mathfrak{k} \oplus \mathfrak{p}$ the Cartan decomposition of the Lie algebra and $\mathfrak{g}_\CC = \mathfrak{g} \otimes \CC$ its complexification. Let $B$ be a non-degenerate bilinear symmetric form on $\mathfrak{g}$, which is invariant under the adjoint action of $G$ on $\mathfrak{g}$. We assume $B$ to be negative on $\mathfrak{p}$, and positive on $\mathfrak{k}$. Let $S_\fg$ and $S_\fk$ be the irreducible $\mathbb{Z}_2$-graded spin representation for $\text{Cliff}(\fg_\CC)$ and $\text{Cliff}(\fk_\CC)$ respectively. Then we have the following decomposition
\[
S_\fg \cong S_\fk \otimes S_\fp. 
\]
The adjoint actions on $\fg$ induce maps
\begin{align*}
\text{ad}^\fg \colon& \fg_\CC \to \End(S_\fg) \cong \text{Cliff}(\fg_\CC)	\\
\text{ad}^\fk \colon& \fk_\CC \to \End(S_\fk) \cong \text{Cliff}(\fk_\CC)
\end{align*}
defined by the following formulas: for any $X \in \fg_\CC$, 
\[
\text{ad}^\fg(X) \colon = \frac{1}{4}\sum_{i}^{\dim \fg} c\left([X, e_i]\right)\cdot c(e_i), \quad \text{ad}^\fk(X) \colon =\frac{1}{4}\sum_{i}^{\dim \fk} c\left([X, e_i]\right)\cdot c(e_i) 	
\]
where 
\[
e_1, \cdots e_{\text{dim} \mathfrak{k}}, \quad \text{and}  \quad  e_{\text{dim} \mathfrak{k}+1}, \cdots, e_{\text{dim} \mathfrak{g}}
\]
are the orthonormal bases for $\fk$ and $\fp$ respectively. We have the following relation
\[
\text{ad}^\fg(X) = \text{ad}^\fk(X) + \text{ad}^\fp(X), \quad X \in \fk_\CC,
\]
where 
\[
\text{ad}^\fp(X)  = -\frac{1}{4}\sum_{i, j=\dim \fk+1}^{\dim \fg} B(X, [e_i, e_j]) \cdot c(e_i)\cdot c(e_j) \in \text{Cliff}(\fp), \quad X \in \fk_\CC
\]

By left-trivialization, we can identify the vector field $\mathfrak{X}(G)$ with $C^\infty(G) \otimes \mathfrak{g}$. In particular, for any $X \in \mathfrak{g}$, we can identify it with a left-invariant vector field  on $G$. We define a $K$-invariant inner product on $\mathfrak{g}$ by 
\[
\langle \ ,  \ \rangle = -B|_{\mathfrak{k}} \oplus + B|_{\mathfrak{p}}.
\]
Our computation below is similar to those in \cite[Chapter 9]{MR3052646}, while the author considers the pseudo Riemannian metric B, introduced above, instead of a Riemannian metric. 

By the Koszul formula, the Levi-Civita connection is determined by  
\[
\langle \nabla_X Y, Z \rangle = \frac{1}{2}\left( \langle [X, Y], Z \rangle - \langle [Y, Z], X\rangle  +\langle [Z, X], Y \rangle \right), \quad X, Y, Z \in \mathfrak{g}. 
\]
More general, for any vector field $V$ on $G$, it can be written as 
\[
V = \sum_{i=1}^{\dim \fg} f_i \cdot V_i
\]
where $f_i \in C^\infty(G)$ and $\{V_i\}$ is a basis for $\fg$. Then 
\[
\nabla_X V = \sum_{i=1}^{\dim \fg} X(f_i) \cdot V_i + \sum_{i=1}^{\dim \fg} f_i \cdot \nabla_X V_i. 
\]
If $X, Y \in \mathfrak{p}$ and $Z \in \mathfrak{k}$, then 
\[
\langle [Y, Z], X\rangle  = -B([Y, Z], X) = B(Z, [Y, X]) = \langle Z, [Y, X]\rangle = -\langle [X, Y], Z\rangle
\]
and 
\[
\langle [Z, X], Y\rangle  = -B([Z, X], Y)  = -B(Z, [X, Y]) = -\langle [X, Y], Z\rangle. 
\]
Thus, 
\[
\nabla_X Y = \frac{1}{2}[X, Y]. 
\]
Similarly, one can check the following:
\begin{itemize}
	\item $\nabla_X Y = \frac{1}{2}[X, Y], \quad X, Y \in \mathfrak{k}$. 
	\item $\nabla_X Y = \frac{3}{2}[X, Y], \quad X \in \mathfrak{k}, Y \in \mathfrak{p}$.
	\item $\nabla_X Y = -\frac{1}{2}[X, Y], \quad X \in \mathfrak{p}, Y \in \mathfrak{k}$.
\end{itemize}
Therefore, the induced Clifford connection $\nabla^{S_\mathfrak{g}}$ acts on left $G$-invariant section of $S_\mathfrak{g}$ by the following formula: 
\[
\nabla^{S_\mathfrak{g}}_X = \frac{1}{8}\sum_{i=1}^{\dim \fk} c\left([X, e_i]\right)\cdot c(e_i)+\frac{3}{8}\sum_{i= \dim \fk+1 }^{\dim \fg} c\left([X, e_i]\right)\cdot c(e_i) , \quad X \in \mathfrak{k}
\]
and 
\[
\nabla^{S_\mathfrak{g}}_X = -\frac{1}{8}\sum_{i=1}^{\dim \fk} c\left([X, e_i]\right)\cdot c(e_i)+\frac{1}{8}\sum_{i= \dim \fk+1 }^{\dim \fg} c\left([X, e_i]\right)\cdot c(e_i) , \quad X \in \mathfrak{p}.
\]
Note that  
\[
e_1, \cdots e_{\text{dim} \mathfrak{k}}, \quad \text{and}  \quad  \sqrt{-1}e_{\text{dim} \mathfrak{k}+1}, \cdots, \sqrt{-1} e_{\text{dim} \mathfrak{g}}
\]
give an orthonormal basis for $\fg_\CC$ with respect to the inner product $\langle \ , \ , \rangle$. The corresponding Dirac operator on $L^2(G, S_\mathfrak{g})$ is given by 
\begin{align*}
D_G= \sum_{i=1}^{\dim \fk}	c(e_i) \nabla^{S_\mathfrak{g}}_{e_i} - \sum_{i=\dim \fk+1}^{\dim \fg}	c(e_i) \nabla^{S_\mathfrak{g}}_{e_i}.
\end{align*}
As an element in $\mathcal{U}(\fg_\CC)\otimes \text{Cliff}(\fg_\CC)$, 
\begin{equation}
\label{Dirac compute}
\begin{aligned}
D_G= &\sum_{i=1}^{\dim \fk} e_i \otimes c(e_i) - \sum_{i=\dim \fk+1}^{\dim \fg} e_i \otimes c(e_i)\\
+& \frac{1}{8}\sum_{j=1}^{\dim \fk}\sum_{i=1}^{\dim \fk} c\left([e_j, e_i]\right)\cdot c(e_i)c(e_j) + \frac{3}{8}\sum_{j=1}^{\dim \fk} \sum_{i= \dim \fk+1 }^{\dim \fg} c\left([e_j, e_i]\right)\cdot c(e_i)c(e_j)\\
-&\frac{1}{8}\sum_{j= \dim \fk+1 }^{\dim \fg} \sum_{i=1}^{\dim \fk} c\left([e_j, e_i]\right)\cdot c(e_i)c(e_j)+\frac{1}{8}\sum_{j= \dim \fk+1 }^{\dim \fg} \sum_{i= \dim \fk+1 }^{\dim \fg} c\left([e_j, e_i]\right)\cdot c(e_i)c(e_j)
\end{aligned}
\end{equation}
Let us introduce a \emph{cubic term} 
\[
\phi^\fk = \frac{1}{4}\sum_{j=1}^{\dim \fk}\sum_{i=1}^{\dim \fk} c\left([e_j, e_i]\right)\cdot c(e_i)c(e_j) = \sum_{i=1}^{\dim \fk} \text{ad}^\fk(e_j) \cdot c(e_j)
\]
and a \emph{torsion term} by 
\[
\Omega =  \frac{1}{4}\sum_{j=1}^{\dim \fk} \sum_{i= \dim \fk+1 }^{\dim \fg} c\left([e_j, e_i]\right)\cdot c(e_i)c(e_j) =  \sum_{i=1}^{\dim \fk} \text{ad}^\fp(e_j) \cdot c(e_j)
\]
By straightforward computation, we check that 
\[
\frac{1}{4}\sum_{j= \dim \fk+1 }^{\dim \fg} \sum_{i=1}^{\dim \fk} c\left([e_j, e_i]\right)\cdot c(e_i)c(e_j) = \Omega
\]
and
\begin{align*}
\frac{1}{4}\sum_{j= \dim \fk+1 }^{\dim \fg} \sum_{i= \dim \fk+1 }^{\dim \fg} c\left([e_j, e_i]\right)\cdot c(e_i)c(e_j)= \Omega. 
\end{align*}
Thus, (\ref{Dirac compute}) can be simplified as follow
\[
D_G = \left(\sum_{i=1}^{\dim \fk} e_i \otimes c(e_i) +\frac{1}{2}\phi^\fk \right)- \sum_{i=\dim \fk+1}^{\dim \fg} e_i \otimes c(e_i) +\frac{3}{2}\Omega. 
\]
If we identify 
\[
L^2(G, S_\fg) \cong \left[L^2(G) \otimes S_\fp \otimes L^2(K, S_\fk)\right]^K,
\]
the split Dirac operator is given by 
\begin{align*}
D_{G, \text{split}} =& D_{G,K} \otimes 1+ 1 \otimes D_K \\
=&- \sum_{i=\dim \fk+1}^{\dim \fg} e_i \otimes c(e_i)
 + \left(\sum_{i=1}^{\dim \fk} e_i \otimes c(e_i) +\frac{1}{2}\phi^\fk \right)\end{align*}
Therefore, the difference between the Dirac operator associated to the Clifford connection and split Dirac operator is given by the torsion term $\frac{3}{2}\cdot \Omega$ (see also \cite[Theorem 10.38]{BGV}).

\end{document}